\documentclass[11pt]{amsart}
\usepackage{amsmath,amsthm,amssymb,amsfonts,graphicx,amscd}
\usepackage[applemac]{inputenc}
\usepackage[active]{srcltx} 
\usepackage{mathrsfs}

\usepackage{tikz-cd}

\usepackage{xcolor}
\usepackage[backref=page]{hyperref}
\definecolor{couleur_cite}{rgb}{0.05,.4,0.05}
\definecolor{couleur_link}{rgb}{0.05,0.05,0.4}
\hypersetup{
bookmarksopen,bookmarksnumbered,
colorlinks,
linkcolor=couleur_link,citecolor=couleur_cite,
}
\usepackage{enumerate}
\usepackage{upgreek}
\usepackage{ifpdf}
\usepackage{booktabs}
\usepackage{comment}
\usepackage[all]{xy}

\usepackage[left=3cm,right=3cm,tmargin=2cm,asymmetric]{geometry}

\usepackage{soul}

\numberwithin{equation}{section}
\makeindex

\newcommand{\g}{\mathfrak{g}}

\newcommand{\PO}{\mathrm{PO}}

\newcommand{\p}{\mathfrak{p}}

\newcommand{\C}{\mathbb{C}}

\newcommand{\GL}{\mathrm{GL}}
\newcommand{\PGL}{\mathrm{PGL}}
\newcommand{\SL}{\mathrm{SL}}

\newcommand{\R}{\mathbb{R}}
\newcommand{\Z}{\mathbb{Z}}
\newcommand{\Q}{\mathbb{Q}}

\newcommand {\pnorm}[1]   {\left\lVert #1 \right\rVert}
\DeclareMathOperator {\Mdiag}{diag}
\newcommand{\abs}[1]{\ensuremath{\left|#1\right|}}

\newcommand{\Ad}{\mathrm{Ad}}
\newcommand{\Run}{\underset{s=1}{\rm Res}\;}

\newcommand{\Mtroistrois}[9]{
\begin{pmatrix}
#1&#2&#3 \\
#4&#5&#6 \\
#7&#8&#9
\end{pmatrix}
}

\newcommand{\splus} {s^+_{\rm cusp}}
\newcommand{\sminus} {s^-_{\rm cusp}}

\DeclareFontFamily{OT1}{rsfs}{}
\DeclareFontShape{OT1}{rsfs}{n}{it}{<-> rsfs10}{}
\DeclareMathAlphabet{\mathscr}{OT1}{rsfs}{n}{it}

\newcommand{\fig}[1]{\medskip {\bf Figure: #1} \medskip}
\graphicspath{{drop/}}
\usepackage{caption}
\usepackage{subcaption}

\newtheorem{thm}{Theorem}[section]
\newtheorem{lem}[thm]{Lemma}
\newtheorem{prop}[thm]{Proposition}
\newtheorem{cor}[thm]{Corollary}

\newtheorem{defn}[thm]{Definition}

\theoremstyle{definition}

\newtheorem*{ex*}{Example}

\newtheorem{remark}[thm]{Remark}

\newtheorem{hyp}[thm]{Hypothesis}

\theoremstyle{remark}

\begin{document}

\title[]{Large values of cusp forms on $\GL_n$}

\author[]{Farrell Brumley}
\address{Institut Galil\'ee, Universite Paris 13, 93430 Villetaneuse France}
\email{brumley@math.univ-paris13.fr}

\author[]{Nicolas Templier}
\address{Department of Mathematics, Cornell University, Malott Hall, Ithaca NY 14853-4201.}
\email{templier@math.cornell.edu}

\begin{abstract}
We establish the transition behavior of Jacquet-Whittaker functions on split semi-simple Lie groups. As a consequence, we show that for certain finite volume Riemannian manifolds, the local bound for normalized Laplace eigenfunctions does not hold globally.  
\end{abstract}

\thanks{{\it 2010 Mathematics Subject Classification.} Primary: 11F70. Secondary: 58K55.}
\keywords{Sup norms, Maass forms, Whittaker functions, oscillatory integrals, Lagrangian mappings, Pearcey function}

\maketitle
\tableofcontents

Let $M$ be a complete $d$-dimensional Riemannian manifold without boundary. A central question in semiclassical analysis is to understand the concentration features of Laplacian $L^2$-eigenfunctions $\Delta f = \lambda f$, 
in relation with the geometry of $M$.  
A touchstone is the well-known 
bound of H\"ormander~\cite{Avakumovic:eigenfunktionen,Ho1} which implies that
\begin{equation}\label{convex}\tag{A}
|f(x)|\ll \lambda^{\frac{d-1}{4}}\pnorm{f}_2,
\end{equation}
where the multiplicative constant depends continuously on $x\in M$. This bound is local, being based on the principle that if an eigenfunction is large at a point, it remains so in a small neighborhood.

When $f$ is bounded globally on $M$, one may go further and 
compare the sup norm $\|f\|_\infty=\sup_{x\in M}|f(x)|$ with the $L^2$-norm, as a function of  $\lambda$. This boundedness is known to hold on any of the following three classes of manifolds: when both the sectional curvature and the injectivity radius of $M$ are bounded~\cite{Do};
under certain assumptions on the isoperimetric or isocapacity inequality of $M$~\cite{Cianchi-Mazya:noncompact}; and when $M$ is a finite volume locally symmetric
space and $f$ is cuspidal~\cite{HC}.

We shall focus in this article on this last case. Thus $M=\Gamma\backslash S$, where $S$ is a Riemannian globally symmetric space and $\Gamma$ is a lattice in the Lie group $G$ of isometries of $S$. Recall that a function on $M$ is said to be cuspidal if the constant term integrals around each cusp are zero, and that a cusp form is, by definition, a cuspidal $L^2$-eigenfunction of the Laplacian. When $\Gamma$ is congruence arithmetic then it is known~\cite{LV07} that cusp forms obey 
a Weyl law. 
This makes the automorphic setting well adapted to a global study of sup norms.

Nevertheless, when $\Gamma\backslash S$ is non-compact little is known on bounds (of any quality) on the sup norm $\|f\|_\infty$ of cusp forms. A sample qualitative question is whether an eigenfunction attains its largest value in a fixed bounded subset. 
For example, it is shown in \cite{IwSa,SaMo} that the local bound \eqref{convex} extends as a global bound to all cusp forms on the modular surface, the non-compact arithmetic hyperbolic surface associated with $\Gamma=\SL_2(\Z)$. But this statement masks the curious fact that such eigenfunctions can be large in the cusp, due to a transition from an oscillatory to a decay regime. As a corollary to our main result, Theorem \ref{th:gln}, we shall
show that for $\Gamma = \SL_n(\Z)$, $n\ge 6$, the local bound~\eqref{convex} does not 
extend globally.

In light of this, it is of interest to estimate the size of eigenfunctions over various regions in $M$, such as on bounded sets escaping to infinity. We approach these questions of transition behavior by modeling cusp forms by higher rank Whittaker functions. In the case of the modular surface, it is the
classical Bessel function which accounts for the large size of cusp forms: in a small region of height close to $\sqrt{\lambda}$, it is well-known that such functions admit a turning point and inherit the behavior of the Airy function. We establish higher rank transition behavior for more general Whittaker functions on split semisimple groups, showing, in particular, that the Lagrangian manifold associated with the Jacquet integral lies as an open subset of the Toda isospectral manifold, well-known in the physics literature. Our sharpest quantitative results are for $\GL_3(\R)$, where we show that the Pearcey function plays the role of the Airy function. 

\section{Statement of results}\label{s:statement}

Let $G$ be a split semisimple group over $\R$, $K$ a maximal compact subgroup of $G$. We 
denote by $S=G/K$ the associated Riemannian globally symmetric 
space. Let $\Gamma$ be a non-uniform lattice in $G$, which we shall assume to
be of congruence type. We consider functions $f\in C^\infty(\Gamma\backslash S)$ which 
are eigenfunctions of the ring of $G$-invariant differential operators
$\mathscr{D}_G(S)$ on $S$, eigenfunctions of the Hecke operators, and cuspidal. We shall 
refer to functions $f$ satisfying the above conditions as {\it
Hecke-Maass cusp forms}.  In particular, if $\Delta$ denotes the (non-negative) Laplacian 
on $M$ then $\Delta f=\lambda f$ for some $\lambda  >  0 $.

Let $\mathfrak{a}$ be a maximal abelian subalgebra of the Lie algebra $\g$ of $G$, chosen 
to 
lie in the $-1$ eigenspace of the Cartan involution associated with $K$. Let $H: 
G\rightarrow \mathfrak{a}$ denote the Iwasawa projection. Then $f$ shares the same 
$\mathscr{D}_G(S)$-eigenvalues as the function
$\exp((\rho+\nu)(H(x))$ on $S$, for some $\nu\in\mathfrak{a}_\C^*$. We call $\nu$ the 
spectral
parameter of $f$. We shall restrict our attention to Hecke--Maass cusp forms $f$ whose 
spectral parameters lie in a cone $\R_+\Omega\subset i\mathfrak{a}^*_{\rm
reg}$, where $\Omega\subset i\mathfrak{a}^*_{\rm reg}$ is an open bounded set. Here the 
non-singular set $\mathfrak{a}^*_{\rm reg}$ is the complement in
$\mathfrak{a}^*$ of the root hyperplanes. In other words, we ask that $f$ be both 
tempered and sufficiently regular.

We introduce a constant associated with $G$ arising from the integral representation of 
Whittaker functions. Let $B$ be a Borel subgroup of $G$ with unipotent radical $U$. Let 
${\rm ht}(G)$ be the sum of the heights of the positive roots. Then we define the 
non-negative half-integer
\[
c(G)=({\rm ht}(G)-\dim U)/2.
\]
Note that $c(G)=0$ if and only if $G$ is a product of rank one groups. For cusp forms 
admitting a Whittaker expansion, this constant $c(G)$ will be a useful exponential 
benchmark for their sup norms.

\subsection{Large values of cusp forms on $\GL_n$}\label{sub:firstmain}

Our first result concerns the arithmetic locally symmetric spaces associated with the 
group $\PGL_n(\R)$. Write
\[
S_n=\PGL_n(\R)/{\rm PO}(n)
\]
for the associated Riemannian globally symmetric space. This can be identified with the 
space of real positive definite symmetric matrices up to scalars, and the quotient 
$\GL_n(\Z)\backslash S_n$ is the space of rank $n$ Euclidean lattices, i.e., lattices up 
to rotation and dilation. In the present case, the existence of an infinite number of 
linearly independent Hecke--Maass cusp forms is a well-known result of 
M\"uller~\cite{Muller07}. Our first main theorem furnishes a lower bound on sup norms of 
cusp forms on $\Gamma \backslash S_n$, for $\Gamma$ a congruence subgroup of $\GL_n(\Z)$, 
with respect to the above constant $c(n)=c(\PGL_n(\R))$. 

\begin{thm}\label{th:gln} Let $\Omega\subset i\mathfrak{a}_{\rm reg}^*$ be an open 
bounded set. For any Hecke--Maass cusp form $f$ on $\Gamma\backslash S_n$ whose spectral 
parameter lies in $\sqrt{\lambda}\Omega$, we have
\begin{equation*}
\pnorm{f}_\infty \gg\lambda^{\frac{c(n)}{2}-\varepsilon} \pnorm{f}_2.
\end{equation*}
The implied multiplicative constant depends on $\varepsilon, \Omega$, and $\Gamma$.
\end{thm}

From the explicit values of the exponents
\[
\frac{c(n)}{2}=\frac{n(n-1)(n-2)}{24}\qquad\text{ and }\qquad \frac{d-1}{4}=\frac{n^2 + n 
- 4}{8},
\]
one sees that $\frac{c(n)}{2}>\frac{d-1}{4}$ for all $n\geq 6$, and thus 
we deduce the following. 
\begin{cor}
For $n \ge 6$ and $M=\Gamma\backslash S_n$, the bound~\eqref{convex} on Laplace 
eigenfunctions does not hold globally. 
\end{cor}

We also see that under the same assumptions, Hecke--Maass cusp forms on $\Gamma\backslash 
S_n$ achieve their maximum in the
cusp, not in the bulk. We speculate that this could be true for all $n\ge 2$. 
We show in the next corollary that it holds for $n\ge 5$. We need the recent result of~Blomer--Maga~\cite{Blomer-Maga:gln} and 
Marshall~\cite{Marshall:supnorm}, which relies on the
uniform bound for spherical functions of~\cite{BP:sup-norm-Siegel,Matz-Templier}, and 
says that there exists a constant 
$\delta(n) >  0$ depending only on $n$, such that
\begin{equation*}\label{d-r}
|f(x)|\ll \lambda^{\frac{n(n-1)}{8} -\delta(n)}\pnorm{f}_2,
\end{equation*}
where the implied constant depends continuously on $x\in M$. Since $\frac{c(5)}{2}=\frac{5}{2} = \frac{5\cdot (5-1)}{8}$, we have the following.
\begin{cor}\label{cor:bulk}
For $n\ge 5$ and any bounded subset $\mathcal{B} \subset \Gamma\backslash S_n$, all but 
finitely many Hecke--Maass cusp forms $f$ as in Theorem \ref{th:gln} satisfy
\[
\pnorm{f}_\infty
> \sup_{g\in \mathcal{B}} |f(g)|.
\]
\end{cor}

The exceedingly large values of cusp forms in Theorem \ref{th:gln} can be viewed as the 
semiclassical expression of a result of Kleinbock-Margulis~\cite{KM99}, according to 
which almost all geodesics penetrate the cusp at logarithmic speed $1/{\rm ht}(G)$. This 
reflects the small volume carried by the cusps, creating a bottleneck phenomenon as 
standing waves transition from an oscillatory to a decay regime. 

We emphasize that the lower bounds of Theorem \ref{th:gln} are of a very different nature 
than those of Rudnick-Sarnak \cite{RuSa}, Mili\'cevi\'c
\cite{Milicevic}, or Lapid-Offen \cite{Lapid-Offen}, all of which show power growth of 
sup norms of certain special Hecke--Maass forms. These latter results stem
from the functorial (in the sense of Langlands) origin of these eigenfunctions, and their 
proofs involve compact periods. The large values of such special eigenfunctions occur in 
bounded subset of $\Gamma \backslash S$, and the behavior in the cusp is thus not 
reflected in these bounds.

\subsection{Lower bounds on Whittaker functions}\label{intro:Whit}
Theorem \ref{th:gln} is deduced from corresponding lower bounds of Whittaker functions, 
through the Fourier-Whittaker period of $f$ along the maximal unipotent subgroup $U$,
following the method of~\cite{Temp:p-adic}. This passage makes use of some special 
features of the group $\PGL_n$, but the bounds on Whittaker functions themselves should 
be 
valid in wider generality. We thus return to the setting of a split semisimple real Lie 
group $G$ with associated symmetric space $S$.

A Whittaker function $W$ on $S$ is a $\mathscr{D}_G(S)$-eigenfunction of moderate growth which transforms under the $U$-action by a non-degenerate additive character $\psi$. 
One can think of $W$ as a section of a line bundle defined by $\psi$ over the quotient 
$U\backslash S$. Whittaker functions on $U\backslash S$ are expected to 
vanish at 0 and 
be of exponential decay at infinity, so are bounded (see 
Figure~1 for an illustration of how to partition the space between a 
neighborhood of 0 and infinity). A key point of the present work is that one may normalize $W$ in a natural way using its 
expression as an oscillatory integral. The existence of such an integral expression was 
proved by Jacquet
\cite{Jacquet67}. Combined with other results from representation theory, such as the 
multiplicity one theorem~\cite{Shalika74},
 this allows one to canonically define
 normalized Whittaker functions on $S$, which we call the Jacquet-Whittaker functions, see \S\ref{Whitt-structures} for details. 

\begin{thm}\label{general tau}  Let $\Omega\subset 
i\mathfrak{a}_{\rm reg}^*$ be an open bounded set. Assume that $G=\PGL(n)$, or that 
Hypothesis~\ref{hypothesis} is satisfied. Then the Jacquet-Whittaker functions 
$W$ on $S$ whose spectral parameter lie in $\sqrt{\lambda}\Omega$ satisfy
\begin{equation*}
\pnorm{W}_\infty \gg \lambda^{\frac{c(G)}{2}}.
\end{equation*}
The implied multiplicative constant depends on $\Omega$.
\end{thm}

\begin{remark}\label{intro-rmk}
We make two remarks on the general linear group in the formulation and proofs of Theorems \ref{th:gln} and \ref{general tau}.

\medskip

\noindent {\it i)} In the case of $G=\PGL_n(\R)$, there is a naturally defined inner 
product with respect to which any Whittaker function on 
$S$ is $L^2$-integrable and which, moreover, assigns the Jacquet-Whittaker function 
$L^2$-norm $1$. Following~\cite{Temp:p-adic}, one can then express Theorem~\ref{general 
tau} in a scale-invariant way 
as 
\[
\pnorm{W}_\infty \gg \lambda^{\frac{c(n)}{2}}\pnorm{W}_2.
\]
One reflection of a special feature of $\PGL_n$ is the existence of a formula (due to 
Stade \cite{Stade02}) relating the $L^2$-norm of the Whittaker function to local 
Rankin-Selberg $L$-functions. We exploit this fact to give a proof of 
Theorem \ref{general tau} for $\PGL_n(\R)$ in \S\ref{GLn-argument}.

\medskip

\noindent {\it ii)} The restriction to $G=\PGL_n(\R)$ in Theorem \ref{th:gln} is in part 
due to the genericity of Hecke--Maass cusp forms on $\Gamma\backslash
S_n$. This property is used to reduce lower bounds on $f$ to those on any given (non-degenerate) Fourier-Whittaker coefficient. It is well known that cusp forms
on other groups may fail to be generic. To extend the statement of Theorem \ref{th:gln} to such a setting (using Theorem \ref{general tau}), one might either
wish to use different special functions and investigate their size, or retain the Fourier-Whittaker coefficients and simply restrict one's attention to the
generic spectrum. In either approach, one must be able to control the relation between the $L^2$ normalization of the cusp form and that of the special function.
For Whittaker functions on $\GL_n$ this is provided by Rankin-Selberg theory and known bounds on $L$-functions (see \S\ref{sec:intro:gl2}). Outside the context of
$\GL_n$, recent conjectures of Lapid-Mao \cite{Lapid-Mao:conj} are relevant.

\end{remark}

The constant $c(n)$ in Theorem \ref{general tau} arises from the representation of Whittaker functions as oscillatory integrals over $U$, see
\S\ref{non-deg-phase}. The ${\rm ht}(G)/2$ term can be thought of as the asymptotics of a half-density, while $-\dim(U)/2$ is square-root cancellation over $U$. The next subsections provide a deeper study of these oscillatory integrals, by examining the regimes where square-root cancellation fails (in which case the lower bound can be improved slightly) due to degeneracies. 

In a different context, it is interesting to mention~\cite[Corollary~12.4]{Bernstein-Krotz} which establishes lower bounds for matrix coefficients when the $K$-types vary.

\subsection{Lagrangian mappings associated with Whittaker functions}\label{sec:Lagrangian}
We return to the general setting of sup norms on Riemannian manifolds.

It is a general principle in semiclassical analysis (see \cite{SZ:maximal,Toth-Zelditch:uniformly-bounded}) and the theory of Fourier Integral Operators (see \cite{Hormander:FIO-I, Duistermaat:oscillatory}) that eigenfunctions which exhibit extremal $L^p$ growth, if they exist, should concentrate in phase space $T^*(M)$ along certain Lagrangian submanifolds $\Lambda$ which are invariant under the action of the underlying Hamiltonian dynamics. For example, the zonal spherical harmonics on the sphere saturating the $L^\infty$ bound \eqref{convex} concentrate on the meridian torus $\Lambda$ consisting of geodesics joining the poles (the antipodal points of the fixed rotation axis).
The zonal spherical harmonics achieve their largest values at the poles, which are precisely the singularities of the projection $\Lambda\rightarrow M$. 
See also~\cite[Prop.~4.4]{Do} for the study of certain related manifolds.

Similarly, a Whittaker function $W$, since it can be represented as an oscillatory integral (see e.g~\eqref{W-form2}), gives rise in \S\ref{Whitt-structures} to a Lagrangian submanifold $\Lambda$. We call ${\rm Im}(\Lambda\rightarrow U\backslash S)$ the {\it essential support} of the Whittaker function $W$. The singularities of the Lagrangian mapping $\Lambda\to U\backslash S$ produce large values of $W$. More generally, the singularities of the Lagrangian mapping $\Lambda\rightarrow U\backslash S$ induce a stratification of $\Lambda$ according to the degeneracy of the fibers. The type of degeneracy determines, via its singularity index that we discuss in \S\ref{s:singularities}, the corresponding bump in the asymptotics for the Whittaker function $W$.

There is a certain quantum integrable system whose eigenstates are the spherical Whittaker functions; see for example~\cite{Kost78}. The classical integrable system is the Toda lattice~\cite{Moser:integrable} which we take to be defined on the space $\mathcal{J}^*$ of linear functionals in $\mathfrak{p}^*$ vanishing on $[\mathfrak{u},\mathfrak{u}]$. Here $\mathfrak{p}$ is the tangent space at the origin in $S$ and $\mathfrak{u}$ is the Lie algebra of $U$. Let $\mathscr{L}\subset \mathcal{J}^*$ be the compact isospectral submanifold corresponding to the infinitesimal character of $W$. We review these structures in detail in \S\ref{KT-lattice}.

One of the tools we develop in this paper is an explicit description of $\Lambda\to U\backslash S$ for symmetric spaces $S$ associated with split semisimple real Lie groups $G$. We use in an essential way the symplectic reduction of the Hamiltonian action of $U$ on $T^*(S)$. See Theorem \ref{t:critical} for a more precise statement.

\begin{thm}\label{whitt supp} The Lagrangian $\Lambda$ of a spherical Whittaker function embeds as an open subset of the Toda isospectral manifold $\mathscr{L}$.
\end{thm}

The complement of the essential support ${\rm Im}(\Lambda\rightarrow U\backslash S)$ describes the classically forbidden region of the Toda flow. The corresponding quantum eigenstates -- the Whittaker functions -- then decay rapidly in this region, as we shall establish in \S\ref{s:decay-regime}. So while the archimedean Whittaker functions are not of compact support, the essential support provides a substitute. This is parallel to the theory of Fourier Integral Operators, where we could view $W$ as a distribution whose microlocal support is the Lagrangian $\Lambda$. From the above description of $\mathscr{L}$ as an isospectral variety, we may immediately deduce from Theorem \ref{whitt supp} that all the simple roots evaluated at an element in ${\rm Im}(\Lambda\rightarrow U\backslash S)$ have size at most $\sqrt{\lambda}$. Information of this sort is a crucial input for the proof of Theorem \ref{general tau}.

\subsection{Applications to $\GL_3$}\label{sec:GL3sing} 

Reduction theory allows us convert the rapid decay of $W$ into that of (generic) 
Hecke--Maass cusp forms. We carry this out for $\PGL_3$ and thereby quantify the 
threshold distance into the cusp beyond which a cusp form on $\Gamma\backslash S_3$ must 
decay rapidly. Then, by truncating $\Gamma\backslash S_3$ at this threshold, we can 
quickly establish polynomial {\it upper} bounds on the sup norm. We obtain the following 
sample result, see Remark \ref{rem:upper3} for a discussion of why we have limited the 
scope to $\PGL_3$.
\begin{prop}\label{ess supp} In a Siegel domain, any Hecke--Maass cusp form $f$ on $\Gamma\backslash S_3$ with Laplacian eigenvalue $\lambda$ decays rapidly at height greater than $\lambda^{1+\epsilon}$. Moreover,
\begin{equation}\label{bad-upper}
\pnorm{f}_\infty\ll_\epsilon \lambda^{5/2+\epsilon}\pnorm{f}_2.
\end{equation}
\end{prop}

We now turn to a refinement of Theorem \ref{th:gln} for $\PGL_3$. As was mentioned at the 
end of \S\ref{intro:Whit}, despite the surprising large exponent $c(n)$
for large $n$, the proof of Theorem \ref{th:gln} does not take into account possible singularities of the underlying oscillatory integrals of Whittaker functions.
We consider again the Lagrangian submanifold $\Lambda$ associated to a self-dual 
spherical Whittaker function for $n=3$.

\begin{thm}[Theorem~\ref{prop:CRIT}]\label{n=3 tau} The Lagrangian mapping $\Lambda\rightarrow U\backslash S_3$ induces a stratification
\[
\Lambda^{(0)}\subset \Lambda^{(1)}\subset \Lambda^{(2)}=\Lambda
\]
where $\Lambda^{(i)}$ is a closed submanifold of dimension $i$.
Here the most singular stratum $\Lambda^{(0)}$ consists of two points of type 
$A_3$ singularity, moreover
\[
\Lambda^{(1)} - \Lambda^{(0)}:=\{x\in\Lambda: x\text{ is a type } A_{2} \text{ singularity}\},
\]
and $\Lambda - \Lambda^{(1)}$ is the open dense submanifold of regular points.
\end{thm}

In particular $\Lambda \to U \backslash S_3$ contains two \emph{Whitney pleats} in the 
neighborhood of $\Lambda^{(0)}$.  We refer to \cite{AV} for background on
singularity theory and \S\ref{num-inv} for a brief summary.
We establish an analogous stratification for the Toda isospectral manifold 
$\mathscr{L}\rightarrow \mathfrak{u}_{\rm ab}^*$ in \S\ref{sec:fine-phase} which is 
easier to work with. Then we use Theorem~\ref{whitt supp} to deduce the result for 
$\Lambda \rightarrow U\backslash S_3$.

The above Theorem~\ref{n=3 tau} allows us
to improve the lower bound of Theorem \ref{general tau} for $n=3$, using the method of normal forms of degenerate phase functions. This kind of analysis goes back to~\cite{Berry:waves-Thom}, where each generic 
singularity of corank $1$ and $2$ is studied.
\begin{thm}\label{cor:Pearcey} For any non-zero Whittaker function $W$ on $S_3$ as above, 
with self-dual infinitesimal character and Laplacian eigenvalue $\lambda$, we have
\begin{equation}\label{Pearcey-bump}
\pnorm{W}_\infty\gg \lambda^{3/8}\pnorm{W}_2,
\end{equation}
\end{thm}

We return to the existence of extremal eigenfunctions on compact Riemannian manifolds. For $M$ of negative curvature, one does not expect strong localisation behavior along Lagrangian submanifolds in phase space. For example, the quantum ergodicity theorem establishes the existence of a density $1$ subsequence of an orthonormal basis of eigenfunctions for $L^2(M)$ which do not localise on any proper subvariety of $T^*(M)$. 

Nevertheless, for non-compact Riemannian manifolds, there is a sense in which this 
non-localisation feature of negative curvature asymptotically fails near infinity. 
Heuristically, if a cusp form $f$ on $\Gamma\backslash S$ is well-approximated by its 
Fourier-Whittaker expansion, then $f$ localizes where $W$ does. This is made rigorous 
using the method of~\cite{Temp:p-adic}. The large values of $W$ in Theorem 
\ref{cor:Pearcey} created by their localization along $\Lambda$ then transfer to those of 
$f$.

\begin{cor}\label{cor:38} For any Hecke--Maass cusp form $f$ on $\Gamma\backslash S_3$, with self-dual infinitesimal character and Laplacian eigenvalue $\lambda$, we have
\begin{equation*}
\pnorm{f}_\infty\gg_\varepsilon \lambda^{3/8-\varepsilon}\pnorm{f}_2.
\end{equation*}
\end{cor}

For $\PGL_2$, the Whittaker functions $W$ on the Poincar\'e upper-half plane $S_2$ are 
expressed in terms of $K$-Bessel functions, and classical estimates in the transition 
range~\cite{Ba} imply that 
\begin{equation}\label{Bessel-bump}
\pnorm{W}_\infty\gg \lambda^{1/12}\pnorm{W}_2.
\end{equation}
As a consequence we recover the result of~\cite{SaMo} on the growth of Hecke--Maass cusp forms on the modular surface $\SL_2(\Z)\backslash S_2$. 
The above results are the appropriate generalizations to $\PGL_3$. Observe that \eqref{Pearcey-bump} and \eqref{Bessel-bump} may be rewritten as
\begin{equation}\label{intro-A2}
\pnorm{W}_\infty\gg \sqrt{\lambda}^{c(n)+\beta_n}\pnorm{W}_2 
\end{equation}
for $n=2$ and $3$, where $\beta_n$ is the singularity index of the Lagrangian mapping $\Lambda \to U\backslash S_n$, and $\beta_2=\frac16$, $\beta_3=\frac14$. 
The $A_2$ singularity for the
$\GL_2$ Whittaker function, or equivalently $K$-Bessel function, arises from the 
``turning point" of the projection of the circle $\mathscr{L}$ centered at the origin in 
$\mathfrak{p}^*$
of radius $\sqrt{\lambda-1/4}$ onto the $\mathfrak{u}^*$-axis. At the fold, the $K$-Bessel function is modelled by the Airy function, the prototypical example of a function exhibiting a transition from an oscillatory to a decay regime. 
Several natural thresholds encountered in analytic number theory, 
especially problems having to do with the bounding of periods such as in the work of 
Bernstein and Reznikov~\cite{BeRe}, are directly related to this transition behavior of 
the Airy function. Similarly, the Pearcey function is associated with $A_3$ singularities 
as explicated by Berry~\cite{Berry:waves-Thom}, and we have shown that it models the peak 
behavior of $\GL_3$ Whittaker functions, see \S\ref{LOWER-GL3} for details.

Finally, we remark that the existence of tempered $\PGL_3$ cusp forms satisfying the 
self-dual condition at infinity of Corollary \ref{cor:38} can be seen by taking symmetric 
square lifts (and character twists thereof) of tempered $\PGL_2$ Hecke--Maass cusp forms. 
The restriction to such $f$ should be unnecessary and we have assumed it solely to 
simplify certain local calculations. Note that locally self-dual at infinity does not 
imply globally self-dual, as for example is shown by twisting a globally self-dual form 
by a non-quadratic Dirichlet character. 

\subsection*{Acknowledgments}  We would like to thank Michael Berry, Erez Lapid, Elon 
Lindenstrauss, Simon Marshall, Philippe Michel, and Andre Reznikov for helpful
discussions. We thank the referees for many constructive comments that improved the quality of the paper.
Some of the results of this paper were first announced at the Oberwolfach workshop $1135$ on the analytic theory of automorphic forms and further presented at various other meetings, e.g. the Banff workshop on Whittaker functions and Physics, and the 17th Midrasha Mathematicae at Jerusalem in honor of Peter Sarnak. We thank the organizers for these invitations and the participants for their helpful comments.

\section{Outline of proofs}\label{proof-sketch}

We now provide a brief outline of the proofs of the results stated in the introduction. For the reader's benefit, we follow the same subsection structure of the introduction. 

\subsection{Proof sketch of results in \S\ref{sub:firstmain}}

\subsubsection{Reduction of Theorem \ref{th:gln} to Theorem \ref{general tau}}\label{reduction-step}

The proof of Theorem \ref{th:gln} begins in \S\ref{sec:intro:gl2} by considering the integral of the cusp form $f$ over a closed unipotent orbit against a non-degenerate character. We obtain in this way the global Whittaker function
\[
W_f(g)=\int_{(\Gamma\cap U)\backslash U} f(ug)\overline{\psi (u)}du.
\]
Since the cycle $(\Gamma\cap U)\backslash U$ is compact, we can deduce lower bounds for $f$ from those of $W_f$.

The multiplicity one of the spherical Whittaker space and convexity bounds on {Rankin--Selberg} $L$-functions then allow us to replace $W_f$ by the Jacquet-Whittaker function $W_\nu$. As the notation suggests, this latter function is of purely local nature: it sees only the infinitesimal character $\nu$ but not the global automorphic form $f$. This then reduces the proof of Theorem \ref{th:gln} to Theorem \ref{general tau}.

\subsubsection{Sketch of proof of Theorem \ref{general tau} for $\PGL_n(\R)$}
In the special case of $G=\PGL_n(\R)$ we may prove Theorem \ref{general tau} as follows 
(see \S\ref{GLn-argument} for details). We consider the zeta integral
\[
\Psi(s,W_\nu,\overline{W_\nu})=\Gamma_\R(ns)\int_{U_{n-1}\backslash \GL_{n-1}(\R)}
\left|W_\nu
\begin{pmatrix}
g & 0 \\
0 & 1
\end{pmatrix}
\right|^2 \abs{\det(g)}^{s-1}d\dot g,
\]
where $d\dot g$ is a quotient of normalized Haar measures on $U_{n-1}\backslash 
\GL_{n-1}(\R)$. 
Let $T_{n-1}$ be the torus consisting of positive diagonal matrices in $\GL_{n-1}(\R)$, 
embedded in $\GL_n(\R)$ as above.
Measure identifications and transformation properties of $W_\nu$ allow one to write
\[
\Psi(s,W_\nu,\overline{W_\nu})=
\Gamma_\R(ns)\int_{T_{n-1}} \left|W_\nu
\begin{pmatrix}
a & 0 \\
0 & 1
\end{pmatrix}
\right|^2 \det(a)^{s-2}\delta(a)^{-1}da,
\]
 up to non-zero absolute constant depending on volume normalization, and the Stade formula (see~\eqref{clean stade}) states that the above integral is equal to the local Rankin-Selberg $L$-function
\[
L_\R(s,\pi_\nu\times\tilde\pi_\nu)/L_\R(1,\pi_\nu\times\tilde\pi_\nu).
\]
Specializing to $s=1$ we obtain the $L^2$-norm squared of $W_\nu$, and we see that it is normalized to be equal to $1$. The idea is to take ${\rm Re}(s)$ large which puts a greater weight on the region where $W_\nu$ has large values, and comparing this with the volume of the region will yield the bound of Theorem~\ref{general tau}. 

Carrying out this strategy, we see from Stirling's formula that $\Psi(\sigma,W_{t\nu},\overline{W_{t\nu}})$ has size $t^{(\sigma-1)\dim U}$ as $t\to \infty$. On the other hand, Theorem \ref{whitt supp} implies that $\Psi(\sigma,W_{t\nu},\overline{W_{t\nu}})$ is majorized by
\[
\max_{a\in A} |W_{t\nu}(a)|^2\int_{{\rm Im}(\Lambda_{t\nu}\rightarrow U\backslash S)} 
\det(a)^{\sigma-2} \delta(a)^{-1}da.
\]
For $\sigma>n-1$, the integral converges to a constant times $t^{(\sigma-1)\dim U-c(n)}$. We deduce the bound $\max_{a\in A} |W_{t\nu}(a)|^2 \gg t^{c(n)}$, as desired.

\subsubsection{Sketch of proof of Theorem \ref{general tau} for general $G$}

A spherical Whittaker function with infinitesimal character $\nu$ is a constant multiple of the oscillatory integral
\begin{equation}\label{W-form2}
W_\nu(a)=\delta(a)^{1/2}\nu(\mathsf{w}a\mathsf{w})\int_U \delta(\mathsf{w} u)^{1/2}e^{i(B(H_\nu, 
H({\sf w}u))-\langle\ell_1,aua^{-1}\rangle)}du,
\end{equation}
where $a\in A$. See \S\ref{sec:f} for the notation used in the above expression, and 
Lemma~\ref{c-o-v}. The size of the $\delta(a)^{1/2}$ factor is easy to determine; that of 
the 
oscillatory integral is more subtle, for the phase function depends on both parameters 
$\nu$ and $a$.

The method of stationary phase states that the asymptotic of this integral is determined by the critical set of the phase function $B(H_\nu, H({\sf
w}u))-\ell_1(aua^{-1})$ measuring the interaction of the Iwasawa projection $H({\sf w}u)$ (tested by $\nu$) with characters $u\mapsto e^{i\ell_1(aua^{-1})}$. If there are no critical points, then the integral decays rapidly, overwhelming the polynomial growth of $\delta(a)^{1/2}$. If there do
exist critical points, then the asymptotic size of the above integral is governed by local contributions around each one. A non-degenerate critical point makes a
contribution $t^{-\dim U/2}$ to the size of $W_{t\nu}(ta)$. A degenerate critical point will make a larger contribution, of size $t^{-\dim U/2 +\beta}$ for a certain rational number $\beta$ which is a numerical invariant of the degeneracy.

To prove Theorem \ref{general tau} we show in \S\ref{non-deg-phase} that for every $\nu$ there exists $a$ such that the above phase function admits critical points whose local contributions do not cancel. For this, we adapt the method of H\"ormander~\cite{Hormander:FIO-I,Duistermaat:oscillatory} in the theory of Fourier integral operators
as follows. To obtain upper bound estimates for the operators, the symbol is traditionally chosen to be transverse to the Lagrangian $\Lambda$. However
for our purpose of establishing a lower bound we make the opposite choice of a symbol which is tangent to $\Lambda$, and then the modified phase $B(H_\nu, H({\sf
w}u))-\ell_1(aua^{-1})- \langle\xi,a\rangle$ is a Morse-Bott function in the variables $(a,u)$. This produces a {\it lower bound} (not necessarily sharp, since at degeneracies the lower bound could be stronger) on the oscillatory integral of size $t^{-\dim U/2}$. When the size of half-density $\delta(a)^{1/2}$ is taken into account, this yields the exponent $c(G)$.

\subsection{Theorem \ref{whitt supp} and the method of co-adjoint orbits}\label{sub:coadjoint}
We now give an intuitive explanation for why one should expect to realize the Whittaker Lagrangian $\Lambda_\nu$ in $\mathscr{L}_\nu$, as stated in Theorem \ref{whitt supp}. Our inspiration is the geometric setting of the {\it method of co-adjoint orbits}.

Consider the action of $G$ on the space of linear functionals $\g^*$ given by the co-adjoint action. For $g\in G$ and $\lambda\in\g^*$ this is defined as $\Ad_g^*\lambda=\lambda\circ\Ad_{g^{-1}}$, where $\Ad: G\rightarrow {\rm Aut}(\mathfrak{g})$ is the adjoint representation. The orbits under this action are endowed with a natural $G$-invariant symplectic form, which at a point $\lambda$ is given by the formula $\Omega_\lambda(X,Y)=-\lambda([X,Y])$. The action of $G$ on an orbit $\mathscr{O}$ is Hamiltonian with corresponding moment map the inclusion $\Phi_G: \mathscr{O}\hookrightarrow\g^*$. 

We are particularly interested in co-adjoint orbits attached to $\xi\in\mathfrak{p}^*$. A 
natural way of obtaining them is to first consider the cotangent bundle $T^*(S)$. This 
receives a $G$-action inherited from the natural $G$-action on $S$ by isometries. We make 
the equivariant identification $T^*(S)=G\times_K\mathfrak{p}^*$ under which the moment 
map $T^*(S)\rightarrow\g^*$ for the $G$-action is described by 
$[g,\xi]\mapsto\Ad_g^*\xi$, see \S\ref{tau n}. Then the image of any $G$-orbit in 
$T^*(S)$ is a coadjoint orbit in $\g^*$ associated with some $\xi\in\mathfrak{p}^*$. 

The method of co-adjoint orbits states~\cite{Rossmann:representations-orbits} that, in favorable circumstances, irreducible unitary representations $\pi$ of $G$ will be in finite-to-one correspondence with co-adjoint orbits $\mathscr{O}$. The association of a unitary representation with the Hamiltonian system of the symplectic $G$-manifold $\mathscr{O}$ is referred to as {\it geometric quantization}. Moreover, operations in the unitary dual (e.g. induction, restriction) should correspond to operations on corresponding orbits (e.g. intersection, projection). Other parallels exist; for example, the uncertainty principle is expressed in this set-up as a correspondence between the vectors in the unitary representation $\pi$ and balls of unit volume in the co-adjoint orbit $\mathscr{O}=\mathscr{O}_\pi$. 

Of greatest interest to us is the following situation. For a connected subgroup $K$ of $G$ with Lie algebra $\mathfrak{k}$, the level sets of the corresponding moment map $\Phi_K:\mathscr{O}\hookrightarrow\g^*\rightarrow\mathfrak{k}^*$ should in principle correspond to phase states with quantities conserved by $K$. For example, taking $K$ to be a maximal compact subgroup, spherical Whittaker functions $W$ are associated with $K$-fixed vectors of irreducible unitary unramified representations of $G$. From the tempered hypothesis on $W$, these representations are obtained by induction from some $\nu\in i\mathfrak{a}^*\subset i\mathfrak{p}^*$. Letting $\mathscr{O}$ be the coadjoint orbit of ${\rm Im}(\nu)\in\mathfrak{p}^*$, isolating $\Phi_K^{-1}(0)$ in $\mathscr{O}$ then corresponds to picking out $K$-fixed vectors in $\pi$.

Furthermore, given two subgroups $U,K<G$, one can hope to understand the $U$-isotypic distribution of a $K$-fixed vector in $\pi$ via the projection map $\Phi_K^{-1}(0)\rightarrow\mathfrak{u}^*$. Theorem \ref{whitt supp} carries out this yoga for $K$ a maximal compact subgroup of $G$ and $U$ the unipotent radical of a Borel.

On one hand, the Toda isospectral manifold $\mathscr{L}_\nu$ is the intersection $\Phi_K^{-1}(0)\cap\Phi_{U_{\rm der}}^{-1}(0)$ in the coadjoint orbit $\mathscr{O}$ (see \S\ref{sub:reduction}), where $U_{\rm der}=[U,U]$ is the commutator subgroup. This intersection then admits a Lagrangian mapping to $\mathfrak{u}_{\rm ab}^*$, with $\mathfrak{u}_{\rm ab}^*$ denoting the characters of $\mathfrak{u}$ vanishing on $\mathfrak{u}_{\rm der}$. When $G=\GL_2(\R)$, for example, one obtains the projection from the circle of radius $\xi$ to the $\mathfrak{u}^*$-axis. On the other, the Whittaker Lagrangian $\Lambda_\nu$ admits a similar description with respect to the moment maps arising from the natural $G$-action on the cotangent bundle $T^*S\rightarrow S$ (see Proposition \ref{p:moment}). Reducing $\Lambda_\nu$ by the $U$-action then defines an open embedding from $\Lambda^{\rm red}_\nu\rightarrow U\backslash S$ into $\mathscr{L}_\nu\rightarrow\mathfrak{u}_{\rm ab}^*$. This is the statement of Theorem \ref{t:critical}, which makes more precise Theorem \ref{whitt supp} from the introduction.

\subsection{Proof sketch of results in \S\ref{sec:GL3sing}}

\subsubsection{Sketch of proof of Proposition \ref{ess supp}}
To establish the rapid decay of $f$ high in the cusp, one first expands $f$ in its Fourier-Whittaker expansion, see \S\ref{sec:ess sup}. One must then check that every term in the expansion is itself evaluated high enough into the cusp for the decay estimates of Theorem \ref{whitt supp} to apply; this is an exercise in reduction theory, which we carry out for $\GL_3(\R)$. In this way, the decay estimate on Whittaker functions of Theorem \ref{whitt supp} transfers, at least for $n=3$, to the cusp form $f$.

To deduce an upper bound on the sup norm of $f$ from a quantitative estimate of its essential support, we argue as follows, see \eqref{uniform Sarnak} for details. First recall a result of Sarnak \cite{SaMo} which states that a cusp form $f$ of eigenvalue $\lambda$ on a compact locally symmetric space of dimension $d$ and rank $r$ satisfies
\begin{equation}\label{d-r-omega}
	\pnorm{f}_\infty\ll\lambda^{(d-r)/4}\pnorm{f}_2.
\end{equation}
In fact, this holds for non-compact locally symmetric spaces as well, as long as one restricts to nice enough bounded subsets, such as geodesic balls. The key is that the quantitative dependence of the implied constant on the injectivity radius in \eqref{d-r-omega} is rather easy to explicate. So we simply go through Sarnak's proof of \eqref{d-r-omega} on the \emph{truncation} of $\Gamma\backslash S_3$ to the essential support of $f$, since it has positive calculable global injectivity radius. 

\subsubsection{Proof sketch of Theorem \ref{n=3 tau} and Corollary \ref{cor:38}}

The description of $\Lambda_\nu$ given in Theorem \ref{whitt supp} is convenient for computations: roughly speaking, the equations defining the fiber over $a\in U\backslash S=A$ are the tridiagonal symmetric matrices with off-diagonals the positive simple roots of $a$ and characteristic polynomial agreeing with that of $\nu$. For $G=\PGL_3(\R)$ and $\nu$ self-dual, this boils down to the following problem.

{\it Let $t>1$. Let $\mathcal{J}$ denote the real tridiagonal symmetric $3\times 3$ matrices. Determine the intersection configuration of the solutions $s\in\mathcal{J}$ having fixed non-zero off-diagonal entries to the cubic equation $\det(s)=0$ and the quadratic equation $\pnorm{s}=t^2$.}

The \S\ref{sec:fine-phase} is dedicated to the solution of this problem. In particular, the $A_3$ singularities are created when the two equations have two intersection points, both with multiplicity 3. Stationary phase asymptotics for $A_3$ singularities then produce the $\lambda^{3/8}$ lower bound for the corresponding spherical Whittaker function. Finally, to deduce the bounds on the cusp form $f$ as stated in Corollary \ref{cor:38} one follows the argument sketched in \S\ref{reduction-step}.

\section{Notation and preliminaries}\label{sec:f}

In this section we establish basic notation that we'll need for later calculations. We will take $G$ to be a split semi-simple real Lie group throughout this section.

\subsection{Basic notation on roots}\label{sec:root-notation}
Let $\Theta$ denote a Cartan involution on $G$. Denote by $\theta$ the differential of $\Theta$ on $\mathfrak{g}$, the (real) Lie algebra of $G$. One has an orthogonal direct sum decomposition $\mathfrak{g}=\mathfrak{p}\oplus\mathfrak{k}$ into the $-1$ and $+1$ eigenspaces of $\theta$. Then $\mathfrak{k}$ is the Lie algebra of $K$, the group of fixed points of $\Theta$. 

Choose a maximal abelian subalgebra $\mathfrak{a}$ of $\mathfrak{p}$. The Weyl group of 
$G$ is $W=W(\mathfrak{g},\mathfrak{a})=N_K(\mathfrak{a})/Z_K(\mathfrak{a})$, the quotient 
of the normalizer by the centralizer of the adjoint action of $K$ on $\mathfrak{a}$. 
Denote by $A=\exp(\mathfrak{a})$ the associated closed connected subgroup of $G$. Then 
$A$ is a maximal split torus of $G$ preserved by $\Theta$. Let $\mathfrak{a}^*={\rm 
Hom}(\mathfrak{a},\R)$ be the dual of $\mathfrak{a}$ and 
$\mathfrak{a}_\C^*=\mathfrak{a}^*\otimes\mathbb{C}={\rm 
Hom}(\mathfrak{a},\mathbb{C})=\mathfrak{a}^*+i\mathfrak{a}^*$ its complexification. We 
agree to the notational convention for which $\langle \nu,X\rangle$ is the evaluation of 
$\nu\in\mathfrak{a}_\C^*$ at $X\in\mathfrak{a}$.  Moreover, when $a\in A$ we write 
$\nu(a)$ for $e^{\langle\nu,\log a\rangle}$. 

Let $\Delta=\Delta(\mathfrak{g},\mathfrak{a})$ denote the set of (restricted) roots. We 
have $\mathfrak{g}=\mathfrak{a}\oplus\bigoplus_{\alpha\in\Delta}\mathfrak{g}_\alpha$, 
with each $\mathfrak{g}_\alpha$ of dimension one. For $\alpha\in\Delta$ let $H_\alpha$ be 
the corresponding co-root; this, by definition, is the unique element in 
$\mathfrak{a}_\alpha=[\g_\alpha,\g_{-\alpha}]\subseteq \mathfrak{a}$ such that $\langle 
\alpha, H_\alpha \rangle =2$. For a root $\alpha\in\Delta$ let $X_\alpha\in\g_\alpha$ be 
choosen such that $[X_{-\alpha},X_{\alpha}]=H_\alpha$. The choice of a system of simple 
roots $\Pi$ determines a set of positive roots $\Delta_+$. Let 
$\mathfrak{u}=\bigoplus_{\alpha\in\Delta_+}\mathfrak{g}_\alpha$ and 
$\overline{\mathfrak{u}}=\bigoplus_{\alpha\in\Delta_+}\mathfrak{g}_{-\alpha}$. When $\ell\in\mathfrak{u}^*$ we sometimes write $\langle \ell,u\rangle$ or $\ell(u)$ to mean $\langle\ell,\log u\rangle$. Let $\rho\in\mathfrak{a}^*$ be half the sum of the positive roots; thus $\langle \rho, \cdot \rangle$ is half the trace of the adjoint action on $\mathfrak{u}$.

Let $U$ and $\overline{U}$ be the connected closed subgroups of $G$ whose Lie algebras are $\mathfrak{u}$ and $\overline{\mathfrak{u}}$, respectively. We have $\overline{U}=\Theta U$. Let $B$ be the unique Borel subgroup of $G$ containing $A$ and $U$. Then $U$ is is the unipotent radical of $B$, the Lie algebra of $B$ is $\mathfrak{b}=\mathfrak{a}\oplus\mathfrak{u}$, and one has the Langlands decomposition $B=MAU$, where $M=B\cap K$.

Denote by ${\rm Ad}:G\rightarrow {\rm Aut}(\mathfrak{g})$ the adjoint representation. For $g\in G$ and $X\in\mathfrak{g}$ we will often use the shorthand $X^g$ to denote ${\rm Ad}_{g^{-1}}X$. (The inverse in the latter notation is there for the right-action rule $X^{gh}=(X^g)^h$ to hold.) Similarly, for $g,z\in G$ we write $z^g=g^{-1}zg$. For $g\in G$ and $\lambda\in \mathfrak{g}^*$ we let $\mathrm{Ad}^*_g \lambda := \lambda \circ\mathrm{Ad}_{g^{-1}}$.

Fix a choice $B(\cdot,\cdot):\g\times\g\rightarrow \R$ of $\Ad$-invariant non-degenerate symmetric bilinear form, normalized to be positive definite on
$\mathfrak{p}$. Then $-B(X,\theta Y)$ is positive definite on $\mathfrak{g}$; let $\pnorm{X}^2=-B(X,\theta X)$ be the associated norm on elements of
$\mathfrak{g}$. The restriction of $B(\cdot,\cdot)$ to $\mathfrak{a}$ defines a positive definite bilinear form. We use $B(\cdot,\cdot)$ to identify
$\mathfrak{a}^*$ with $\mathfrak{a}$ as follows. 
For $\xi\in \mathfrak{a}^*$, we let $H_\xi$ denote the unique element in $\mathfrak{a}$ such that $\langle \xi,H\rangle =B(H_\xi,H)$ for every $H\in\mathfrak{a}$. 
Furthermore, we can extend $B(\cdot,\cdot)$ to a hermitian scalar product on $\mathfrak{a}_\C$, allowing us to identify $\mathfrak{a}^*_\C$ with $\mathfrak{a}_\C$. 
If $\nu\in i\mathfrak{a}^*$ with $\xi={\rm Im}\,\nu$, we write $H_\nu$ for $H_\xi$.

The root hyperplane (or wall) associated to the element $\alpha\in\Delta$ is the linear subspace of $\mathfrak{a}$ on which it vanishes. The Weyl chambers are the connected components of the complement of all walls in $\mathfrak{a}$. The union of all Weyl chambers is the set $\mathfrak{a}_{\rm reg}$ of regular elements. Let $\mathfrak{a}_+$ (resp. $\mathfrak{a}^*_+$) denote the positive Weyl chamber in $\mathfrak{a}$ (resp. $\mathfrak{a}^*$). The Weyl group acts simply transitively on the Weyl chambers. An element $H$ is regular if and only if $H^w=H$ for some $w\in W$ implies $w=e$. The long Weyl element, which we denote by $\mathsf{w}$, sends $\mathfrak{a}_+$ to $-\mathfrak{a}_+$. We make once and for all a choice of a lift of the longest Weyl group element to an element in $K$ and we continue to write it as  $\mathsf{w}$. 

\subsection{Iwasawa decomposition}\label{sec:Iwasawa}
The Iwasawa decomposition is $G=UAK$. We denote by $\upkappa(g)$ the unique element in $K$ such that $g\upkappa(g)^{-1} \in AU$, and $\uptau (g)=g\upkappa(g)^{-1}$.

For $a\in A$ let $\delta (a)=|\det (\Ad(a)_{|_U})|$, the Jacobian of the automorphism of $U$ sending $u$ to $aua^{-1}$. Thus, if $du$ is any Haar measure on $U$, then $\int_U f(aua^{-1})du=\delta (a)^{-1}\int_U f(u) du$. Since $a\in A$ acts on $X\in\g_\alpha$ via the adjoint action by multiplication by $\langle\alpha,\log a\rangle$ we have $\delta(a)=\prod_{\alpha\in\Delta_+} \alpha(a)=\rho(a)^2$. For any choice of left-invariant Haar measures $du$, $da$, $dk$ on $U$, $A$, and $K$, respectively, the product measure $dg=\delta(a)^{-1}du\, da\, dk$ defines a left-invariant Haar measure on $G$.

Recall the Iwasawa decomposition of the Lie algebra $\mathfrak{g}=\mathfrak{u}\oplus\mathfrak{a}\oplus\mathfrak{k}$. We denote by $E_\mathfrak{a}$ (resp., $E_\mathfrak{u}$, $E_\mathfrak{k}$) the projection from $\mathfrak{g}$ onto $\mathfrak{a}$ (resp., $\mathfrak{u}$, $\mathfrak{k}$). Note that unlike $E_{\mathfrak{a}}$, the projections  $E_\mathfrak{u}$, $E_\mathfrak{k}$ are not orthogonal with respect to $B$.

The map $H:G\rightarrow \mathfrak{a}$ sending $g=ue^Xk$ ($u\in U, X\in\mathfrak{a}, k\in K$) to $X$ is called the Iwasawa projection. Its derivative was computed in \cite[Corollary 5.2]{DKV}. We state and prove a consequence of this which will be useful for us in \S\ref{tau n}.

\begin{lem}\label{lem:DKV} For $\xi\in\mathfrak{a}^*$, the right derivative along $X\in \mathfrak{g}$ of the function $g\mapsto \langle \xi, H(g) \rangle$ is
equal to $\langle \xi ,X^{\upkappa(g)^{-1}} \rangle$.
\end{lem}

\begin{proof} From the linearity of $\xi\mapsto \langle \xi,H\rangle$, we may pass the derivative inside the bracket. From \cite[Lemma 5.1]{DKV} the directional derivative along $X$ of the Iwasawa projection $g \mapsto H(g)$ is $E_\mathfrak{a}(X^{\upkappa(g)^{-1}})$. This gives 
\[
\frac{d}{dt} \langle \xi , H(ge^{tX})\rangle |_{t=0} =\langle \xi, E_\mathfrak{a}(X^{\upkappa(g)^{-1}})\rangle.
\]
 As $\mathfrak{a}^*$ is orthogonal to $\mathfrak{k}\oplus\mathfrak{u}$ we have 
 $\langle\xi,E_\mathfrak{a}(X^{\upkappa(g)^{-1}})\rangle = \langle\xi, 
 X^{\upkappa(g)^{-1}}\rangle$, as desired.
\end{proof}

If we apply the above lemma to the case when $g=\mathsf{w}u$, for $u\in U$ then we recover the following result which appears in~\cite[Proposition 9.1]{Cohn}, and also in~\cite{DKV}. Let $\xi\in\mathfrak{a}^*$ be given. Then $u$ is a critical point of $\langle \xi ,H(\mathsf{w}u)\rangle$ if and only if $\xi^{\upkappa (\mathsf{w}u)}\in\mathfrak{a}^*$. In particular, if $\xi$ is regular then the only critical point of $\langle \xi, H(\mathsf{w}u) \rangle$ is the identity $e$.

\subsection{Bruhat decomposition}

Next we recall the Bruhat decomposition,
\[
G=\bigsqcup_{w\in W} G_w,
\]
where $G_w=BwU_w$ and $U_w=U\cap (w^{-1}\overline{U}w).$ The cell $G_\mathsf{w}=B\mathsf{w}U$ associated with the long Weyl element $\mathsf{w}$ is called the {\it big cell;} it is open and dense in $G$. For any $w\in W$ let $\mathfrak{u}_w$ denote the Lie algebra of $U_w$.  Note that $\mathfrak{u}_\mathsf{w}=\mathfrak{u}$. We have
\[
\mathfrak{u}_w=\bigoplus_{\alpha\in\Delta_+(w)}\g_\alpha,\quad\text{where}\quad  \Delta_+(w)=\{\alpha\in\Delta_+: -w\alpha\in\Delta_+\}.
\]
Write $\mathfrak{u}^w$ for the direct sum of the $\g_\alpha$ for $\alpha\in\Delta_+-\Delta_+(w)$, so that $\mathfrak{u}=\mathfrak{u}_w\oplus\mathfrak{u}^w$.

We call an element $\ell\in \mathfrak{u}^*$ {\it degenerate} if it vanishes identically on some simple root space $\mathfrak{g}_\alpha$, $\alpha\in \Pi$. We call it {\it non-degenerate} otherwise. Since at least one of the roots in $\Delta_+-\Delta_+(w)$ is simple, $\ell$ is degenerate if and only if it belongs to $\mathfrak{u}_w^*$ for some $w\neq \mathsf{w}$. The set of non-degenerate functionals is therefore equal to $\mathfrak{u}^* - \bigcup_{w\neq \mathsf{w}} \mathfrak{u}^*_w$.

The Bruhat decomposition of $G$ gives rise to a cellular decomposition on the flag variety $B\backslash G$. By definition, these cells are the orbits of the cosets $Bw$, where $w\in W$, under the natural right-action of $U$ on $B\backslash G$. When we make the identification $B\backslash G=M\backslash K$, the action of $U$ on $B\backslash G$ induces a right-action of $U$ on $M\backslash K$ given by $(k,u)\mapsto M\upkappa (ku)$. The Bruhat cell $B\backslash BwU$ is then identified with the image $S_w^+$ of the map
\[
U\rightarrow M\backslash K, \qquad u\mapsto M\upkappa (wu).
\]
We thus obtain the following decomposition
\[
M\backslash K=\bigsqcup_{w\in W}S_w^+.
\]
We are borrowing the notation $S_w^+$ (for {\it stable} manifold) from \cite[\S3]{DKV}.
We note that $S_w^+$ are stable under right $M$-action. Moreover under the inversion $k\mapsto k^{-1}$, the cell $S_w^+$ is mapped bijectively to $S_{w^{-1}}^+$.

Compare the following result to [{\it loc. cit.}, Proposition 7.1]. 

\begin{lem}\label{lem:dkappa}
For $w\in W$ the differential of the above map $U\rightarrow S_w^+$  is given by
\begin{equation}\label{} 
	d\upkappa(wu)(Y) = E_\mathfrak{k}(Y^{\upkappa (wu)^{-1}})^{\upkappa(wu)}.
\end{equation}
The restriction to $U_w$ induces a diffeomorphism of $U_w$ onto $S_w^+$. 
\end{lem}
\begin{proof}
We begin by writing $wu = \uptau(wu) \upkappa(wu)$; thus for any $t\in\R$ we have $wue^{tY} = \uptau(wu) e^{t Y^{k}} k^{-1}$, where we have set $k=\upkappa(wu)^{-1}$. Then $d\upkappa(wu)(Y)$ is equal to
\[
\frac{d}{dt}\upkappa\left(wue^{tY}\right)|_{t=0}=\frac{d}{dt}\upkappa\left(e^{tY^k}k^{-1}\right)|_{t=0}=\frac{d}{dt}e^{tE_\mathfrak{k}(Y^{k})}k^{-1}|_{t=0}.
\]
Conjugating this by $k$, we obtain the desired formula.	

For the second statement, it suffices to observe that the isotropy subgroup of the point $Bw$ for the $U$-action on $B\backslash G$ is the analytic subgroup $U^w=U\cap w^{-1}Uw$ of $G$ whose Lie algebra is $\mathfrak{u}^w$. Since $U=U_wU^w$ and $U_w\cap U^w=\{e\}$ the claim follows.\end{proof}

\subsection{Spherical representations and invariants} \label{s:spherical-rep}

Let $S=G/K$ be the globally Riemannian symmetric space associated with $G$, and 
$\mathscr{D}_G(S)$ the ring of left $G$-invariant differential operators on $S$. The 
Harish-Chandra isomorphism identifies the differential eigencharacters ${\rm 
Hom}(\mathscr{D}_G(S),\C)$ with the space of (spherical) infinitesimal characters 
$\mathfrak{a}_\C^*/W$. For $\nu\in \mathfrak{a}_\C^*$ let $\lambda_\nu$ be the associated 
Laplacian eigenvalue given by evaluating the associated differential eigencharacter on 
$\Delta$. The Laplacian being an order two differential operator, when we scale $\nu$ by 
$t>1$ we obtain $\lambda_{t\nu}\asymp t^2\lambda_\nu$.

For $\nu\in i\mathfrak{a}^*$ consider the representation of $G$ by right-translation on the space of smooth functions $f:G\rightarrow\C$ satisfying
\[
f(bg)=f(g)\delta(b)^{1/2}e^{\langle\nu,H(b)\rangle}\qquad g\in G, b\in B.
\]
The inner product $\int_K f_1(k) \overline{f_2}(k) dk$ on this space is $G$-invariant.
Here $dk$ is the probability Haar measure on $K$.
 We denote by $\pi_\nu$ the completion of this space relative to this normalized inner product. Then $\pi_\nu$ is an irreducible unitary spherical tempered
representation (spherical principal series~\cite[\S5]{book:Wall1}). We have $\pi_\nu\simeq \pi_{\nu'}$ if and only if $\nu=w\nu'$ for some $w\in W$. We shall only be interested in $\nu$ regular; so that $w\nu\neq \nu$ unless $w=e$. The isomorphism classes of irreducible unitary regular tempered spherical representations of $G$ are parametrized by $\nu$ lying in the positive chamber $i\mathfrak{a}_+^*$.

We define the {\it height of $G$} to be
\[
{\rm ht}(G)=\sum_{\alpha\in\Delta^+}{\rm ht}(\alpha),
\]
where ${\rm ht}(\alpha)$ is the sum of the coefficients of $\alpha$ when written as a linear combination of the positive simple roots. The height of $G$ has the following property: for an element $a\in A$ and a positive real $t>0$ let $ta$ be the unique element in $A$ whose simple roots satisfy $\alpha(ta)=t\alpha(a)$ for all  $\alpha\in\Pi$. Then one deduces that
\begin{equation}\label{a-scaling}
\delta(ta)= \prod_{\alpha\in \Delta_+} \alpha(ta) = t^{{\rm ht}(G)}\delta(a).
\end{equation}
In particular, the height of $G$ describes the size of the {\it spherical vector} in $\pi_\nu$ along directions $ta$. Recall that the spherical vector in $\pi_\nu$ is the unique $K$-fixed vector taking value $1$ at the identity. It has $L^2$-norm $1$ and is given by the expression
\[
f_\nu(g)=e^{\langle\rho+\nu,H(g)\rangle}=\delta(g)^{1/2}e^{\langle\nu,H(g)\rangle}.
\]
Here and elsewhere, $\delta(g)=\delta(a)$ if $g=uak$; alternatively, $\delta(g)=e^{2\langle \rho, H(g)\rangle}$.

The height of $G$ can also be used to describe the continuous spectrum 
$[\lambda_1(S),\infty )$ of the Laplacian acting on $L^2(S)$ for the symmetric space $S$. 
Indeed one has 
$\lambda_1(S)=\lambda_\nu$ for $\nu=0\in i\mathfrak{a}^*$, the Laplace eigenvalue for the 
trivial 
infinitesimal character in the notation of \S\ref{s:spherical-rep}. For example, when 
$G=\PGL_n(\R)$ and $S_n=\PGL_n(\R)/{\rm 
PO}(n)$ we have $\lambda_1(S_n)=(n^3-n)/24$ (see, for example, ~\cite{Miller02}) and
\[
{\rm ht}(\PGL_n)= \sum_{i=1}^{n-1}\sum_{j=1}^ii=8\lambda_1(S_n)-\dim U.
\]

\subsection{Whittaker models, phase functions, and associated Lagragians}\label{Whitt-structures}
We now describe various Whittaker structures associated with the above representations $\pi_\nu$.

Let $\psi$ be a unitary character of $U$. Then $\psi$ factors through $U_{\rm der}=[U,U]$ and since the abelianization of $U$ is $U_{\rm ab}=U/U_{\rm der}=\prod_{\alpha\in\Pi}U_\alpha$, where $U_\alpha$ is the analytic subgroup with Lie algebra $\g_\alpha$, we may factorize $\psi$ as $\psi=\prod_{\alpha\in\Pi}\psi_\alpha$. We call $\psi$ {\it non-degenerate} if each $\psi_\alpha$ is non-trivial. Using the elements $X_\alpha\in\mathfrak{g}_\alpha$ from \S\ref{sec:root-notation}, we identify $\R$ with $U_\alpha$ via the map $t\mapsto e^{t X_\alpha}$. We denote by $\psi_1=\prod_{\alpha\in\Pi}\psi_{1,\alpha}$ the unique character of $U$ such that $\psi_{1,\alpha}(e^{t X_\alpha})=e^{2\pi it}$ for all $\alpha\in\Pi$. We let $\ell_1$ be the unique element in $\mathfrak{u}^*_{\rm ab}$ such that $\psi_1(u)=e^{i \langle\ell_1,u\rangle}$.

Now consider the space $C^\infty(U\backslash G,\psi)$ of smooth functions $W$ on $G$ satisfying the transformation formula $W(ug)=\psi(u)W(g)$ for all $g\in G$ and $u\in U$. Then $G$ acts on $C^\infty(U\backslash G,\psi)$ by right-translation. This is the Whittaker space associated to $\psi$; it is the induction to $G$ of the one-dimensional representation $\psi$ of $U$.

For $\nu\in i\mathfrak{a}^*$, one can define~\cite{Jacquet67, Shahidi80}  a non-zero linear form on $\pi_\nu$ by setting
\begin{equation*}
\mathbb{J}^\psi(f)=\int_U f(\mathsf{w}u)\overline{\psi(u)}du,
\end{equation*}
a conditionally convergent integral~\cite[\S15]{book:Wall2}. One readily verifies that for $u\in U$, $\mathbb{J}^\psi(\pi_\nu(u) f)=\psi(u)\mathbb{J}^\psi(f)$, so that $0\neq \mathbb{J}^\psi\in {\rm Hom}_U (\pi_\nu,\psi)$. It is known that $\dim_\C {\rm Hom}_U (\pi_\nu,\psi)=1$. Thus $\mathbb{J}^\psi$ is the unique non-zero element up to scaling. We can replace $f$ by its translate by a group element to form
$\mathbb{J}^\psi(\pi_\nu(g)f)$, which as a function on $G$ satisfies $\mathbb{J}^\psi(\pi_\nu(ug)f)=\psi(u)\mathbb{J}^\psi(\pi_\nu(g)f)$ for every $g\in G$ and $u\in U$. The assignment $f\mapsto\mathbb{J}^\psi(\pi_\nu(\cdot )f)$ is a non-zero intertwining from (the smooth subspace of) $\pi_\nu$ to $C^\infty(U\backslash G,\psi)$. We denote the image by $\mathcal{W}(\pi_\nu,\psi)$ and refer to it as the Whittaker model of $\pi_\nu$. 

Let $W_\nu^\psi$ denote the image of the spherical function $f_\nu\in \pi_\nu^K$ under this intertwining: $W_\nu^\psi(g)=\mathbb{J}^\psi(\pi_\nu(g)f_\nu)$ for $g\in G$. This is the Jacquet-Whittaker function, given explicitly by
\begin{equation}\label{Jacquet-int}
W_\nu^\psi (g)=\int_U \delta(\mathsf{w}ug)^{1/2} e^{i B(H_\nu, H(\mathsf{w}ug))}\overline{\psi(u)}du.
\end{equation}
Clearly $W_\nu^\psi$ lies in $\mathcal{W}(\pi_\nu,\psi)^K$, the one-dimensional space of 
$K$-fixed vectors in $\mathcal{W}(\pi_\nu,\psi)$. When $\psi=\psi_1$ we simplify the 
notation and write $W_\nu$ in place of $W_\nu^{\psi_1}$. From the above integral we may 
extract the oscillatory dependence via
\begin{equation}\label{def:F}
F_\nu(u,g)=B(H_\nu,H(\mathsf{w}ug))-\langle\ell_1,u\rangle,
\end{equation}
the {\it Whittaker phase function}. By the right $K$-invariance in the second variable we 
often view $F_\nu$ as a function on $U\times S$, and write $F_\nu(u,x)$ for $x=gK$. 

Denote by $\Sigma_\nu$ the fiber critical set  of $F_\nu$ with respect to the natural projection $U\times S\rightarrow S$; thus
\[
\Sigma_\nu=\{ (u,x)\in U\times S: d_uF_\nu (u,x)=0\}.
\]
There is an associated fiber preserving map~\cite{Duistermaat:oscillatory,Hormander:FIO-I}
\begin{equation}\label{immersion}
\Sigma_\nu\rightarrow T^*(S),\qquad (u,x)\mapsto (x, d_xF_\nu(u,x)),
\end{equation}
into the cotangent bundle $T^*(S)\rightarrow S$ of $S$, whose image we denote by $\Lambda_\nu$.

If $\nu$ is regular then $F_\nu$ is a non-degenerate phase function \cite[Theorem 
6.7.1]{Kost79:Toda}, which implies that $\Sigma_\nu$ is a smooth manifold of dimension 
$\dim S$ and $\Lambda_\nu$ is a Lagrangian submanifold of $T^*(S)$. In particular, 
$\Lambda_\nu\rightarrow S$ is a Lagrangian mapping. 

\section{Reduction to local estimates}\label{sec:intro:gl2} The purpose of this section is to reduce the proof of Theorem \ref{th:gln} to Theorem \ref{general tau}, and of Proposition~\ref{ess supp} to Theorem \ref{whitt supp}.

\subsection{Reduction of Theorem \ref{th:gln} to Theorem \ref{general 
tau}}\label{sec:intro:outline} We follow the method of~\cite{Temp:p-adic}.
Let $G=\PGL_n(\R)$ and $K=\PO(n)$, and let $f$ be a Hecke--Maass form on $\Gamma \backslash S_n$ with eigenvalue $\lambda  >  0$. 

 We view $f$ as a right $K$-invariant function in $L^2(\Gamma \backslash G)$. Choose a non-degenerate character $\psi$ of $U$, trivial on $\Gamma_U=\Gamma\cap U$, and consider the Whittaker integral
\[
W_f(g)=
\int_{\Gamma_U\backslash U} 
f(ug) 
\overline{\psi(u)} du,
\quad g\in G.
\]
Since $\Gamma_U\backslash U$ is compact, we deduce that $\pnorm{f}_\infty \ge {\rm vol}(\Gamma_U\backslash U)^{-1}\pnorm{W_f}_\infty$.

From the Hecke assumption on $f$, we know that $W_f$ is a non-zero vector belonging to the one-dimensional space $\mathcal{W}(\pi_\nu,\psi)^K$ of $K$-fixed vectors in the local Whittaker model  of $\pi_\nu$.  A result of Baruch~\cite[Corollary 10.4]{Baruch}, extending to the archimedean case the analogous result of Bernstein over non-archimedean local fields, shows that there is a unique (up to scaling) $G$-invariant inner product on $\mathcal{W}(\pi_\nu,\psi)$ given by
\begin{equation*}
\int_{U\backslash P_n}W_1(p)\overline{W_2}(p)d\dot{p}.
\end{equation*}
Here $P_n$ is the mirabolic subgroup of $G$ consisting of (homothety classes of) matrices 
with $(0,\dots,1)$ in the bottom row. Note that $P_n=G_{n-1}\ltimes U/U_{n-1}$, where 
$G_{n-1}=\PGL_{n-1}(\R)$ and $U_{n-1}$ is the unipotent radical of the standard Borel in 
$G_{n-1}$. The right-invariant Haar measure $d\dot{p}$ on $U\backslash P_n$ is given by 
choosing Haar measures on the unimodular groups $G_{n-1}$, $U_{n-1}$, and $U$. 

The unfolding of the Rankin-Selberg integral implies~\cite{FLO} that
\[
\pnorm{f}_2^2
=
c\, \Run L(s,\pi\times \widetilde\pi)
\pnorm{W_f}_2^2,
\]
where $\pi$ is the cuspidal automorphic representation generated by the Hecke-Maass form $f$. Here $c>0$ is a constant depending only on the volume
normalization. Moreover, by Li~\cite{Li09} (see also \cite{Brum06, Molt02, RW}), we have 
\[
\Run L(s,\pi\times \widetilde \pi)\ll_\varepsilon \lambda^\varepsilon,
\quad\text{for all $\epsilon>0$.}
\]
The implicit constant depends on $\Gamma$, but since we view the space $\Gamma\backslash S_n$ as being fixed, we will always drop the dependence on $\Gamma$. From this we deduce the lower bound
\[
\pnorm{W_f}_\infty\gg_\varepsilon \lambda^{-\varepsilon}\pnorm{W_f}_\infty/\pnorm{W_f}_2.
\]

From its scale invariance and the multiplicity one of spherical Whittaker functions, this last quotient is unchanged under the substitution of the
global Whittaker period $W_f$ by any other non-zero vector $W\in \mathcal{W}(\pi_\nu,\psi)^K$, where $\nu\in i\mathfrak{a}^*$ is the spectral
parameter. Taking $W=W_\nu$, this yields
\[
\pnorm{f}_\infty \gg_\varepsilon \lambda^{-\varepsilon}\pnorm{W_\nu}_\infty/\pnorm{W_\nu}_2 .
\] 
Now Theorem~\ref{general tau} says that $\pnorm{W_\nu}_\infty\gg 
\lambda^{\frac{c(G)}{2}}\pnorm{W_\nu}_2$ for $\nu \in \sqrt{\lambda}\Omega$.
This completes the reduction of Theorem \ref{th:gln} to Theorem \ref{general tau}.\qed

\begin{remark}
In~\cite{GeL06}, Gelbart, Lapid, and Sarnak establish a lower bound on Langlands-Shahidi $L$-functions $L(1+it,f,r)$ for generic automorphic cusp forms $f$ and $|t|\to \infty$. Their method, like that of this paper, relies on lower bounds for Whittaker functions. To compare,
\begin{itemize}
\item[--] in this paper, a lower bound for $\pnorm{W}_\infty$ and the convexity upper bound for $L(1,f\times \tilde f)$ together imply a lower bound for $\pnorm{f}_\infty$; 
\medskip
\item[--] in~\cite{GeL06}, a lower bound for Whittaker functions [{\it loc. cit}, Lem.\,7] and an upper bound for $\pnorm{\Lambda^T E(\frac12+it,f)}_2$ [{\it loc. cit.}, Prop.\,2] together imply a lower bound for $L(1+it,f,r)$. 
\end{itemize}
\end{remark}

\subsection{Reduction of Proposition~\ref{ess supp} to Proposition \ref{Lapid}}\label{sec:ess sup}

For simplicity we assume that $\Gamma=\PGL_3(\Z)$; the general case for arbitrary congruence $\Gamma$ is similar. We use the following subgroup notation: ${\rm B}_2$ is the standard Borel subgroup of $\GL_2$, ${\rm U}_2$ its unipotent radical, and ${\rm A}_2$ the group of diagonal matrices. We view $\GL_2$ as embedded in $G=\PGL_3(\R)$ via $g\mapsto\left(\begin{smallmatrix}g & \\ & 1\end{smallmatrix}\right)$. The analogous subgroups $U$, $A$, and $K$ of $G$ have the same meaning as in the introduction.

\medskip

\noindent {\sc -- Fourier expansion:} The Fourier--Whittaker expansion of the $L^2$-normalized cusp form $f$ at the unique cusp for $\Gamma=\PGL_3(\Z)$ is given by
\begin{equation}\label{FW-expansion}
f(g)=\sum_m\sum_{[\gamma]}  \rho_f(m)W_\nu\left(d_m\left(\begin{smallmatrix}\gamma & \\ & 1\end{smallmatrix}\right) g\right),
\end{equation}
where $m=(m_1,m_2)$ ranges over all vectors in $\Z_{\neq 0}^2$, $d_m={\rm diag}(m_1 m_2,m_2, 1)$, and $[\gamma]$ ranges over cosets ${\rm U}_2(\Z)\backslash \GL_2(\Z)$. The coefficients $\rho_f(m)$ are certain complex numbers satisfying $\rho_f(1,1)\neq 0$. They grow at most polynomially in $t \max(|m_1|,|m_2|)$, a fact established in \cite{BrumAJM}. (Recall that the spectral parameter $\nu$ of $f$ is in $t\Omega\subset i\mathfrak{a}_{\rm reg}^*$ with $t>1$).

\medskip

\noindent {\sc -- Staying in the cusp:} Writing $g\in\PGL_3(\R)$ in its Iwasawa 
decomposition $g=uak$, we can clearly assume that $k=e$. The hypothesis of the first 
statement in Proposition~\ref{ess supp} is that $a={\rm diag}(y_1 y_2, y_2,1)$ satisfies 
$\min(y_1,y_2)\geq\sqrt{3}/2$ and $\max(y_1,y_2)\gg t$ for a large parameter $t$. 
Theorem~\ref{whitt supp} then states that such $g$ lie in the rapid decay regime for 
$W_{t\nu}$. We would like to say that this is equally true for every translate $d_m 
\left(\begin{smallmatrix}\gamma & \\ & 1\end{smallmatrix}\right)g$ appearing in the 
Fourier-Whittaker expansion above. Now since $d_m$ normalizes $U$, the $A$-part of $d_m 
\left(\begin{smallmatrix}\gamma & \\ & 1\end{smallmatrix}\right)g$ in the Iwasawa 
 decomposition $\PGL_3(\R)=UAK$ is equal to $d_m$ times the $A$-part of 
$\left(\begin{smallmatrix}\gamma & \\ & 1\end{smallmatrix}\right)g$. For the latter 
matrix, we have the following lower bound on the maximum of the roots.

\begin{lem}\label{lem:reduction}
Let $g\in\PGL_3(\R)$ be as above and let $\gamma\in \GL_2(\Z)$. Let $a'={\rm diag}(y_1' y_2',y_2',1)$ be the Iwasawa $A$-part of $\left(\begin{smallmatrix}\gamma & \\ & 1\end{smallmatrix}\right)g$. Then
\[
\max (y_1',y_2') \gg \max (y_1,y_2).
\]
\end{lem}
\begin{proof}
If $\gamma\in {\rm B}_2(\Z)$ then there is $k'={\rm diag}(\pm 1,\pm 1)$ such that $\gamma k'\in {\rm U}_2(\Z)$. We see then that $\left(\begin{smallmatrix}\gamma & \\ & 1\end{smallmatrix}\right)g\in UaK$. Thus in this case we in fact have $y_i'=y_i$.

If $\gamma\notin {\rm B}_2(\Z)$ then we use the Bruhat decomposition of $\GL_2(\Z)$ to write $\gamma$ as $bw_0u'$, for some $b\in {\rm B}_2(\Q)$ and $u'\in {\rm U}_2(\Q)$, where $w_0:=\left(\begin{smallmatrix} 0 & 1 \\ 1& 0 \end{smallmatrix}\right)$. Since ${\rm B}_2(\Q)={\rm U}_2(\Q){\rm A}_2(\Q)$ we can assume that $b\in {\rm A}_2(\Q)$, say $b={\rm diag}(\pm 1/q,q)$ for $q\in\Q^\times$. Since $\gamma$ has integer entries one in fact has $q\in\Z-\{0\}$.

Writing $w:=\left(\begin{smallmatrix} w_0 &  \\ & 1 \end{smallmatrix}\right)=\left(\begin{smallmatrix} 0 & 1 & 0\\ 1& 0 & 0 \\
0 & 0 & 1 \end{smallmatrix}\right)$, we have
\[
\begin{pmatrix}\gamma & \\ & 1\end{pmatrix}g=\begin{pmatrix}b & \\ & 1\end{pmatrix}w\begin{pmatrix}u' & \\ & 1\end{pmatrix}ua=\begin{pmatrix}b & \\ & 1\end{pmatrix}\cdot waw\cdot wv,
\]
where $v=a^{-1}\left(\begin{smallmatrix}u' & \\ & 1\end{smallmatrix}\right)ua\in U$. The roots ${}^wy_i$ of $waw$ are ${}^wy_1=y_1^{-1}$ and ${}^wy_2=y_1y_2$. Moreover, if $v=\left(\begin{smallmatrix} 1 & x & *\\ 0 & 1 &*\\ 0 & 0 & 1\end{smallmatrix}\right)$ then $wv$ has Iwasawa $A$-part ${\rm diag}(1/\sqrt{1+x^2},\sqrt{1+x^2},1)$. Thus $y_1'=q^{-2}\, y_1^{-1} (1+x^2)^{-1}$ and $y_2'=|q| y_1y_2 \sqrt{1+x^2}$. The first root $y_1'$ can be very small, but since $|q|\geq 1$, $\sqrt{1+x^2}\geq 1$, and $\min(y_1,y_2)\gg 1$ we have $y_2'\gg\max(y_1,y_2)$ as desired.
\end{proof}

\medskip
\noindent {\sc -- Conclusion of the proof:}
We continue with the reduction of the first statement of Proposition~\ref{ess supp} to 
Proposition~\ref{Lapid}. Recall that $\nu\in t\Omega\subset i\mathfrak{a}_{\rm reg}^*$. 
By hypothesis, we have $\max (y_1,y_2)\gg t^{1+\epsilon}$, for some $\epsilon>0$.

We return to the Fourier--Whittaker expansion \eqref{FW-expansion} of $f$. Note that the 
maximum of the roots of $d_m\left(\begin{smallmatrix}\gamma & \\ & 
1\end{smallmatrix}\right) g$ is equal to $\max(m_1y'_1,m_2y'_2)$, and that, since $g$ 
is assumed to lie in a Siegel domain, we have $y_1,y_2\gg 1$. We may therefore 
apply Proposition~\ref{Lapid} (as well as the aforementioned bound on $\rho_f(m)$) to every term in \eqref{FW-expansion}. In this way we obtain
\[
f(g)\ll_N t^{N+O(1)} \sum_m  (m_1m_2)^{O(1)} \sum_{[\gamma]}  \max(m_1y'_1,m_2y'_2)^{-N+O(1)}.
\]
We shall prove that the double sum is at most $\max (y_1,y_2)^{-N+O(1)}$. In this way, we will have established that
\begin{equation}\label{decay-claim}
f(g)\ll_N \big(t\max (y_1,y_2)\big)^{O(1)} \big(\max (y_1, y_2)/t\big)^{-N},
\end{equation}
proving rapid decay in the region $\max(y_1,y_2)\gg t^{1+\epsilon}$, as desired.

We begin by treating, for every fixed $[\gamma]$, the sum over $m$. We claim that, for $N$ large enough,
\[
\sum_m (m_1m_2)^{O(1)} \max(m_1y'_1,m_2y'_2)^{-N+O(1)}\ll  \max(y'_1,y_2')^{-N+O(1)}.
\]
To see this, assume that $y_1'\geq y_2'$. Breaking up into ranges yields
\[
{y'_2}^{-N+O(1)}\sum_{m_1 y_1'\leq m_2y_2'} m_1^{O(1)}m_2^{-N+O(1)}+{y'_1}^{-N+O(1)}\sum_{m_1y_1'>m_2y_2'} m_2^{O(1)} m_1^{-N+O(1)}.
\]
The second sum can be estimated elementarily as
\[
{y'_1}^{-N+O(1)}\sum_{m_1\geq 1}m_1^{-N+O(1)}\sum_{m_2< m_1y_1'/y_2'}m_2^{O(1)}= {y'_1}^{-N+O(1)}(y_1'/y_2')^{O(1)}\sum_{m_1\geq 1}m_1^{-N+O(1)},
\]
which, using $y_2'\gg1$, is at most ${y'_1}^{-N+O(1)}$. We treat the first sum by breaking it up into dyadic pieces, as
\[
{y'_2}^{-N+O(1)}\sum_{m_1\geq 1}m_1^{O(1)}\sum_{M\geq 0} \sum_{m_2\sim 2^M m_1y_1'/y_2'} m_2^{-N+O(1)}. 
\]
This is then bounded by
\[
{y_1'}^{-N+O(1)}(y_1'/y_2')^{O(1)}\sum_{m_1\geq 1}m_1^{-N+O(1)}\sum_{M\geq 0}2^{M(-N+O(1))}\ll {y_1'}^{-N+O(1)}. 
\]
A similar argument applies when $y_2'\geq y'_1$. This proves the claim, and we get
\[
f(g)\ll_N t^N \sum_{[\gamma]}  \max(y'_1,y_2')^{-N+O(1)}.
\]

We distinguish two ranges, according to the size of the roots of $\left(\begin{smallmatrix}\gamma & \\ & 1\end{smallmatrix}\right)$ in the $A$-part of its Iwasawa decomposition. We label these roots as $y_1(\gamma), y_2(\gamma)$, and note that, unlike $y_1',y_2'$, they do not depend on $g$. In the first range we consider the set of $[\gamma]$ such that $\max (y_1(\gamma), y_2(\gamma)) \le \max(y_1,y_2)^M$, for a real $M>0$ to be chosen below. This is a finite set of cardinality at most $\max(y_1,y_2)^{c_1M}$, for some constant $c_1$. Now $\max(y'_1,y'_2) \gg \max(y_1,y_2)$, uniformly for all $\gamma$, by Lemma \ref{lem:reduction}. Thus the sum over all elements $\gamma$ in this range is less than $\max(y_1,y_2)^{c_1M - N+O(1)}$.

In the second range we consider the tail of the sum, consisting of elements $[\gamma]$ 
such that $\max (y_1(\gamma), y_2(\gamma))\ge \max(y_1,y_2)^M$. We apply an 
Iwasawa decomposition of the element $\left(\begin{smallmatrix}\gamma & \\ & 
1\end{smallmatrix}\right) g$ and find that its $A$-part is the product of the $A$-part of 
$\left(\begin{smallmatrix}\gamma & \\ & 1\end{smallmatrix}\right)$ times the $A$-part of 
$kg$ for some element $k\in K$ (specifically $k$ is the $K$-part of 
$\left(\begin{smallmatrix}\gamma & \\ & 1\end{smallmatrix}\right)$). The roots of the 
$A$-part of $kg$ are greater than $\max(y_1, y_2)^{-c_2}$, for some constant $c_2$, 
independent of $\gamma$. The contribution of this second range is thus bounded by
\[
\max(y_1, y_2)^{c_2N}\sum_{\substack{\max\{y_1(\gamma), y_2(\gamma)\}\\ \ge \max(y_1,y_2)^M}} \max (y_1(\gamma), y_2(\gamma))^{-N+O(1)}.
\]
Breaking up this last sum into dyadic intervals, we obtain
\[
\sum_{2^n\ge \max(y_1,y_2)^M} \sum_{\max (y_1(\gamma), y_2(\gamma))\sim 2^n}\max (y_1(\gamma), y_2(\gamma))^{-N+O(1)}\ll \sum_{2^n\ge \max(y_1,y_2)^M} 2^{n c_1}2^{n(-N+O(1))},
\]
where $2^{nc_1}$ is a bound on the total number of $[\gamma]$ in the dyadic interval. As long as $N>c_1+O(1)$, this geometric series converges. The total contribution from this range is then
\[
 \max(y_1, y_2)^{c_2N+M(c_1-N+O(1))}.
\]
By choosing $M$ large enough so that the exponent is negative, this is of the desired form.

We have obtained a rapid decay bound for both ranges, completing the proof of \eqref{decay-claim}.

\bigskip

\noindent {\sc -- Proof of \eqref{bad-upper}:} 
The proof of \eqref{bad-upper} is based on an explication of the dependence of the implied constant on the injectivity radius of $\Gamma\backslash S$ in Sarnak's bound~\eqref{d-r-omega}.

For the next few paragraphs, we take $G$ to be a connected split semi-simple Lie group without compact factors, and $K$ a maximal compact subgroup of $G$. Then the associated globally Riemannian symmetric space $S=G/K$ is of non-compact type. Let $\Gamma$ be an arithmetic non-uniform lattice in $G$, so that the quotient $\Gamma\backslash S$ is non-compact. (We will specialize to $\Gamma\backslash S_3$ momentarily.) For $p\in\Gamma\backslash S$, and $R>0$ smaller than the local injectivity radius about $p$, let $B(p,R)$ denote the geodesic ball of radius $R$. When $(\Delta+\lambda)f=0$, a direct inspection of the proof of Sarnak~\cite{SaMo} yields
\begin{equation}\label{uniform Sarnak}
|f(p)|\leq C \bigg(\int_{B(p,R)}|\omega_\lambda(x)|^2dx\bigg)^{-1/2}\bigg(\int_{B(p,R)}|f(x)|^2dx\bigg)^{1/2},
\end{equation}
where  $\omega_\lambda$ is the unique spherical function on $G$ about $p$ having the same $\mathscr{D}_G(S)$-eigenvalues as $f$ (and thus of eigenvalue $\lambda$) and normalized so that $\omega_\lambda(p)=1$. Going high in the cusp, we can find $p\in\Gamma\backslash S$ with arbitrarily small injectivity radius; in particular we can take $0<R<1/\sqrt{\lambda}$. On such balls, the spherical function $\omega_\lambda$ is $\asymp 1$ and one has
\begin{equation}\label{eq:non-osc}
\int_{B(p,R)}|\omega_\lambda(x)|^2dx\asymp {\rm vol}(B(p,R))\asymp R^d \qquad (0<R<1/\sqrt{\lambda}),
\end{equation}
where $d$ is the dimension of $S$.

We shall now describe a truncation of $\Gamma\backslash S$ up to height $T$. Recall 
the notation from \S\ref{sec:root-notation}. We assume that $B$ is defined over $\Q$. A 
Siegel set with respect to $B$ is a subset of $G$ of the form $\mathfrak{S}=\omega A_c 
K$, where $\omega$ is a relatively compact subset of $U$ and $A_c=\{a\in A: 
\alpha(a)>\frac1c\; 
\forall \alpha\in\Pi\}$. By reduction theory, there is $c>1$ and a finite 
subset $\Xi$ of $G(\Q)$ such that  $G=\Gamma\Xi\mathfrak{S}$. For $T> c$, let 
$A_c^T=\{a\in A: \frac1c<\alpha(a)<T\;\forall\alpha\in\Pi\}$ and $\mathfrak{S}^T=\omega 
A_c^TK$. Then we denote by $(\Gamma\backslash S)^{\leq T}$ the image of 
$\Gamma\Xi\mathfrak{S}^T$ in $\Gamma\backslash S$.

We claim that the injectivity radius on $(\Gamma\backslash S)^{\leq T}$ is at least $1/T^r$, where $r=\#\Pi$ is the rank of $G$. To see this, let $p$ lie in the truncation $(\Gamma\backslash S)^{\leq T}$, so that $p$ is the image of $xuak\in G$ under the natural projection map, where $x\in\Xi$, $u\in \omega$, $a\in A_c^T$, $k\in K$. Suppose that there is $g\in G$ with ${\rm dist}(e,g) \ll 1/T^r$ and $\gamma \in \Gamma$ such that $pg=\gamma p$. Our goal is to prove that $\gamma =e$. We write the equality $pg=\gamma p$ as $a.kgk^{-1}.a^{-1}=(xu)^{-1}\gamma xu$ and observe that ${\rm dist}(e,kgk^{-1}) \ll 1/T^r$ since $k\in K$ varies in a compact. The conjugation by $a$ is described by its roots $\alpha(a)$ for $\alpha\in\Pi$; since $a\in A_c^T$, the largest dilation is $T^r$. Since $xu$ also varies in a compact this implies ${\rm dist}(e,\gamma) \ll 1$. Thus if the constant is chosen small enough, $\gamma =e$ as desired.

We may therefore bound the value at any point $p\in (\Gamma\backslash S)^{\leq T}$ by its 
$L^2$-norm over the geodesic ball of radius $1/T^r$ about $p$. In particular, it follows 
from \eqref{uniform Sarnak} and \eqref{eq:non-osc} that for any $p\in (\Gamma\backslash 
S)^{\leq T}$, with $T=\lambda^{1/2+\epsilon}$, we have 
\begin{equation*}
|f(p)| \ll_\epsilon \lambda^{dr/4+\epsilon}\pnorm{f}_2.
\end{equation*}

We now specialize to the case $S=S_3$, $\Gamma$ a congruence subgroup of $\GL_3(\Z)$ and $f$ a Hecke-Maass cusp form. Since by the first half of Proposition \ref{ess supp} the size of $f$ on the complement of $\Gamma\backslash S_3^{\leq T}$, where $T=\lambda^{1/2+\epsilon}$, is smaller than any power of $\lambda$, the bound \eqref{bad-upper} is proved, where we use $d=\dim S_3=5$ and $r=2$.

\begin{remark}\label{rem:upper3}
	 It would be interesting to investigate more the essential support of cusp forms in higher rank, and for general groups.
Although Theorem \ref{whitt supp} is valid for arbitrary $n$, it requires additional work to extend Proposition \ref{ess supp} to all $n$.
 The Fourier expansion of $f$ on $\Gamma\backslash S_n$ still holds, but Lemma \ref{lem:reduction} is not true for $n\geq 4$. This can be seen in the following example.
Let $\gamma=\left( 
\begin{smallmatrix}
0 & 1 & 0 \\
1 & N & -1 \\
0 & N & -1 
\end{smallmatrix}
\right)
$, and $g=a\in A$ with roots $y_i$. It can be verified that 
\[
y'_1 = y_1^{-1} (1+N^2y_2^2)^{-1/2},\quad
y'_2 = y_1(N^2+y_2^{-2})^{-1/2},\quad
y'_3 = y_3(1+N^2y_2^2)^{1/2}.
\]
Letting $y_2=y_3=1$ and $y_1$ large of size about $N^2$ we see that Lemma~\ref{lem:reduction} is not valid in this case.
\end{remark}

\section{Rapid decay estimates} \label{sec:main-proof}
In this section we establish several estimates for Whittaker functions with large eigenvalue. In the first two subsections, we give quantitative information on the rapid decay regime of spherical Whittaker functions in the general setting of split semisimple real Lie groups. In the third subsection, we use these results to prove Theorem \ref{general tau} in the case of $\GL_n(\R)$.

\subsection{Rapid decay} Let $W_\nu$ be a spherical Whittaker function on a split 
semisimple real Lie group. The following proposition gives the rapid decay of $W_\nu(a)$ 
for $a$ large with respect to $\nu$. The proof is through repeated integration by parts 
and a convolution identity. This kind of argument is relatively standard, e.g., in 
estimates of Eisenstein series (see \cite[\S4]{Arthur78}). 

\begin{prop}\label{Lapid}
For every $\nu \in i \mathfrak{a}^*$, $N\ge 0$, and $a\in A$,
\begin{equation*}
|W_{\nu}(a)|\ll_{N,\epsilon}  \delta(a)^{\frac12} \|a\|^\epsilon 
\|\nu\|^{O(1)}(\max\limits_{\alpha\in\Pi} 
\alpha(a)/\|\nu\|)^{-N}.
\end{equation*}
\end{prop}

\begin{proof} 
For any $\varphi\in C_c^\infty(K \backslash G /K)$ we have 
$\widehat\varphi(\nu)W_\nu=W_\nu\star\varphi$, where $\varphi\mapsto\widehat\varphi(\nu)$ 
is the spherical transform. We decompose the following integration using the Iwasawa 
coordinates to get
\begin{align*}
(W_\nu\star\varphi)(a)&=\int_GW_\nu(ag)\varphi(g)dg=\int_A\int_U W_\nu(aua_1)\varphi(ua_1)\delta(a_1)^{-1}du\, da_1\\
&=\int_A\int_U \psi_a(u)W_\nu(aa_1)\varphi(ua_1)\delta(a_1)^{-1}du\, da_1,
\end{align*}
where $\psi_a(u)=\psi(aua^{-1})$.
Fix $c\in 
\mathbb{R}_{\ge 1}$, and let $A_0=A_c^c \subset A$ be defined by 
inequalities $c^{-1} \leq \alpha(a)\leq c$ for $\alpha\in\Pi$.
We choose $\varphi$ of sufficient small compact support, such that $\varphi(ua_1)$ 
vanishes for all $u\in U$ and $a_1\in A - A_0$. Then
\begin{equation}\label{W-convo}
(W_\nu\star\varphi )(a)= \int_{A_0} W_\nu(aa_1)\delta(a_1)^{-1} \int_U \psi_a(u) 
\varphi(ua_1)du\, da_1.
\end{equation}

For $G=\PGL_n(\R)$, an application of the Cauchy-Schwarz inequality in \eqref{W-convo} yields
\[
|\widehat 
\varphi(\nu)|^2|W_\nu(a)|^2\le\int_{A_0}\abs{W_\nu(aa_1)}^2\delta(a_1)^{\epsilon-1}da_1
\cdot
\int_{A_0}\left|\int_U\psi_a(u)\varphi(ua_1) du\right|^2\delta(a_1)^{-1-\epsilon}da_1.
\]
We change variables to find that the first integral is
\[
\delta(a)^{1-\epsilon}\int_{aA_0}\abs{W_\nu(a_1)}^2\delta(a_1)^{\epsilon-1}da_1
=
\delta(a)^{1-\epsilon}
\int_{aA_0}\abs{W_\nu(
\begin{pmatrix}
a_1 & 0 \\
0 & 1
\end{pmatrix}
)}^2
\delta(a_1)^{\epsilon-1} \det(a_1)^{\epsilon-1} da_1.
\]
We write $\delta(a_1)=||a||^{O(1)}$, so that the integral is
\[
\leq \delta(a)^{1-\epsilon} ||a||^{O(\epsilon)}
\int_{A}\abs{W_\nu(
\begin{pmatrix}
a_1 & 0 \\
0 & 1
\end{pmatrix}
)}^2
\delta(a_1)^{-1} \det(a_1)^{\epsilon-1} da_1
\leq
\delta(a) \|a\|^{O(\epsilon)} 
\|\nu\|^{O(1)},
\]
by the Stade formula~\eqref{clean stade} applied to $s=\epsilon$. Hence,
\begin{equation}\label{GLn-case}
|\widehat \varphi(\nu)|^2|W_\nu(a)|^2\ll
\delta(a)\|a\|^{O(\epsilon)}
\|\nu\|^{O(1)}
\int_{A_0}\left|\int_U\psi_a(u)\varphi(ua_1) du\right|^2\delta(a_1)^{-1}da_1.
\end{equation}

For general $G$, we apply absolute values to \eqref{W-convo} and treat the integral over 
$A_0$ by bounding $W_\nu(aa_1)$ pointwise as follows. Using the integral 
representation 
\eqref{Jacquet-int} we have
\[
|W_{\nu+\epsilon 
\rho}(a)|\leq \int_U  \delta(\mathsf{w}ua)^{\frac12+\epsilon} du
=
\delta(a)^{\frac12-\epsilon}\int_U \delta(\mathsf{w}u)^{\frac12+\epsilon} du.
\]
The Gindinkin-Karpelevi\u{c} formula implies the convergence of the above integral 
(see~\cite[Thm.~2.8]{Jacquet67} 
or~\cite[(3.57)]{DKV:spectra}), from which one deduces $W_{\nu+\epsilon 
\rho}(a)\ll_\epsilon
\delta(a)^{\frac12-\epsilon}$. The functional 
equation~\cite[Prop.~3.3 and 
(4.2.3)]{Jacquet67}, combined with the Phragm\'en-Lindel\"of maximum principle imply that 
$|W_{\nu}(a)|\ll_\epsilon 
\delta(a)^{\frac12} \|a\|^\epsilon \|\nu\|^{\epsilon}$. To apply the 
Phragm\'en-Lindel\"of 
principle, we need that 
$W_\nu$ be of finite order, which is established in~\cite{McKee:order-one}. Inserting 
this bound, we 
obtain
\begin{equation}\label{general-G-case}
|\hat\varphi(\nu)W_\nu(a)|\ll \delta(a)^{\frac12} \|a\|^\epsilon \|\nu\|^{\epsilon} 
\int_{A_0} 
\bigg|\int_U \psi_a(u) 
\varphi(ua_1)du\bigg|\, da_1.
\end{equation}

In either \eqref{GLn-case} or \eqref{general-G-case}, we must still bound $\int_U\psi_a(u)\varphi(ua_1) du$, which we do for a specific choice of the function $\varphi$. For any 
$\varphi_0\in C_c^\infty(A_{\mathrm{reg}})$, put
\[
\varphi_\nu(k_1ak_2)= \varphi_0(a)\sum_{w\in W} \nu(waw^{-1}) , \quad k_1,k_2\in K,\ 
a\in 
A.
\]
From~\cite[Lem.~6.3 and Prop.~6.9]{DKV:spectra}, there exists $\varphi_0\in 
C_c^\infty(A_{\mathrm{reg}})$ of arbitrarily small support such that $|\widehat 
\varphi_\nu(\lambda)|\gg \|\nu\|^{-O(1)}$ for all
$\lambda, \nu\in i\mathfrak{a}^*$ such that $\pnorm{\lambda-\nu}\leq 1$. In particular 
$|\widehat \varphi_\nu(\nu)|\gg \|\nu\|^{-O(1)}$. We claim that for 
any $a\in A$ and any
integer $N\ge 1$,
\[
\int_U \psi_a(u) \varphi_\nu(ua_1) du \ll_N \pnorm{\nu}^N (\max\limits_{\alpha\in\Pi} 
\alpha(a))^{-N},
\]
where the implied constant depends only on $N$ and the choice of $\varphi$. This will 
finish the proof.

For any $\alpha\in\Pi$ and $0\neq X_\alpha\in\mathfrak{u}_\alpha$ then the first derivative of the additive character is
\[
X_\alpha\cdot \psi_a(u)= i\alpha(a) \psi_a(u).
\]
For $a\in A$ let $\alpha_{\max}$ be such that $\alpha_{\max} (a)=\max\limits_{\alpha\in\Pi} \alpha(a)$ and let $X_{\max}$ be an element of unit norm in the root space $\mathfrak{u}_{\alpha_{\max}}$. Integrating by parts $N$ times the integral is, up to a sign, equal to
\begin{equation}\label{IPP-convo}
\alpha_{\max}(a)^{-N} \int_U \psi_a(u) \varphi_\nu(X_{\max}^N;ua_1) du,
\end{equation}
where $X^N_{\max}$ is viewed as an element in $\mathscr{U}(\mathfrak{g})$. Now using the definition of $\varphi_\nu$ we have $\varphi_\nu(X_{\max}^N;ua_1) \ll
\pnorm{\nu}^N $ where the implied constant depends only on $\varphi$.
\end{proof}

\subsection{Precise decay regime}\label{s:decay-regime}
In this paragraph we give an alternative description of the rapid decay regime of the Whittaker function relative to Proposition \ref{Lapid}. We are again assuming here that $G$ is an arbitrary split semisimple real Lie group. 

The idea here is standard: the Whittaker function is given as an oscillatory integral, 
and where there are no critical points one has rapid decay (again by integration by 
parts). To make the link with later sections, we express the rapid decay regime in terms 
of the fibers of an associated Lagrangian mapping $\Lambda_\nu\rightarrow S$ introduced 
in \S\ref{Whitt-structures}. 

We begin with the following lemma which establishes~\eqref{W-form2}. Recall our 
conventions that
\[
\nu(a) = e^{\langle \nu,\log a \rangle}\quad\text{and}\quad\ell_a(u)=\Ad^*_a\ell_1(u)=\ell_1(aua^{-1}).
\]

\begin{lem}\label{c-o-v}
For any $\nu\in\mathfrak{a}_\C^*$ and $a\in A$ we have
\[
W_\nu(a)=\delta(a)^{1/2}\nu(\mathsf{w}a\mathsf{w})\int_U 
\delta(\mathsf{w}u)^{1/2}e^{i(B(H_\nu,H(\mathsf{w}u))-\ell_a(u))}du.
\]
\end{lem}
\begin{proof}
We have
\[
H(\mathsf{w}ua)=H(\mathsf{w}a\mathsf{w}\cdot \mathsf{w}\cdot 
a^{-1}ua)=H(\mathsf{w}a\mathsf{w})+H(\mathsf{w} a^{-1}ua).
\]
Since $\delta(\mathsf{w}a\mathsf{w})=\delta(a)^{-1}$, it follows that 
$\delta(\mathsf{w}ua)=\delta(a)^{-1}\delta(\mathsf{w} 
a^{-1}ua)$. In view of~\eqref{Jacquet-int}, we deduce
\[
W_\nu(a)=\delta(a)^{-1/2}\nu(\mathsf{w}a\mathsf{w})
\int_U \delta(\mathsf{w} 
a^{-1}ua)^{1/2}e^{i(B(H_\nu,H(\mathsf{w}a^{-1}ua))-\langle\ell_1,u\rangle)}du.
\]
Applying the automorphism $u\mapsto a^{-1}ua$, whose Jacobian is $\delta(a)$, we deduce the result.
\end{proof}

We now turn toward the estimation of the integral in Lemma \ref{c-o-v} by means of an 
integration by parts. Contrary to the integration in \eqref{IPP-convo}, the $U$-integral 
in 
Lemma \ref{c-o-v} does not involve a compactly supported amplitude function. To deal with 
the subtle issues of convergence inherent in such a non-compact setting, we introduce the 
following hypothesis. Recall from \S\ref{s:spherical-rep} that for $t>0$ and $a\in A$ we 
denote by $ta$ the unique element in $A$ whose simple roots are those of $a$ scaled by 
$t$.

\begin{hyp}\label{hypothesis} 
For every $\nu\in i\mathfrak{a}^*$ and $a\in A$, there 
exists a smooth compactly supported function $\alpha\in \mathcal{C}_c^\infty(U)$ such 
that for every $N\ge 1$,
\[
\int_U \delta(\mathsf{w}u)^{1/2} (1-\alpha(u)) e^{it F_\nu(u,a)}  du \ll_{N,\nu,a} t^{-N},
\]
where the multiplicative constant depends continuously on $\nu,a$.
\end{hyp}

It would take us too far afield to verify this hypothesis in the present article; we hope to address this question in a subsequent work.

\begin{prop}\label{rapid-fiber}
Assume Hypothesis~\ref{hypothesis}.
Let $\nu\in i\mathfrak{a}^*$ be non-zero and $X\in\mathfrak{a}$ be such that $e^X\in S$ lies outside the image of $\Lambda_\nu\rightarrow S$. Then for $t>0$ large enough we have
\[
W_{t\nu}(te^X)\ll_{N,\nu,X}t^{-N}
\]
for every $N\geq 1$.
\end{prop}

\begin{proof}
We have $\ell_{te^X}(u)=t\ell_{e^X}(u)$ and $H_{t\nu}=tH_\nu$ so that
\[
B(H_{t\nu},H(\mathsf{w}u))-\ell_{te^X}(u)=t(B(H_\nu,H(\mathsf{w}u))-\ell_{e^X}(u))=tF(u,e^X).
\]
From the preceding lemma, and the scaling relation \eqref{a-scaling}, we have
\[
|W_{t\nu}(te^X)|=t^{{\rm ht}(G)/2}\delta(e^X)^{1/2}\bigg|\int_U \delta(\mathsf{w}u)^{1/2}e^{itF_\nu(u,a)}du\bigg|.
\]
Since $e^X$ lies outside the image of $\Lambda_\nu$, the phase function $u\mapsto F_\nu(u,e^X)$ has no critical points. 
Rapid decay follows from Hypothesis~\ref{hypothesis} and repeated integration by parts. 
\end{proof}

We shall see in \S\ref{associated-asymp} that (under a regularity assumption on $\nu$) 
for every $X\in \mathfrak{a}$ satisfying
\begin{equation}\label{light-zone-ineq}
\sum_{\alpha\in \Pi} e^{2\langle \alpha, X \rangle} > \|\nu\|^2,
\end{equation}
we have that $e^X$ lies outside the image of $\Lambda_\nu\rightarrow S$.
This allows one to compare Proposition~\ref{Lapid} and Proposition~\ref{rapid-fiber}.

\subsection{Proof of Theorem \ref{general tau} for $G=\GL_n(\R)$}\label{GLn-argument}
In this subsection we prove Theorem \ref{general tau} in the special case of $\PGL_n(\R)$, with the loss of $\varepsilon$ in the exponent. 
Fix an open bounded subset $\Omega\subset i\mathfrak{a}^*_{\text{reg}}$.
Then $\R_{>0} \Omega$ is an open cone of regular spectral parameters.
For $\nu\in \R_{>0}\Omega$,  we recall from \S\ref{s:spherical-rep} that $\lambda_{\nu}$ is the associated Laplacian eigenvalue. We want to prove that 
\[
\|W_\nu\|_\infty = \sup_{g\in G}|W_\nu(g)|=\sup_{a\in A}|W_\nu(a)| \gg_\varepsilon 
\lambda^{\frac{c(G)}{2}-\varepsilon}_\nu.
\] 
The argument will combine Proposition \ref{Lapid} (specialized to $\PGL_n(\R)$) with the Stade formula (see \eqref{clean stade} below).

The unramified principal series representation $\pi_\nu$ and the Whittaker model 
$\mathcal{W}(\pi_\nu,\psi)$ come equipped with canonically normalized $G$-invariant inner 
products (see \S\ref{sec:intro:outline}). The Jacquet-Whittaker function $W_\nu$ is the 
image of a unitary intertwining of the $L^2$-normalized $K$-fixed vector in $\pi_\nu$. We 
deduce that $\pnorm{W_\nu}_2=1$. 

Let $\Psi(s,W_\nu,\overline{W_\nu})$ denote the integral
\[
\Gamma_\R(ns)\int_{T_{n-1}} \left|W_\nu
\begin{pmatrix}
a & 0 \\
0 & 1
\end{pmatrix}
\right|^2 \det(a)^{s-1}\delta(a)^{-1}da.
\]
By the Stade formula \cite{Stade02}, we have for some $\eta$ depending on $n$,
\begin{equation}\label{clean stade}
\Psi(s,W_\nu,\overline{W}_\nu)=2^\eta\frac{L_\R(s,\pi_\nu\times\widetilde\pi_\nu)}{L_\R(1,\pi_\nu\times\widetilde\pi_\nu)}.
\end{equation}
Here, the Rankin-Selberg local $L$-function is
\begin{equation}\label{RS-def}
L(s,\pi_\nu\times \widetilde\pi_\nu)=\prod^n_{i=1}
\prod^{n}_{j=1}
\Gamma_\R(s+\mu_i-\mu_j),
\end{equation}
for certain $\mu_i\in i\R$, sometimes called the {\it Langlands parameters} of $\pi$. Since the central character of $\pi$ is trivial, we have $\sum_i \mu_i=0$. Applying Stirling's formula to the quotient of Gamma factors we obtain
\[
\Psi(\sigma,W_\nu,\overline{W}_\nu)\asymp \prod_{i\neq j}(1+|\mu_i-\mu_j|)^{(\sigma-1)/2}.
\]
Since $\nu\in \R_{>0}\Omega \subset i\mathfrak{a}_{\rm reg}^*$, each of the latter 
factors is of size $\lambda_\nu^{1/2}$,  yielding 
$\Psi(\sigma,W_\nu,\overline{W}_\nu)\asymp\lambda_\nu^{(\sigma-1)\dim U/2}$.

It will be convenient to introduce explicit coordinates in the integral defining $\Psi(s,W_\nu,\overline{W}_\nu)$ in order to extract the size of $W_\nu$. Writing
\[
\begin{pmatrix}
a & 0 \\
0 & 1
\end{pmatrix}={\rm diag}(y_1\cdots y_{n-1},y_2\cdots y_{n-1},\ldots ,y_{n-1},1)\in A,
\]
and using \S\ref{sec:Iwasawa}, we have
\[
\delta(a)=\prod_{i=1}^{n-2}y_i^{(n-1-i)i}, \quad \det(a)=\prod_{i=1}^{n-1}y_i^i,\quad 
da=\frac{dy_1}{y_1}\cdots \frac{dy_{n-1}}{y_{n-1}}.
\]
With these coordinates we may write $\Psi(s,W_\nu,\overline{W}_\nu)$ as
\[
\Gamma_\R(ns)\int_0^\infty\cdots  \int_0^\infty \left| W_\nu\left({\rm diag}(y_1\cdots y_{n-1},\ldots,
y_{n-1},1)\right)\right|^2\prod_{i=1}^{n-1}y_i^{i(s+i-n)}\frac{dy_i}{y_i}.
\]

We now decompose $\Psi(\sigma,W_\nu,\overline{W}_\nu)$ as $I_1+I_2$, where the integral $I_1$ is taken over the range $\max_i y_i \ll \lambda_\nu^{\frac12+\epsilon}$ and the integral $I_2$ over the complementary range. By Proposition \ref{Lapid}, after passing to a one-dimensional integral, we have
\[
I_2\ll \int_{r\gg \lambda_\nu^{\frac12+\epsilon}}r^{-N}\frac{dr}{r}\ll \lambda_\nu^{-N/2}
\]
for $N>1$ large enough. On the other hand, if $\sigma  >  n-1$, then we have $I_1\leq V\pnorm{W_\nu}^2_\infty$, where for a large enough constant $C>1$ we have put
\[
V=\prod_{i=1}^{n-1}\int_0^{C \lambda_\nu^{\frac12+\epsilon}}y_i^{i(\sigma+i-n)}\frac{dy_i}{y_i}.
\]
We deduce $\pnorm{W_\nu}^2_\infty\gg V^{-1}\lambda_\nu^{(\sigma-1)\dim U/2}$, and also
\[
V\asymp\lambda_\nu^{(\frac12+\epsilon)\cdot\sum_{i=1}^{n-1}i(\sigma+ i -n)}=\lambda_\nu^{(\frac12+\epsilon)\cdot \left(\sigma-1-\frac{n-2}{3}\right)\dim U}.
\]
Thus $\pnorm{W_\nu}_\infty^2\gg_\varepsilon \lambda_\nu^{\frac{(n-2)\dim U}{6}-\epsilon} = \lambda_\nu^{c(n)-\epsilon}$, which concludes the argument.\qed

\bigskip

The above argument can be refined to give lower bounds on $W_\nu$ even when $\nu$ is irregular. Indeed Proposition \ref{Lapid} (as well as Proposition \ref{rapid-fiber}) are valid for irregular $\nu$, as is the Stade formula. We have included the regularity assumption in Theorem \ref{general tau} to simplify notation and bring the idea of the proof to the forefront.

\begin{remark}\label{rem:KM}
We speculate on the geometric significance of the exceptionally large exponent $c(n)$ in Theorem \ref{th:gln}. For convenience, we restrict to the case $\Gamma=\PGL_n(\Z)$ in this paragraph.

Using standard notation for Siegel sets we consider the collar $\mathfrak{S}^{\asymp Y}:=\omega A^{\asymp Y}K$, where $\omega\subset U$ is a compact subset of $U$ and
\begin{equation*}
A^{\asymp Y}=\{a={\rm diag}(y_1\cdots y_{n-1},\ldots ,y_{n-1},1)\in A: y_i\asymp Y \; \forall \; i=1,\ldots ,n-1\},
\end{equation*}
for some parameter $Y\geq 1$. The right $G$-invariant measure, when expressed in the Iwasawa $UAK$ coordinates, is given by $dg=\delta(a)^{-1}dudadk$. Then the volume of this collar is
\begin{equation*}
\int_{\mathfrak{S}^{\asymp Y}}dg\asymp \int_{A^{\asymp Y}}\delta(a)^{-1}da\asymp Y^{-{\rm ht}(\PGL_n)}.
\end{equation*}
The relative volume of $\mathfrak{S}^{\asymp Y}$ is therefore seen to decrease as $n$ gets large, and this by a cubic power of $n$. In other words, the cuspidal regions of $\Gamma\backslash S_n$ become dramatically more ``pinched" as $n$ gets large. The narrower cusps of the higher rank spaces $\Gamma\backslash S_n$ create a bottleneck as the cusp forms transition from the oscillatory to the decay regime. With so little space to do so they get exceedingly large.

As mentioned in the introduction, Kleinbock and Margulis proved in \cite{KM99} that almost all geodesics penetrate the cusp at logarithmic speed $1/{\rm ht}(G)$. There, the collar plays the role of a moving target for the geodesic flow.
\end{remark}

\section{Proof of Theorem \ref{whitt supp}}\label{tau n}
The goal in this section is to study the critical points of the Whittaker phase function 
and to deduce Theorem~\ref{whitt supp} from the introduction.

\subsection{The Kostant-Toda lattice}\label{KT-lattice}
Since it figures prominently in the statement of Theorem \ref{whitt supp}, we now review 
some of the basic structures involved in the Kostant-Toda lattice, extensively studied in 
\cite{Kost79:Toda,Adler79}. This is the generalization to arbitrary split semisimple Lie 
groups of the classical Toda lattice, the latter being a totally integrable physical 
system of $n-1$ points of unit mass on the real line with nearest-neighbor exponentially 
attractive particle interaction. The classical Toda system can be represented by the 
Dynkin diagram of $\SL_n$.

We begin by writing, in the notation of \S\ref{sec:root-notation}, 
$\mathcal{T}:=\mathfrak{a}\oplus\bigoplus_{\alpha\in\Pi}(\g_\alpha+\g_{-\alpha})$. For 
example, if $G=\SL_n(\R)$ then $\mathcal{T}$ is the space of tridiagonal  matrices. 
Furthermore we write $\mathcal{J}:=\mathcal{T}\cap\p$. We have a projection map 
$\mathcal{J}\to\mathcal{Y}$ onto the subspace 
$\mathcal{Y}=\sum_{\alpha\in\Pi}(\g_\alpha+\g_{-\alpha})\cap\p$ of $\mathcal{J}$. In 
particular, for $G=\SL_n(\R)$, the map $\mathcal{J}\rightarrow\mathcal{Y}$ is just the 
extraction of the off-diagonal entries.

Let $\mathcal{T}_+\subset \mathcal{T}$ denote the open subset of those elements having 
positive $(\g_\alpha+\g_{-\alpha})$-coordinates. The elements of $\mathcal{T}_+$ are 
called {\it generalized Jacobi matrices (or elements).}  We let 
$\mathcal{J}_+:=\mathcal{T}_+\cap\p$ be the space of {\it generalized symmetric Jacobi 
matrices (or elements)} \cite[(5.4.2)]{Kost79:Toda}; it is a closed connected 
$2r$-dimensional submanifold of $\mathcal{T}_+$, where $r$ is the rank of $G$. We may 
write an arbitrary element in $\mathcal{J}_+$  as
\begin{equation}\label{symplectic}
\sum_{\alpha\in\Pi} p_\alpha H_\alpha+\sum_{\alpha\in\Pi} \alpha(a)(X_\alpha+X_{-\alpha}),\qquad p_\alpha\in\R,\; a\in A,
\end{equation}
referred to as \textit{Flaschka coordinates}. We have a map $\Upsilon: 
A\rightarrow \mathcal{J}_+$, given by sending $a$ to the element 
$\Upsilon (a)$ having Flaschka coordinates $p_\alpha=0$. The image of 
$\Upsilon$ is $\mathcal{Y}_+ = \mathcal{T}_+ \cap \mathcal{Y}=\mathcal{J}_+ \cap 
\mathcal{Y}$.

We have a decomposition $\g=\p\oplus\mathfrak{u}$. We let $X\mapsto X_s$ denote the 
corresponding projection onto $\p$. We can define a left-action of $B$ on 
$\mathfrak{p}$ 
by setting 
$(b,X)\mapsto (\Ad_{b} X)_s$. We have that $\mathcal{J}_+$ is the orbit of $B^\circ$ of 
$\Upsilon(e)$, where $B^\circ=UA$ is the identity component of $B$ (see~\cite[\S 
4.2.2]{Perelomov1990Book} for the case $G=\SL_n(\R)$, and the next 
\S\ref{dual-setting} for 
how to deduce it from a result of Kostant for general $G$).

We may endow $\mathcal{J}_+$ with the symplectic form 
$\sum_{\alpha\in\Pi}dp_\alpha\wedge d\alpha(a)/\alpha(a)$. The projection map 
$\mathcal{J}_+\rightarrow\mathcal{Y}_+$ is a Lagrangian fibration, with a 
canonical Lagrangian section $\mathcal{Y}_+ \subset \mathcal{J}_+$ given by inclusion.
As 
$H$ varies through the positive 
Weyl chamber $\mathfrak{a}_+$, the isospectral manifolds $\mathrm{Ad}_K(H)  \cap 
\mathcal{J}_+$ form a Lagrangian foliation of $\mathcal{J}_+$  \cite[\S 
4.1]{Kost79:Toda}. 

The Kostant-Toda lattice~\cite[(4.1.10)]{Perelomov1990Book}, is given by the 
Hamiltonian
\begin{equation}\label{Hamiltonian}
\frac12\sum_{\alpha\in\Pi}p_\alpha^2+\sum_{\alpha\in\Pi}\alpha(a)^2
\end{equation}
on the phase space $\mathcal{J}_+$. For example, for $G=\SL_n(\R)$, since 
$\alpha_i=E_{i,i}-E_{i+1,i+1}$, one has $\alpha_i(\exp(\frac12{\rm diag}(H_1,\ldots 
,H_n)))=e^{(H_i-H_{i+1})/2}$, recovering the nearest-neighbor exponential repulsion. 
It is an integrable system, and the flow preserves the above Lagrangian foliation. The 
Hamiltonian \eqref{Hamiltonian} is the restriction to $\mathcal{J}_+$ of the Killing form 
(appropriately normalized) on $\g$.

\subsection{The dual setting}\label{dual-setting}
From the Iwasawa decomposition 
$\g=\mathfrak{b}\oplus\mathfrak{k}$, we 
may deduce a corresponding decomposition
\begin{equation}\label{dual-decomp}
\g^*=\ker(\g^*\rightarrow\mathfrak{k}^*)\oplus \ker(\g^*\rightarrow\mathfrak{b}^*),
\end{equation}
into the spaces of linear functionals vanishing on 
$\mathfrak{k}$ and $\mathfrak{b}$, respectively.
The Cartan decomposition $\g=\p\oplus \mathfrak{k}$ induces an 
identification $\ker(\g^*\to\mathfrak{k}^*)\cong\p^*$ via $\xi\mapsto\xi|_\p$, which will be in place throughout this entire \S\ref{tau n}.
In particular we view $\p^*$ as a subspace of $\g^*$.

Through the use of a non-degenerate $\Ad$-invariant symmetric bilinear form 
$B(\cdot,\cdot)$ on $\g$, we 
may transport the structures of the previous subsection into similar ones for the dual 
space $\g^*$ using the induced isomorphism $\g^*\simeq\g$. We write $X_\xi\in\g$ for the image of $\xi\in\g^*$ under this identification. This yields an identification $\ker(\g^*\rightarrow\mathfrak{k}^*)\simeq\p$, since $\mathfrak{p}$ and $\mathfrak{k}$ are 
orthogonal, and 
$\ker(\g^*\rightarrow\mathfrak{b}^*)\simeq\mathfrak{u}$, which is a standard result in 
Lie theory. In this way, the decomposition 
\eqref{dual-decomp} corresponds to the previous decomposition $\g=\p\oplus\mathfrak{u}$.
Continuing, we see that:

\smallskip

\noindent $\bullet$ $\mathcal{J}\subset\p$ corresponds to the subspace 
$\mathcal{J}^*\subset\p^*$ of functionals vanishing on $[\mathfrak{u},\mathfrak{u}]\oplus 
\mathfrak{k}$;

\smallskip

\noindent $\bullet$ $\mathfrak{a}\subset \mathcal{J}$ corresponds to the subspace 
$\mathfrak{a}^*\subset \mathcal{J}^*$ of functionals vanishing on $\mathfrak{u}\oplus 
\mathfrak{k}$;

\smallskip

\noindent $\bullet$ $\mathcal{Y}\subset \mathcal{J}$ corresponds to the subspace 
$\mathfrak{u}_{\rm 
ab}^*\subset \mathcal{J}^*$ of functionals vanishing on 
$[\mathfrak{u},\mathfrak{u}]\oplus \mathfrak{a}\oplus 
\mathfrak{k}$;

\smallskip

\noindent $\bullet$  $\mathcal{J}^*=\mathfrak{a}^*\oplus 
\mathfrak{u}_{\rm ab}^*$, and the projection $\mathcal{J}\rightarrow\mathcal{Y}$ 
corresponds to  
$\mathcal{J}^*\rightarrow \mathfrak{u}_{\rm ab}^*$;

\smallskip

\noindent $\bullet$ $\mathcal{J}_+$ corresponds to 
$\mathcal{J}^*_+=\{v\in\mathcal{J}^*:v(X_\alpha)>0 \ \forall\, \alpha\in\Pi\}$, and 
$\mathcal{Y}_+$ to $\mathfrak{u}_{{\rm ab},+}^*=\mathfrak{u}_{\rm 
ab}^*\cap\mathcal{J}_+^*$;

\smallskip

\noindent $\bullet$ the Borel subgroup $B$ acts on 
$\mathfrak{p}^*$ by transport of the $B$-action on $\mathfrak{p}$, via the above 
isomorphism $\mathfrak{p}^*\simeq \mathfrak{p}$. More precisely, $(b,\xi)$ is sent to 
the linear form $Y\mapsto B(b.X_\xi,Y)$. This indeed lies in 
$\ker(\g^*\rightarrow\mathfrak{k}^*)$ from the orthogonality of $\p$ and $\mathfrak{k}$;

\smallskip

\noindent $\bullet$ the Lagrangian fibration $\mathcal{J}_+\rightarrow\mathcal{Y}_+$ 
corresponds to $\mathcal{J}^*_+\rightarrow \mathfrak{u}_{{\rm ab},+}^*$;

\smallskip

\noindent $\bullet$ the Lagrangian section $\mathcal{Y}_+ \subset \mathcal{J}_+$ 
corresponds to $\mathfrak{u}_{{\rm ab},+}^* \subset \mathcal{J}^*_+$;

\smallskip

\noindent $\bullet$ the isomorphism 
$\Upsilon:A \overset{\sim}{\longrightarrow} \mathcal{Y}_+$ corresponds to the isomorphism 
$A \overset{\sim}{\longrightarrow} \mathfrak{u}_{{\rm 
ab},+}^*$, 
given by $a\mapsto\ell_a=\Ad^*_a(\ell_1)$, 
where the standard functional $\ell_1\in\mathfrak{u}_{\rm ab}^*$ is introduced in 
\S\ref{Whitt-structures}.

\smallskip
 
From the $\Ad$-equivariance of the map $\xi\mapsto X_\xi$, the left-action of $B$ on $\mathfrak{p}^*$ can alternatively be described as
\begin{equation}\label{dual-B-action}
(b,\xi) \mapsto \Ad_{b}^*(\xi)|_\p.
\end{equation}
As above, we view the functional $\ell_1\in \mathfrak{u}_{\rm ab}^*$ as an 
element of 
$\mathcal{J}^*=\mathfrak{a}^*\oplus \mathfrak{u}^*_{\rm ab}\subset \mathfrak{p}^*\subset 
\mathfrak{g}^*$, trivial on $\mathfrak{a}$.
Then $\mathcal{J}^*_+$ is the $B^\circ$-orbit of $\ell_1$ as a consequence of Kostant's 
result to be recalled below. An arbitrary element in $\mathcal{J}^*_+$ may be written as $\Ad_{b}^*(\ell_1)|_\p$, 
where the element $b\in B^\circ$ is well-defined up to right-multiplication by $[U,U]$.

The Lagrangian fibration 
$\mathcal{J}^*_+\rightarrow\mathfrak{u}_{{\rm ab},+}^*$ is given by 
\begin{equation}\label{u-ab-proj}
\Ad_{b}^*(\ell_1)|_\p\longmapsto \Ad_{b}^*(\ell_1)|_\mathfrak{u} = \ell_{a},
\end{equation}
where $a=e^{H(b)}$ in the Iwasawa decomposition, that is $b\in aU=Ua$.

Next, we have a canonical isomorphism between $\p^*$ and $\mathfrak{b}^*$ given by 
$\xi 
\mapsto \xi|_{\mathfrak{b}}$, where we recall that we identify $\xi\in \mathfrak{p}^*$ 
with a functional on $\mathfrak{g}$ vanishing on $\mathfrak{k}$.
Equivalently, this isomorphism between $\p^*$ and $\mathfrak{b}^*$ is obtained by
duality from $c: \mathfrak{b}\subset \mathfrak{g} \twoheadrightarrow \mathfrak{p}$, where 
the second map is Cartan symmetrization map. Indeed, $\xi|_{\mathfrak{b}}=c^*(\xi)$ 
because for every 
$Y\in\mathfrak{b}$, we have
\[
\langle \xi|_{\mathfrak{b}},Y \rangle= \langle \xi,Y \rangle= \langle \xi, c(Y) \rangle= \langle 
c^*(\xi),Y \rangle,
\]
since $\xi$ is trivial on $\mathfrak{k}$. Similarly, the inverse isomorphism from 
$\mathfrak{b}^*$ to $\p^*$ 
is obtained by duality from $\mathfrak{p}\subset \mathfrak{g} \twoheadrightarrow 
\mathfrak{b}$, where the second map is Iwasawa projection.
Hence the isomorphism intertwines the co-adjoint $B$-action on $\mathfrak{b}^*$, and the 
$B$-action on $\p^*$ given by~\eqref{dual-B-action}.
Let $v_b\in \mathfrak{b}^*$ be the element that corresponds to $\Ad_{b}^*(\ell_1)|_\p\in 
\mathcal{J}_+^*\subset \mathfrak{p}^*$ under the isomorphism.
Specifically $v_b=\operatorname{Ad}^*_{b}(v_1)$, and $v_1$ is the standard functional on 
$\mathfrak{b}=\mathfrak{a}\oplus \mathfrak{u}$ that vanishes on $\mathfrak{a}$, and 
coincides with $\ell_1$ on $\mathfrak{u}$. The analogue of~\eqref{u-ab-proj} is simply
$v_b|_{\mathfrak{u}}=\ell_a$, where $a=e^{H(b)}$.

As was shown in \cite[\S2]{Adler79} for $G=\SL_n(\R)$ 
and \cite[\S6]{Kost79:Toda} in general, the image of $\mathcal{J}_+^*$ in 
$\mathfrak{b}^*$ can 
be realized as the 
co-adjoint orbit of $B^\circ$ acting on $v_1\in \mathfrak{b}^*$ (see also 
\cite[Thm.~5.1]{Bloch-Gay--Balmaz-Ratiu}). 
One may then use the Lie-Poisson structure on 
$\mathfrak{b}^*$ to define a $B^\circ$-invariant symplectic form on 
$\mathcal{J}^*_+$ (the basic formula is described in 
\S\ref{sub:coadjoint}). It is a key result of Kostant~\cite[Prop. 
6.4]{Kost79:Toda}, generalizing an earlier work 
of Flashka for $G$ of type $A$, that the resulting symplectic form on $\mathcal{J}_+^*$
corresponds to the one defined on $\mathcal{J}_+$ in 
\S\ref{KT-lattice}. The same result is also established in~\cite[Theorem 
5.4]{Bloch-Gay--Balmaz-Ratiu} (to see this, use Prop.~5.2 in \emph{loc. cit.} to show 
that the Flashka 
map denoted $F$ in \emph{loc. cit.} coincides with the isomorphism between
$\mathcal{J}^*_+$ and $\mathcal{J}_+$ composed with the coordinate 
map~\eqref{symplectic}).
In particular, the composition 
  $\mathcal{J}_+^*\subset \p^* \overset{\sim}{\rightarrow} \mathfrak{b}^*$ is a moment 
  map for the 
 action~\eqref{dual-B-action} of $B^\circ$. 

With those symplectic structures in place, we may form the {\it Toda isospectral manifold}
\[
\mathscr{L}_\nu:=\mathrm{Ad}_K^*({\rm Im}\,\nu) \cap \mathcal{J}^*\qquad (\nu\in i\mathfrak{a}^*),
\]
which appears in the statement of Theorem \ref{whitt supp}. When $\nu$ is regular,  
$\mathscr{L}_\nu^+:=\mathrm{Ad}_K^*({\rm Im}\,\nu) \cap \mathcal{J}^*_+$ is a Lagrangian 
leaf of $\mathcal{J}^*$ \cite[Theorem 6.7.1]{Kost79:Toda}. Composing the Lagrangian 
immersion $\mathscr{L}_\nu^+ \rightarrow\mathcal{J}^*_+$  with the Lagrangian fibration 
$\mathcal{J}^*_+\rightarrow\mathfrak{u}_{{\rm ab},+}^*$ yields the Lagrangian mapping
\begin{equation}\label{first-triple}
\mathscr{L}_\nu^+\longrightarrow\mathcal{J}^*_+\longrightarrow\mathfrak{u}_{{\rm ab},+}^*.
\end{equation}
The next paragraphs explore the relation of this mapping with similar structures coming 
from our stationary phase analysis of the Jacquet integral.

\subsection{An explicit description of $\Lambda_\nu$} 
We now return to the Lagrangian $\Lambda_\nu$, where $\nu\in i\mathfrak{a}^*$, which is
the image of the map 
$\Sigma_\nu\rightarrow T^*(S)$ of \eqref{immersion}. We would 
like to calculate defining equations for $\Lambda_\nu$. In this subsection, we do not 
assume 
that $\nu$ is regular. The main tool is the moment map for the Hamiltonian action of $G$ 
on the cotangent bundle $T^*(S)$, endowed with its canonical $G$-invariant symplectic 
form $\omega_S$. 

We begin by recalling the definition of the moment map. The symmetric space $S=G/K$ 
admits a left isometric action by $G$, which extents canonically to a Hamiltonian 
left-action of $G$ on $T^*(S)$. When $T^*(S)$ is trivialized as the fiber product 
$G\times_K\mathfrak{p}^*$, the $G$-action is given by $g'.[g,\xi]=[g'g,\xi]$. Recall that 
$\mathfrak{p}^*$ is the space of functionals on $\g$ which vanish 
on $\mathfrak{k}$, and thus the trivialization includes the natural identification $\xi 
\mapsto [1,\xi]$ of $\mathfrak{p}^*$ with $T^*_e(S)$. The moment map is
\[
m_G: T^*(S)\longrightarrow\mathfrak{g}^*,\qquad [g,\xi] \longmapsto \mathrm{Ad}^*_g(\xi).
\]
As should be the case for a moment map, 
note that $m_G$ is $G$-equivariant with respect to the translation action on 
$T^*(S)$ 
and the $\Ad^*$-action on $\mathfrak{g}^*$.

\begin{prop}\label{p:moment}
Let $\nu\in i\mathfrak{a}^*$ be arbitrary. Then $\Lambda_\nu$ consists of $[g,\xi]\in 
 T^*(S)$ such that
\begin{equation}\label{K-orbit}
\Ad_{\upkappa (g)}^*(\xi) \in \mathrm{Ad}_{S_\mathsf{w}^+}^*({\rm Im}\,\nu)
\end{equation}
and
\begin{equation}\label{moment}
\mathrm{Ad}^*_g(\xi)|_{\mathfrak{u}} =\ell_1.
\end{equation}
\end{prop}

\begin{proof} 
We use Lemma~\ref{lem:DKV} to evaluate the partial derivative of the Whittaker phase 
function \eqref{def:F}, in the second variable $x\in S$, with respect to $Y\in 
\mathfrak{p}$. We obtain
\begin{equation}\label{Y-derivative}
F_\nu(u,g;Y)=
B\left(H_\nu, \mathrm{Ad}_{\upkappa(\mathsf{w}ug)}(Y)\right)
=
\langle 
\mathrm{Ad}^*_{\upkappa(\mathsf{w}ug)^{-1}}({\rm Im}\,\nu), Y
\rangle.
\end{equation}
We deduce from the definition \eqref{immersion} of the map $\Sigma_\nu\rightarrow T^*(S)$ that
\[
\Lambda_\nu=\left\{[g, \mathrm{Ad}^*_{\upkappa(\mathsf{w}ug)^{-1}}({\rm Im}\,\nu)] :  (u,x)\in \Sigma_\nu,\; x=gK \right\}.
\]
Recall the notation $\uptau(g)=g\upkappa(g)^{-1}$ from \S\ref{sec:Iwasawa}. For $u\in U$ 
and $g\in G$, set $v=e^{-H(g)}u\uptau (g)\in U$. By the left-$A$-invariance of 
$\upkappa:G\rightarrow K$, we have $\upkappa(\mathsf{w}u\uptau 
(g))=\upkappa(\mathsf{w}v)$. Thus
\[
\mathrm{Ad}^*_{\upkappa (g)}\mathrm{Ad}^*_{\upkappa(\mathsf{w}ug)^{-1}}({\rm 
Im}\,\nu)=\mathrm{Ad}^*_{\upkappa(\mathsf{w}u\uptau (g))^{-1}}({\rm 
Im}\,\nu)\in\Ad^*_{S_\mathsf{w}^+}({\rm Im}\,\nu),
\] 
and condition \eqref{K-orbit} is thus satisfied. Then, we use again Lemma~\ref{lem:DKV} 
to evaluate the partial derivative of $F_\nu(u,g)$ with respect to $Z\in
\mathfrak{u}$. We write $\mathsf{w}ue^{tZ}g= \mathsf{w}ug e^{t \mathrm{Ad}_{g^{-1}}(Z)}$, and thus obtain
\begin{equation}\label{F-nu-U-derivative}
F_\nu(u;Z,g) =
\langle
\mathrm{Ad}^*_g \mathrm{Ad}^*_{\upkappa(\mathsf{w}ug)^{-1}}({\rm Im}\,\nu),
Z
\rangle
-\langle \ell_1, Z \rangle.
\end{equation}
We deduce that the fiber critical set $\Sigma_\nu$ consists of pairs $(u,x)$, with $x=gK$, such that for all $Z\in \mathfrak{u}$
\[
\langle\mathrm{Ad}^*_g \mathrm{Ad}^*_{\upkappa(\mathsf{w}ug)^{-1}}({\rm Im}\,\nu),Z\rangle=\langle \ell_1, Z \rangle,
\]
showing that the condition \eqref{moment} is also met. 

Conversely let $[g,\xi]\in G\times_K \mathfrak{p}^*$ satisfy \eqref{K-orbit} and 
\eqref{moment}. We need to show that there exists $u\in U$ such that $\xi = 
\mathrm{Ad}^*_{\upkappa(\mathsf{w}ug)^{-1}}({\rm Im}\,\nu)$; indeed, assuming this, then 
it follows from \eqref{moment} and \eqref{F-nu-U-derivative} that $(u,gK)\in\Sigma_\nu$. 
Using \eqref{K-orbit} there is $k\in M\cdot S^+_{\mathsf{w}}\subset K$ such that $\mathrm{Ad}^*_{\upkappa(g)}(\xi)=\Ad^*_{k^{-1}}({\rm Im}\,\nu)$; indeed recall that since
$\mathsf{w}^2=1$, the big Bruhat cell is invariant under $k\to k^{-1}$. By Lemma~\ref{lem:dkappa} there is $v\in U_{\mathsf{w}}$ such that
$k=\upkappa(\mathsf{w}v)$. Furthermore letting $u=e^{H(g)}v\uptau(g)^{-1}\in U$, we see that $k=\upkappa(\mathsf{w}v)=\upkappa(\mathsf{w}u\uptau(g))$ and
therefore $\mathrm{Ad}^*_{\upkappa(g)}(\xi) = \mathrm{Ad}^*_{\upkappa(\mathsf{w}u\uptau(g))^{-1}}(\mathrm{Im}\, \nu )$ as desired.
\end{proof} 

\begin{remark}
Equation~\eqref{moment} is essentially the definition of the Peterson variety, see~\cite{Kostant:Toda-flag}. 
\end{remark}

\subsection{Symplectic reduction}\label{sec:sympl-reduction}
We now restrict the translation $G$-action on $T^*(S)$ to the subgroup $U$ and 
symplectically reduce 
the tangent space $T^*(S)$ with respect to this restricted action. Let
\begin{equation}\label{defn-Umoment}
m_U: T^*(S) \longrightarrow\mathfrak{u}^*,\qquad [g,\xi]\longmapsto 
\Ad_g^*(\xi)|_{\mathfrak{u}}
\end{equation}
denote the corresponding moment map.

Any abelian functional on $\mathfrak{u}$ is fixed under the adjoint action of $U$ on 
$\mathfrak{u}^*$ and, if non-degenerate,
is a regular value under $m_U$. Thus the Hamiltonian action of $U$ on $T^*(S)$ preserves 
$m_U^{-1}(\ell_1)$, and on this fiber the action is free and proper. Let 
\[
T^*(S)_1= U \backslash\! \backslash_{\ell_1} T^*(S)
\]
be the symplectic reduction of the $U$-action on $T^*(S)$ over $\ell_1$. In other 
words, $T^*(S)_1$ is the orbit space of $U$ acting on $m_U^{-1}(\ell_1)$. By the Marsden-Weinstein theorem, $T^*(S)_1$ is endowed with a natural symplectic 
structure. Namely, there is a unique symplectic form $\omega_S^{\rm red}$ on $T^*(S)_1$ 
whose pullback under the quotient map to $m_U^{-1}(\ell_1)$ coincides with $\omega_S$.

To facilitate the study of $T^*(S)_1$ here and in the next subsection, we observe that 
$B^\circ$ acts simply transitively 
by left-multiplication on $S$. The map $(T^*(S),\omega_S)\rightarrow 
(T^*(B^\circ),\omega_{B^\circ})$ given in their respective left-trivialisations 
by
\begin{equation}\label{left-trivialisation}
G\times_K \mathfrak{p}^* \longrightarrow B^\circ\times\mathfrak{b}^*,\qquad 
[g,\xi]\longrightarrow 
(\uptau(g),\Ad^*_{\upkappa(g)}(\xi)|_{\mathfrak{b}})
\end{equation}
is a symplectomorphism. Here and in what follows we shall continue to identify 
$T^*(S)$ 
with $G\times_K \mathfrak{p}^*$ and $T^*(B^\circ)$ with $B^\circ\times\mathfrak{b}^*$ 
in writing down formulas.
\begin{prop}\label{p:TStoA}
The map $(T^*(S)_1,\omega_S^{\rm red})\to A$ sending the $U$-orbit of $[g,\xi]\in 
m_U^{-1}(\ell_1)$ to $e^{H(g)}$ is a Lagrangian fibration. It admits a (non-unique) lift 
to a symplectomorphism $(T^*(S)_1,\omega_S^{\rm red}) \overset{\sim}{\longrightarrow} 
(T^*A,\omega_A)$.
\end{prop}

\begin{proof}
The left-action of $U$ on $B^\circ$ 
given by left-multiplication lifts to an action of $U$ on 
$T^*(B^\circ)$ by the rule $u.(b,v)=(ub,v)$. The induced moment map
\[
T^*(B^\circ)
\rightarrow \mathfrak{u}^*,\qquad (b,v)\mapsto 
\Ad_b^*(v)|_{\mathfrak{u}}
\]
agrees, under the identification \eqref{left-trivialisation}, with the map $m_U$ from 
\eqref{defn-Umoment}, and we shall denote it by the same letter $m_U$. Letting
\[
T^*(B^\circ)_1:= U \backslash\!\backslash_{\ell_1} T^*(B^\circ) =  U\backslash 
m_U^{-1}(\ell_1),
\]
we must therefore prove the statement of the proposition for the map 
$T^*(B^\circ)_1\rightarrow A$ sending the $U$-orbit of $(b,v)\in m_U^{-1}(\ell_1)$ to 
$e^{H(b)}$. We shall 
do so by identifying the reduced space $T^*(B^\circ)_1$ as the cotangent bundle over $A$.

According to the theory of cotangent bundle reduction~\cite[\S6]{Ortega-Ratiu}, 
any choice of a left-$U$-invariant $1$-form $\alpha_1$ taking values in 
$m_U^{-1}(\ell_1)$ gives rise to a symplectomorphism
\[
(T^*(B^\circ)_1,\omega_{B^\circ}^{\rm red})\overset{\sim}{\longrightarrow} 
(T^*(U\backslash B^\circ),\omega_{U\backslash B^\circ}^{\rm mod}),
\]
extending the map $T^*(B^\circ)_1\rightarrow U\backslash B^\circ$ which sends $U.(b,v)$ 
to $U.b$. Here, $\omega_{U\backslash B^\circ}^{\rm mod}$ is the modified symplectic form 
given by the shift of $\omega_{U\backslash 
B^\circ}$ by a \textit{magnetic term} induced by the choice of $\alpha_1$. It is shown in 
\cite[Lemma 5.3]{Bloch-Gay--Balmaz-Ratiu} that one can chose $\alpha_1$ with vanishing 
magnetic term, so that $\omega_{U\backslash B^\circ}^{\rm mod}=\omega_{U\backslash 
B^\circ}$. (This identification is non-unique, as it depends on the choice of a $1$-form 
$\alpha_1$ whose associated magnetic 
term vanishes. However in our situation there is a standard 
choice~\cite[p.487]{Bloch-Gay--Balmaz-Ratiu}, see also the discussion at 
\S\ref{sub:Flashka}.)

Using the decomposition $B^\circ=UA$, we may identify $(T^*(U\backslash 
B^\circ),\omega_{U\backslash B^\circ})$ with $(T^*(A),\omega_A)$. The composite 
Lagrangian fibration 
\[
T^*(S)_1\simeq T^*(B^\circ)_1\overset{\sim}{\longrightarrow} T^*(U\backslash 
B^\circ)=T^*(A)\rightarrow A
\]
sends $[g,\xi]$ to $e^{H(g)}$.
\end{proof}

We now symplectically reduce $\Lambda_\nu$ with respect to the $U$-action. Note that  
\eqref{moment} states that $\Lambda_\nu$ lies in $m_U^{-1}(\ell_1)$. Moreover, both 
formulae in Proposition \ref{p:moment} are $U$-invariant, showing that $\Lambda_\nu$ is 
invariant under the $U$-action. Finally, we deduce from \eqref{Y-derivative} that the 
immersion $\Sigma_\nu\rightarrow T^*(S)$ of \eqref{immersion} is injective, hence a 
diffeomorphism onto its image $\Lambda_\nu$. Indeed, since $\nu$ is regular, 
$\Ad^*_{\upkappa}({\rm Im}\, \nu)$ determines $\upkappa=\upkappa(\mathsf{w}ug)^{-1}$ and 
then $u$ under the Bruhat decomposition.

Let $\Lambda_\nu^{\rm red}$ be the quotient $ U \backslash \Lambda_\nu$ inside $T^*(S)_1$.

\begin{prop}
If $\nu \in i\mathfrak{a}^*$ is regular, then $\Lambda_\nu^{\rm red}$ is Lagrangian 
inside $T^*(S)_1$ and
\begin{equation}\label{2nd-triple}
\Lambda_\nu^{\rm red}\longrightarrow T^*(S)_1\longrightarrow A
\end{equation}
defines a Lagrangian mapping.
\end{prop}

\begin{proof}
Recall that for $\nu$ regular $\Lambda_\nu$ is Lagrangian in $T^*(S)$. The reduction of 
an invariant Lagrangian is again Lagrangian, see 
e.g.~\cite[Thm.~3.2]{reduced-lagrangian}. It follows that $\Lambda_\nu^{\rm red}$ is 
Lagrangian inside $T^*(S)_1$.
\end{proof}

\subsection{Lagrangian equivalence}\label{sub:reduction} 
An equivalence of Lagrangian mappings
\[
L_i\longrightarrow M_i\longrightarrow N_i\qquad (i=1,2)
\]
is a fiber-preserving symplectomorphism from $M_1$ to $M_2$, sending $L_1$ to $L_2$.
We would now like to show that the above two triples \eqref{first-triple} and 
\eqref{2nd-triple} are equivalent. Theorem \ref{t:critical} below accomplishes this, 
providing a precise version of Theorem \ref{whitt supp}. 

Recall that $T^*(B^\circ)$ has, in fact, \textit{two} moment maps, coming from left and 
right multiplication of $B^\circ$ on itself. Right-multiplication acts as follows on 
$T^*(B^\circ)\simeq B^\circ\times\mathfrak{b}^*$:
\begin{equation}\label{LR-actions}
(b,v).b'=(bb',\Ad_{b'^{-1}}^*(v)),
\end{equation}
and has corresponding $\Ad_{b^{-1}}^*$-equivariant moment map
\[
m_B^r: B^\circ\times\mathfrak{b}^*\rightarrow\mathfrak{b}^*,\qquad (b,v)\mapsto v.
\]
Using the identification $\xi \mapsto \xi|_{\mathfrak{b}}$ of $\mathfrak{p}^*$ with 
$\mathfrak{b}^*$ as in \S\ref{dual-setting}, we 
see that $m_B^r$ composes with \eqref{left-trivialisation} to yield
\[
T^*(S) \longrightarrow\mathfrak{p}^*,\qquad [g,\xi]\longmapsto 
\Ad_{\upkappa(g)}^*(\xi).
\]
As the left- and right-actions of $B^\circ$ on itself commute, we shall see in the 
proof below that $T^*(S)_1$ admits a right-action by $B^\circ$, and that the associated 
reduced moment map is given by the same formula:
\begin{equation}\label{key}
T^*(S)_1\longrightarrow\mathfrak{p}^*,\qquad U.[g,\xi]\longmapsto 
\Ad_{\upkappa(g)}^*(\xi).
\end{equation}
This is well-defined, since $g\mapsto \upkappa(g)$ is left $U$-invariant. It is this 
last map which will allow us to compare the two triples \eqref{first-triple} and 
\eqref{2nd-triple}.

\begin{thm}\label{t:critical}
The map \eqref{key} defines a symplectomorphism
\begin{equation}\label{map}
T^*(S)_1 \longrightarrow \mathcal{J}_+^*,
\end{equation}
and induces a Lagrangian equivalence between \eqref{first-triple} and 
\eqref{2nd-triple}. 
In particular, the symplectomorphism~\eqref{map} induces an open embedding of 
$\Lambda_\nu^{\rm 
red}$ to $\mathscr{L}^+_\nu$ making the diagram
\begin{equation}\label{commutative} \begin{CD}
\Lambda_\nu^{\rm red}
@>>>
T^*(S)_1 @>>>  A  \\
@VVV @VVV @ VVV\\
\mathscr{L}_\nu^+ @>>>  \mathcal{J}_+^* @>>> \mathfrak{u}_{{\rm ab}, +}^*
\end{CD}
\end{equation}
commutative. The third vertical map is the isomorphism 
$A\rightarrow\mathfrak{u}_{{\rm ab}, 
+}^*$, given by $a\mapsto\ell_{a^{-1}}$.
\end{thm}

\begin{proof}
We begin by showing that \eqref{map} maps onto $\mathcal{J}^*_+$. The fiber $m_U^{-1} 
(\ell_1)$ is contained in the set of $[g,\xi]\in G\times_K \p^*$ 
satisfying~\eqref{moment}, 
that is $\mathrm{Ad}^*_g \xi|_{\mathfrak{u}} =  \ell_1$. For 
any 
such $[g,\xi]$, and any $X\in [\mathfrak{u},\mathfrak{u}]$ we have
\[
\langle \mathrm{Ad}^*_{\upkappa(g)} \xi, X \rangle =\langle \mathrm{Ad}^*_g \xi, 
\mathrm{Ad}_{\uptau(g)} X \rangle=0,
\]
where we wrote $g=\uptau(g)\upkappa(g)$, and we used the fact that 
$\mathrm{Ad}^*_{\tau(g)} X\in [\mathfrak{u},\mathfrak{u}]$, because 
$[\mathfrak{u},\mathfrak{u}]$ is stable under the adjoint action for the Borel subgroup 
$B$. Thus the restriction of the map \eqref{key} to $m_U^{-1}(\ell_1)$ takes values in 
$\mathcal{J}^*$, in fact, in $\mathcal{J}_+^*$. The map \eqref{map} is surjective, since 
the element 
of $U.[b^{-1},\operatorname{Ad}^*_{b}(\ell_1)|_\mathfrak{p}]\in m_U^{-1}(\ell_1)$ is sent 
 to
$\operatorname{Ad}^*_{b}(\ell_1)|_\mathfrak{p}\in\mathcal{J}^*_+$. 

Next we show that \eqref{map} defines a symplectomorphism. We again use the 
identification $T^*(S)\simeq T^*(B^\circ)$ from~\eqref{left-trivialisation}. Also recall 
the notation 
$T^*(B^\circ)_1=U\backslash\!\backslash_{\ell_1}T^*(B^\circ)=U\backslash 
m_U^{-1}(\ell_1)$ from the proof of Proposition~\ref{p:TStoA}. Under this identification, 
the map 
\eqref{map} 
becomes 
\begin{equation}\label{key2}
T^*(B^\circ)_1\rightarrow \mathcal{J}_+^*, \qquad U.(b^{-1},v_b)\mapsto 
\Ad^*_b(\ell_1)|_{\mathfrak{p}}.
\end{equation}
We now show that it is a symplectomorphism.

\begin{enumerate}
\item The moment map $m^l_B: T^*(B^\circ) \rightarrow \mathfrak{b}^*$ for 
the left $B^\circ$-action on $T^*(B^\circ)$ is given by $m^l_B(b,v)=\Ad_b^*(v)$. Note 
that $[U,U]$ is the isotropy subgroup $\{b\in B^\circ: \Ad_b^*(v_1)=v_1\}$ of 
$v_1$ under the co-adjoint action $B^\circ$. Let 
$B^\circ\backslash\!\backslash_{v_1}T^*(B^\circ)$ denote the reduction $[U,U]\backslash 
(m^l_B)^{-1}(v_1)$ of $T^*(B^\circ)$ over $v_1\in\mathfrak{b}^*$. We have
\[
(m^l_B)^{-1}(v_1)=\{(b^{-1},v_b): b\in B^\circ\}\subset 
\{(b^{-1},\Ad_{b}^*(v_1+v)):b\in B^\circ,\ v\in\mathfrak{a}^*\}=m_U^{-1}(\ell_1).
\]
The map
\[
m_U^{-1}(\ell_1)\rightarrow (m^l_B)^{-1}(v_1),\qquad 
(b^{-1},\Ad_{b}^*(v_1+v))\mapsto 
(b^{-1},v_b)
\] 
descends to the respective quotients; indeed if $u=u_1u'$, where $u_1\in [U,U]$, then
\[
(b^{-1}u^{-1},\Ad_{b}^*(v_1+v))=(b^{-1}u^{-1},\Ad_{ub}^*(v_1+v'))
\]
is sent to
\[
\begin{aligned}
(b^{-1}u^{-1},v_{ub})&=(b^{-1} u'^{-1} u_1^{-1},\Ad_{u_1u'b}^*(v_1)) \\
&=u_1^{-1}.(b^{-1}u'^{-1}, 
\Ad_{u'b}^*\Ad_{u_1}^*(v_1))=u_1^{-1}.(b^{-1}u'^{-1}, v_{u'b}).
\end{aligned}
\]
The resulting map
\[
T^*(B^\circ)_1
\longrightarrow B^\circ\backslash\!\backslash_{v_1}T^*(B^\circ)
\]
is a symplectomorphism.
\item Let $T^*(B^\circ)/\!/_{\!v_1}B^\circ$ be the reduction 
$(m_B^r)^{-1}(v_1)/[U,U]$ of $T^*(B^\circ)$ over $v_1$ under the 
\textit{right} action of $B^\circ$ on $T^*(B^\circ)$, as described in \eqref{LR-actions}. 
Note that the fiber of $m_B^r$ over $v_1$ is $B^\circ\times\{v_1\}$. One can pass 
from the left and right reductions via the symplectomorphism $T^*(B^\circ)\rightarrow 
T^*(B^\circ)$ sending $(b,v)\mapsto (b^{-1},\Ad_b^*(v))$. This map preserves the 
respective fibers over $v_1$ and descends to a symplectomorphism on the reduced spaces
\[
B^\circ\backslash\!\backslash_{v_1} T^*(B^\circ)\overset{\sim}{\longrightarrow} 
T^*(B^\circ)/\!/_{\!v_1}B^\circ,\qquad [U,U].(b^{-1},v_b)\mapsto (b^{-1},v_1).[U,U].
\]
This map is well-defined, since $u.(b,v_b)=(ub,v_b)\in [U,U].(b,v_b)$ is sent to
\[
(b^{-1}u^{-1},v_1)=(b^{-1},\Ad^*_{u^{-1}}(v_1)).u^{-1}=(b^{-1},v_1).u^{-1}\in 
(b^{-1},v_1).[U,U].
\]

\item Finally we observe that $\mathcal{J}^*_+=\{\Ad^*_b(\ell_1)|_{\mathfrak{p}}: b\in 
B^\circ\}$ is the symplectic reduction of the right $B^\circ$-action on $T^*(S)$ over 
$v_1$. Indeed, quotienting the fiber $B^\circ\times\{v_1\}$ by the centralizer subgroup 
yields a symplectomorphism
\[
T^*(B^\circ)/\!/_{\!v_1}\, B^\circ\rightarrow\mathcal{J}_+^*,\qquad 
(b^{-1},v_1).[U,U]\mapsto \Ad^*_b(\ell_1)|_{\mathfrak{p}}.
\]
\end{enumerate}
The composition of these maps yields \eqref{key2}.

Let us now show that \eqref{map} is fiber preserving, that is, that the right square 
of~\eqref{commutative} is commutative. On one hand, we see from Proposition 
\ref{p:TStoA} that the preimage of $a\in A$ under $T^*(B)_1 \rightarrow A$ is 
$\{U.(a,v_{a^{-1}}+v),\ v\in \mathfrak{a}^*\}$. On the other hand, it follows 
from 
\eqref{u-ab-proj} that the preimage of $\ell_{a}\in \mathfrak{u}_{{\rm ab}, 
+}^*$ under the map \eqref{map}
 is $\{U.[a^{-1}u^{-1},\operatorname{Ad}^*_{u}(\ell_a)|_{\mathfrak{p}}],\ u\in U\}$. 
We deduce that \eqref{map} respects fibers, and the induced map $A\to 
\mathfrak{u}_{{\rm ab}, +}^*$ is given by $a\mapsto \ell_{a^{-1}}$.  

Now let $[g,\xi]\in\Lambda_\nu$. Since $\xi$ lies in $\Ad_K^*({\rm Im}\,\nu)$ by 
\eqref{K-orbit} then so does $\Ad_{\upkappa(g)}^*\xi$, whence $\Lambda_\nu$ is mapped to 
$\mathscr{L}_\nu=\Ad^*_K({\rm Im}\,\nu)\cap\mathcal{J}^*$. The same is therefore true of 
$\Lambda_\nu^{\rm red}$. The assertion that the left square of~\eqref{commutative} is 
commutative is a 
reformulation of Proposition~\ref{p:moment}. Also as a 
consequence the map $\Lambda_\nu^{\rm red}\to \mathscr{L}^+_\nu$ is an open embedding. 
\end{proof}

\subsection{Remarks on the constructions in this section}\label{s:critical-remarks}
\hfill

\subsubsection{}
Olshanetsky--Perelomov~\cite[\S 9]{OP80} showed that $T(S)_1$ is 
symplectomorphic to the Toda lattice $\mathcal{J}_+$. They also proved that the 
Hamiltonian of the geodesic flow on $S$ reduces to the Toda Hamiltonian. The book 
\cite[\S 4.4]{Perelomov1990Book} contain further exposition of some of these results,
and~\cite[\S4.4.3]{Perelomov1990Book} is essentially step (1) of our proof above. 
Since the 
Lagrangian submanifolds $\Lambda_\nu$ and $\mathcal{L}_\nu^+$ are determined by the 
corresponding Hamiltonians, their result should imply that the two Lagrangian 
mappings \eqref{first-triple} and \eqref{2nd-triple} are equivalent.  We 
have opted to provide an independent proof of Theorem \ref{t:critical} for completness, 
and because our emphasize is less on the Hamiltonian and more on the Lie-theoretic 
constructions. 

\subsubsection{}\label{sub:Flashka} Our approach is largely inspired by modern geometric 
treatments such 
as~\cite{Bloch-Gay--Balmaz-Ratiu}.
Specifically, the symplectomorphism in step (3) of the proof above is a special case of 
the symplectomorphism $\mathbf{J}^{\nu_0}_R$ of~\cite[Thm.~4.5]{Bloch-Gay--Balmaz-Ratiu} 
for the choice on p.487 of \emph{loc. cit.} of the section $s_{\mu_0}$ and associated 
one-form $\alpha_{\nu_0}$ in the notation of \emph{loc. cit.} 
The Flashka map $F$ in \emph{loc. cit.} is the symplectomorphism $\mathcal{J}_+^* \to 
T^*A$. In summary, if we combine this with the construction of~\eqref{map} in the proof 
above, and 
Kostant's result recalled in \S\ref{dual-setting}, we have the following commutative 
diagram of 
symplectomorphisms:
\[
\begin{tikzcd}
T^*(S)_1 \ar[r] & T^*(B^\circ )_1 \ar[r,"\mathbf{J}_R^{\nu_0}"] 
\ar[d,"\text{Prop.}\ref{p:TStoA}",swap] & 
\mathcal{J}^*_{+} \ar[d,leftrightarrow,"\text{Kostant}"] 
\ar[dl,"F",swap] 
\ar[r,hookrightarrow] & \mathfrak{p}^*  \\
& T^*A \ar[r,"\eqref{symplectic}",swap] &  \mathcal{J}_+ &
\end{tikzcd}
\]

\subsubsection{} $\!\!\!\!\!\!$ For $G=\SL_2(\R)$, $\mathscr{L}_\nu = \Ad^*_K({\rm Im}\,\nu)$ is a circle. For general $G$ and regular $\nu$, $\mathscr{L}_\nu$ is a compact aspherical manifold of dimension the rank of $G$; see \cite{Davis:aspherical,Tomei:isospectral-tridiagonal}. For example for $G=\SL_3(\R)$ it is known that $\mathscr{L}_\nu$ is a genus $2$ surface. 

\subsubsection{}
$\!\!\!\!\!\!$ 
A similar symplectic reduction to that in \S\ref{sec:sympl-reduction} occurs in the 
quantum cohomology of flag manifolds~\cite{GLO:new-integral,Giventhal-Kim}.

\subsubsection{} $\!\!\!\!\!\!$ It is also possible to reduce the fiber critical set 
$\Sigma_\nu\subset U\times G$ under the $U$-action. Indeed, note that for any $v\in U$, 
we have $F_\nu(uv,v^{-1}g) = F_\nu(u,g) - \langle \ell_1, v \rangle$. Reducing under this 
action is equivalent to restricting the second parameter to $A$. We denote by 
$\Sigma_\nu^{\rm red} \subset U \times A$ the set of pairs $(u,a)$ which are critical for 
$u\mapsto F_\nu(u,a)$.

There is a natural section $\Lambda_\nu^{\rm red} \to \Lambda_\nu$ obtained by taking the Iwasawa $A$-part of $g$. Composing with the reduction map $\Lambda_\nu \to \Lambda_\nu^{\rm red}$ and then the 
open embedding to $\mathscr{L}_\nu^{+}$, we obtain a map $\Sigma_\nu^{\rm red} \to 
\mathscr{L}_\nu^{+}$ given by $(u,a) \mapsto \mathrm{Ad}^*_{\upkappa(\mathsf{w}ua)^{-1}} 
({\rm Im}\,\nu)$. The diagram~\eqref{commutative} being a Cartesian square implies that 
for any $a\in A$ the set $(u,a)$ of critical points in $\Sigma_\nu^{\rm red}$ lying over 
$a$ is sent bijectively to the $\nu$-isospectral fiber over $\ell_{a^{-1}}$.

\subsubsection{} $\!\!\!\!\!\!$ The method of co-adjoint orbits described in 
\S\ref{sub:coadjoint} yields a natural interpretation of the construction of 
$\mathscr{L}_\nu$ and the Lagrangian mapping. Indeed, the intersection of the coadjoint 
orbit $\mathrm{Ad}^*_G({\rm Im}\,\nu)$ with 
$\ker(\mathfrak{g}^* \to \mathfrak{k}^*)$ captures the spherical vector of
the representation $\pi_\nu$. This intersection is simply $\mathrm{Ad}^*_K({\rm Im}\,\nu)$, 
which itself is the zero level set of the moment map of the $K$-action on 
$\mathrm{Ad}^*_G({\rm Im}\,\nu)$. To investigate the analytic behavior of the 
associated spherical Whittaker function, relative to non-degenerate additive characters 
of $U$, we consider the intersection $\mathrm{Ad}^*_K({\rm Im}\,\nu)\cap 
\ker(\mathfrak{g}^* \to \mathfrak{u}^*_{\rm der})$, where $\mathfrak{u}_{\rm 
der}=[\mathfrak{u},\mathfrak{u}]$. 
This intersection turns out to coincide with $\mathscr{L}_\nu$.
The resulting Lagrangian fibration 
\[
\mathscr{L}_\nu\rightarrow \ker(\mathfrak{g}^*\to \mathfrak{u}^*_{\rm der})
\cap 
\ker(\mathfrak{g}^*\to \mathfrak{p}^*)
\to  \mathfrak{u}^*_{\mathrm{ab}}
\] 
contains the fibration~\eqref{first-triple} as an open embedding.
It is also natural from this point of view that the variation of its fibers relates to 
the asymptotic behavior 
of $W_\nu$.

\subsubsection{} $\!\!\!\!\!\!$ Note that $\Ad^*_K({\rm Im}\,\nu)$ is Lagrangian inside 
$\Ad^*_G({\rm Im}\,\nu)$, see~\cite{Azad-vandenBan-Biswas}, and that the projection of 
$\mathrm{Ad}^*_K({\rm Im}\,\nu)$ onto $\mathfrak{a}^*$ is, by the Kostant convexity 
theorem, the convex hull of the Weyl group orbit of ${\rm Im}\,\nu$.

\section{Lagrangian singularities}\label{s:singularities}
In \S\ref{tau n} we defined, under a regularity assumption on $\nu$, three Lagrangian 
mappings
\begin{equation}\label{mH-def}
\Lambda_\nu\rightarrow S,\qquad\quad \Lambda_\nu^{\rm red}\rightarrow A,\qquad\text{and}\qquad \mathscr{L}^+_\nu\rightarrow\mathfrak{u}_{\rm ab}^*.
\end{equation}
We described their precise relationship in Theorem \ref{t:critical} and the discussion 
preceding it. In this section, as a way of preparing the proofs of Theorem \ref{n=3 tau} 
and Theorem~\ref{cor:Pearcey}, we discuss Lagrangian mappings in a more general context. 
In particular, we shall recall some facts related to numerical invariants and associated 
asymptotics of singularities, quoting from the extensive literature on this subject .

\subsection{Stratification by singularity type}\label{ad-summary}

Let $E\rightarrow B$ be a Lagrangian fibration and $L\rightarrow E$ a Lagrangian immersion. Then the composition $\pi : L\rightarrow B$ is a Lagrangian mapping to the base space $B$. We write
\begin{equation}\label{trichotomy}
B =\mathsf{C}\sqcup\mathsf{L}\sqcup\mathsf{S}
\end{equation}
where, by definition, the {\it caustic locus} $\mathsf{C}$ is the set of critical values of $\pi$, the {\it light zone} $\mathsf{L}$ is the complement of $\mathsf{C}$ in $\pi(L)$, and the {\it shadow zone} $\mathsf{S}$ is the complement of $\pi(L)$ in $B$.

One could refine $\mathsf{C}$ according to singularity types, which would lead to a stratification of the Lagrangian $L$ via the fibers of $\pi$. See \cite[\S2]{AV} for the general theory of stratifications via coranks of the first differential of a smooth mapping restricted to singular loci, and [{\it loc. cit.}, \S 21] for that same theory applied to the special case of Lagrangian mappings. We have carried this out for the examples in \eqref{mH-def} when $n=3$ (and under a self-dual assumption), where it already exhibits a rich structure.

\subsection{Associated asymptotics}\label{associated-asymp}
We highlight two ways in which the above decomposition of the base $B$ yields information about the corresponding oscillatory integrals. We specialize to the case of the mappings in \eqref{mH-def} associated with Whittaker functions and refer the reader to \cite{Duistermaat:oscillatory, Hormander:FIO-I} for the more general passage from Lagrangian mappings to Fourier integral operators.

\medskip

-- {\it Shadow zone and rapid decay}: The image of either of the first two Lagrangian mappings in \eqref{mH-def} should be thought of as the ``essential
support" of the Whittaker function. For example, we showed in Proposition \ref{rapid-fiber} that the Whittaker function decays rapidly in the shadow
zone. By comparison, viewing the Whittaker function as an eigenfunction of the quantum Toda lattice, $\mathscr{L}_\nu$ is the characteristic variety of the corresponding system of linear partial differential equations (see e.g.~\cite{To:asymptotic-Whittaker}). The image of $\mathscr{L}_\nu\rightarrow\mathfrak{u}^*_{\rm ab}$ corresponds to the classically allowed region. 

Now $\mathscr{L}_\nu$ is closed in $\mathrm{Ad}^*_K({\rm Im}\,\nu)$ and so is compact. It follows that the classically allowed region is also compact. In fact, since the projection $\mathfrak{g}^*\to \mathfrak{u}^*_{\mathrm{ab}}$ is orthogonal for the invariant scalar product (see \S\ref{sec:root-notation}), we see that the image of $\mathscr{L}_\nu\rightarrow\mathfrak{u}^*_{\rm ab}$ is included in the ball of radius $\|\nu\|$. In light of Theorem~\ref{t:critical}, the same inclusion holds for the image of $\Lambda_\nu^{\rm red}\rightarrow A$. This proves inequality~\eqref{light-zone-ineq}.

\medskip

-- {\it Singularities and degenerate critical points:} Let $Q_\nu(u,x)=\nabla^2_u F_\nu(u,x)$ denote the fiber Hessian of $F_\nu$ at $(u,x)\in\Sigma_\nu$ and $d_\nu([g,\xi])$ denote the differential of the mapping $\Lambda_\nu\rightarrow S$ at $[g,\xi]$. In view of the non-degeneracy assumption on $\nu$, the map $\Sigma_\nu\rightarrow\Lambda_\nu$ of \eqref{immersion}, sending $(u,x)$ to $[g,\xi]$, induces an isomorphism
\begin{equation}\label{ker-Hess}
\ker Q_\nu(u,x)\;\overset{\simeq}{\longrightarrow}\;  \ker d_\nu([g,\xi]).
\end{equation}
(See e.g. \cite[\S 19.3]{AV} or \cite[Theorem 3.14]{Hormander:FIO-I}.) In particular, $(u,x)$ is a degenerate critical point for $u\mapsto F_\nu(u,x)$ if and only if $[g,\xi]$ is singular for the mapping $\Lambda_\nu\rightarrow S$.
In other words, $(u,x)$ is non-degenerate if and only if the tangent space of $\Lambda_\nu$ at $[g,\xi]$ is transversal to the fiber of the projection
$[g,\xi]\mapsto x$, where $x=gK$. This correspondence remains true for the reduced mapping $\Lambda_\nu^{\rm red}\rightarrow A$, because the $U$-orbits are
transverse to the fibers.

\subsection{Numerical invariants of Lagrangian singularities}\label{num-inv}
Let $\pi: L\rightarrow B$ be a Lagrangian mapping. In this subsection we discuss several of the numerical invariants one may associate with Lagrangian singularities, which are the map germs of such singular points, viewed up to Lagrangian equivalence. For more information on the theory of singularities, the reader is referred to the classic book by Arnol'd, Gusein-Zade, and Varchenko \cite{AV}.

After the {\it corank} -- the codimension of the image of $d\pi_p$ -- the first numerical 
invariant one typically encounters is the {\it multiplicity} (or Milnor number) of a 
singularity $\mu$. Roughly speaking, $\mu$ is the (maximum) number of non-degenerate 
critical points into which a singularity splits under a small perturbation. Indeed, a 
function having a critical point of finite multiplicity $\mu$ is equivalent, in a 
neighborhood of the point, to a polynomial of degree $\mu+1$ (cf. \cite[\S 6.3]{AV}). For 
the precise definition of $\mu$, the reader can consult \cite[Definition 
2.1]{Greuel-Lossen}.

The {\it modality} of a singularity is a non-negative integer which counts the number of 
continuous parameters (or moduli) that enter into the definition of the associated normal 
form. We refer to \cite[\S 2.4]{Greuel-Lossen} or \cite[p.184]{AV} for the exact 
definition. A singularity of modality $0$ is called {\it simple}. Simple singularities, 
having no moduli, appear discretely. 

Arnol'd has classified stable simple singularities. (Stable singularities are those which persist under small perturbations; they are the only ones visible ``with the naked eye.") Below is a list of the simple stable singularities, along with function germs representing each class:
\begin{itemize}
\item[$(A)$] $A_k$: $\pm x_1^{k+1}+x_2^2+\cdots +x_n^2$ \quad ($k\geq 2$);
\item[$(D)$] $D_k$: $x_2(x_1^2\pm x_2^{k-2})+x_3^2+\cdots +x_n^2$\quad ($k\geq 4$);
\item[$(E)$] $E_6$: $x_1^3\pm x_2^4+x_3^2+\cdots +x_n^2$;
\item[] $E_7$: $x_1(x_1^2+x_2^3)+x_3^2+\cdots +x_n^2$;
\item[] $E_8$: $x_1^3+x_2^5+x_3^2+\cdots +x_n^2$.
\end{itemize}
 
Singularities of type $A$ are of corank $1$ and those of type $D$ and $E$ are of corank 
$2$. Thus any simple singularity is of corank at most $2$ (see \cite[Lemma 4.2]{A2}). 
Moreover, any corank $1$ singularity of finite multiplicity is necessarily simple. Thus 
the type $A$ singularities (also called Morin singularities) can be characterized as 
those having corank $1$ and finite multiplicity; these facts are summarized in \cite[\S 
11.1]{AV} or \cite[Theorem 2.48]{Greuel-Lossen}. The multiplicity of an ADE singularity 
is indicated in its subscript. 

We shall be primarily interested in $A_2$ and $A_3$ singularities. An $A_2$-type singularity is sometimes referred to as a {\it fold singularity}, and an $A_3$-type singularity as a {\it cusp singularity}. As an example of a fold singularity, consider the projection of the sphere to the horizontal plane touching the south pole. The singular points are the points of the equator; they are all fold singularities. They arise from a coalescence of two critical points. One can realize a cusp singularity from the projection of the surface $z=x^3+xy$ to the $(y,z)$-plane; the warp on one half of the surface is known as a Whitney pleat. For visualizations of both of these fundamental examples, see Figures 7 and 8 in Section 1 of \cite{AV}. A theorem of Whitney (see \cite[\S 1.5]{AV}) states that the stable singularities of a differentiable map between surfaces are either non-degenerate, or of type $A_2$ or $A_3$. 

Finally, there is yet another numerical invariant of a critical point, called the {\it 
singularity index} and denoted $\beta$. The singularity index is defined by the 
asymptotic behavior of associated oscillatory integrals~\cite[Definition 4.2.1]{A1}. The 
index of singularity is $\beta$ if the integral is of size $t^{-\frac{n}{2}+\beta}$ for 
generic choice of amplitude function. Arnol'd~\cite{A1} has calculated the singularity 
index for all simple singularities and many others; they turn out to be rational numbers. 
For the simple singularities one has $\beta=1/2-1/N$, where $N$ is the corresponding 
Coxeter number $N(A_k)=k+1,\; N(D_k)=2k+2, \; N(E_6)=12,\; N(E_7)=18,\; N(E_8)=30$.

\section{Proof of Theorem~\ref{general tau}}\label{non-deg-phase}

Let $\nu\in i\mathfrak{a}_{\mathrm{reg}}^*$ be an arbitrary regular spectral parameter 
and recall from \S\ref{s:statement} the value of the constant 
$c(G):=\tfrac{\mathrm{ht}(G)-\dim(U)}{2}$. 
We  shall establish in this section that
\begin{equation}\label{proof-whittaker-bound}
\|W_{t\nu}\|_\infty \gg t^{c(G)}, \quad \text{as $t\to \infty$.}
\end{equation}
The idea is to test the Whittaker function against a symbol localized in phase-space 
inside a single sheet of $\Lambda_\nu \cap T^*V$, where $V\subset A$ is an open set over 
which the cover $\Lambda_\nu\rightarrow A$ is unramified. We then apply a stationary 
phase argument for a Morse-Bott function.
Our argument is inspired from~\cite[\S1.2]{Duistermaat:oscillatory}, except that there the symbol was chosen instead transversal to $\Lambda_\nu\cap T^* V$, and the singularity was then Morse.

\subsection{Non-degenerate critical points}
In this preliminary subsection, we show the existence of an appropriate open set 
$V\subset A$ and testing function $\phi$. We keep the same notation as in \S\ref{tau n}.  
\begin{lem}\label{l:origin}
The origin $0\in \mathfrak{u}_{\mathrm{ab}}^*$ is not a critical value of $\mathscr{L}_\nu \to \mathfrak{u}_{\mathrm{ab}}^* $.
\end{lem}
\begin{proof} 
Inside $\mathscr{L}_\nu$ the fiber above $0 \in  \mathfrak{u}_{\mathrm{ab}}^* $ consists of $\{\mathrm{Im}\, \nu^w,\ w\in W\}$. In a neighborhood of any of these points we have that $\mathscr{L}_\nu\to \mathfrak{u}_{\mathrm{ab}}^*$ is a local diffeomorphism. Indeed we compute that the tangent space of $\mathscr{L}_\nu$ at $\mathrm{Im}\, \nu^w$  is $[\mathfrak{k},\mathrm{Im}\, \nu^w] \cap \mathcal{J}^*$ which surjects onto $ \mathfrak{u}_{\mathrm{ab}}^* $ because $\nu$ is regular, and $[H_{\nu^w},X_\alpha-X_{-\alpha}]=\alpha(H_{\nu^w})(X_\alpha+X_{-\alpha})$ for every $\alpha\in \Pi$.
\end{proof}

Another way to approach Lemma \ref{l:origin} is to remark that we are computing the critical points of the Iwasawa projection $u\mapsto \langle \nu,H(\mathsf{w}u) \rangle$. These are known~\cite{DKV} to be non-degenerate if $\nu$ is regular; see also the discussion following Lemma~\ref{lem:DKV}.

\begin{cor}
	 The light zone $\mathsf{L}\subset A$ contains a translate of the negative Weyl chamber
	\[
	\exp(- \mathfrak{a}_+) = 
	\{a\in A,\ \alpha(a)<1\ \forall \alpha\in \Pi\}.
	\]
\end{cor}
\begin{proof}
	Under the map $A\to \mathfrak{u}_{\mathrm{ab}}^*$, $a\mapsto \ell_a$, the preimage of a neighborhood of $0\in \mathfrak{u}_{\mathrm{ab}}^*$ contains a translate of the negative Weyl chamber.
	Hence the assertion follows from Lemma~\ref{l:origin} and Theorem~\ref{t:critical}. 
\end{proof}

This corollary implies that for certain $g\in S$, the phase function $u\mapsto F_{\nu}(u,g)$ is Morse and has $|W|$ distinct critical points. In particular we can choose a bounded connected open set $V\subset A$ inside the translate of the negative Weyl chamber, which is small enough so that the covering $\Lambda_\nu\to A$ becomes unramified over $V$.

Since $\Lambda_\nu$ is transverse to the vertical fibers of $T^*V\to V$, we may choose a 
real-valued function $\phi\in C^\infty(V)$, depending only on $\nu$, such that the graph 
$\{(a,d\phi(a))\in T^*(V); a\in V\}$ of $d\phi$ is entirely inside $\Lambda_\nu$. This 
graph then picks out a single sheet in $\Lambda_\nu\cap T^*(V)$. Let $\beta\in 
C^\infty_c(A)$ have support contained in $V$.

\subsection{Stationary phase approximation for Morse--Bott functions} 
The proof of Theorem~\ref{general tau} involves oscillatory integrals of the form
\begin{equation}\label{osc-MB}
\int_{\R^d} \alpha(x) e^{it G(x)} dx,
\end{equation}
where $\alpha,G\in \mathcal{C}^\infty(\R^d)$ with $\alpha$ of compact support. The following is a known generalization of the stationary phase approximation to the case of Morse--Bott functions, see e.g.~\cite{Chazarain}.
\begin{prop}\label{p:chaz}
Suppose that $G$ is Morse--Bott and that the set of critical points of $G$ contained in the support of $\alpha$ form a connected submanifold $W\subset \R^d$.
Then the oscillatory integral~\eqref{osc-MB} is asymptotic as $t\to \infty$ to
\[
\left(
\frac{2\pi}{t}
\right)^{\frac{d-e}{2}}
e^{itG(W)-\frac{i\pi}{4}\sigma}
\int_W \alpha(x) \left|
\mathrm{det}_{W}\ G''(x)
\right |^{-\frac12} dx
\]
where $e=\dim W$, $G(W)$ is the value of  $G(x)$ at any point  $x\in W$, and  $\sigma$ (resp. $\mathrm{det}_W\ G''$) is the signature (resp. determinant) of the Hessian of $G$ in the directions transverse to $W$.
\end{prop}

\subsection{Proof of~\eqref{proof-whittaker-bound}, conditional on Hypothesis~\ref{hypothesis}}\label{s:proof-general-tau}  We now proceed to prove Theorem~\ref{general tau}, using Hypothesis~\ref{hypothesis}. Our approach is inspired by that of~\cite[(1.2.6)]{Duistermaat:oscillatory}, where the other extreme case of the graph of $d\phi$ being transverse to $\Lambda_\nu \cap T^* V$ is treated.

With $\phi$ and $\beta$ as 
above, we now form the integral
\begin{equation}\label{1sheet}
\int_A W_{t\nu}(ta)\delta(ta)^{-1/2}\nu(t\mathsf{w}a\mathsf{w})^{-1}e^{-it\phi(a)} \beta(a)da.
\end{equation}
Using the integral representation of $W_\nu$ in~\eqref{W-form2}, and 
Hypothesis~\ref{hypothesis}, this is equal to
\[
\int_A \int_U  \delta(\mathsf{w}u)^{1/ 2}  e^{it(\sigma(u,a) -\phi(a) )}  \beta(a) 
\alpha(u) du da + O(t^{-N}),
\] 
where we have set $\sigma(u,a):=B(H_\nu,H(\mathsf{w}u)) - \langle \ell_1, aua^{-1} \rangle$. 

We wish to show that the new phase function $(u,a)\mapsto \sigma(u,a)- \phi(a)$ is Morse-Bott. Let
\[
S_{\nu,\phi}=\{(u,a)\in U\times V: d_u\sigma(u,a)=0,\; d_a\sigma(u,a)=d\phi(a)\}
\]
be its critical set, a connected submanifold of $U\times V$. 
The equation $d_u\sigma(u,a)=0$ is that of $\Sigma_\nu$, see~\eqref{def:F}.
Note that $\phi$ and $V$ were chosen so that $d_a\sigma(u,a)=d\phi(a)$ holds throughout a 
certain sheet of $\Sigma_\nu$; the solution locus for $S_{\nu,\phi}$ therefore reduces 
locally to the single 
equation $d_u\sigma(u,a)=0$, and in particular, $\dim S_{\nu,\phi}=\dim A$. From this it 
follows that the tangent space $T_{(u,a)} 
S_{\nu,\phi}$ is the kernel of the differential of $d_u\sigma$. Explicitly, $T_{(u,a)} S_{\nu,\phi}$ is equal to 
\begin{equation}\label{twice}
\{(X,Y)\in T_uU\times T_aV: d^2_u\sigma(u,a)\cdot X + d_a d_u \sigma(u,a)\cdot Y=0\}.
\end{equation}
On the other hand, the Hessian of $\sigma(u,a)-\phi(a)$ is 
\[
\begin{pmatrix} 
d^2_u \sigma(u,a) & d_a d_u \sigma(u,a) \\
d_a d_u \sigma(u,a) & d^2_a \sigma(u,a) - d^2 \phi(a)
\end{pmatrix}.
\]
Since $d_a \sigma (u,a) = d \phi(a)$ for every $(u,a)\in S_{\nu,\phi}$, the bottom row vanishes  along $T_{u,a}S_{\nu,\phi}$. Thus the kernel of the Hessian is again given by \eqref{twice}. We deduce that the Hessian of $\sigma(u,a)- \phi(a)$ is non-degenerate in directions transverse to $T_{(u,a)}S_{\nu,\phi}$, so that $(u,a)\mapsto \sigma(u,a)- \phi(a)$ is Morse-Bott.

As $u$ is a non-degenerate critical point of $u\mapsto \sigma(u,a)$, the Hessian 
$d_u^2\sigma(u,a)$ is invertible, which implies that 
$X$ is uniquely determined by $Y$ in \eqref{twice}. Thus,
 under the natural projection $T_uU\times T_aV\rightarrow  T_aV$ the tangent space 
 $T_{(u,a)}S_{\nu,\phi}$ is sent isomorphically to $T_aV$. 

In view of the above considerations, Proposition~\ref{p:chaz} above then shows that, up 
to a non-zero multiplicative constant, the integral in \eqref{1sheet} is asymptotic to 
$t^{-\dim(U)/2}$ as $t\to \infty$. Applying the triangle inequality we deduce that for 
each $t\ge 1$ there exists $a\in V$ such that $W_{t\nu}(ta)$ is asymptotically greater 
than $t^{(\mathrm{ht}(G)- \dim (U))/2} = t^{c(G)}$. 

\subsection{The Whittaker function as superposition of plane waves}\label{s:superposition-waves}
We end this section with some remarks on the proof of Theorem \ref{general tau}.

\medskip

(i) The method of \S\ref{s:proof-general-tau} can be used to show the more precise result that there exists $a\in V$ such that $W_{t\nu}(ta)$ is asymptotically greater than $t^{c(G)}$, as soon as $V$ intersects the image $\mathrm{Im}(\Lambda_\nu \to A)$, regardless on whether or not the points in the fiber are non-degenerate. The important property is that the phase function $F_\nu(u,g)$ with parameters be non-degenerate in the sense of \cite{Duistermaat:oscillatory}, and that the symbol of the amplitude is non-vanishing.

\medskip

(ii) For $a$ in a negative translate of the Weyl chamber inside the light zone $\mathsf{L}$, the Whittaker function $a\mapsto W_\nu(a)$ is asymptotically a linear superposition of $|W|$ plane waves, where $W=W(\mathfrak{g},\mathfrak{a})$ is the Weyl group. This is because 
$F_\nu(u,a)$ is Morse, allowing for an application of the stationary phase approximation in its uniform version with parameters (see \cite[Theorem 7.7.6]{Ho} and \cite[Theorem 2.9]{Var1}).
The fibers of the Lagrangian mapping $\Lambda_\nu \to A$, correspond to the momentum of the plane waves. 
Since by construction the momenta are distinct, these plane waves are linearly independent which implies the lower bound of Theorem~\ref{general tau}.
The argument in \S\ref{s:proof-general-tau} above amounts to directly testing $W_\nu(a)$ against one of the plane waves. Some of the plane waves coalesce when $a$ approaches the caustic $\mathsf{C}_\nu$, which will be studied in the next section for $G=\PGL_3(\R)$.

\section{Proof of Theorem \ref{n=3 tau}}\label{sec:fine-phase}

In this section, we impose a self-duality assumption on the spectral parameter $\nu\in i\mathfrak{a}_+^*$. This allows us to give a precise description of the critical set for $G=\PGL_3(\R)$. More precisely, we shall provide explicit equations defining the shadow zone, the light region, and the caustic locus defined in \S\ref{ad-summary}.

Now a uniform description of the asymptotic behavior of the Jacquet-Whittaker function depends on more than just this partition. One also needs information on the {\it configuration} of the critical points, which is encoded in the singularities of the Lagrangian mapping $\Lambda^{\mathrm{red}}_{H} \to A$. Thus, in the main result of this section, Proposition \ref{prop:CRIT}, we shall decompose the caustic locus $\mathsf{C}$ into strata according to the degeneracy type, and decompose the light region $\mathsf{L}$ according to the size of the fibers. 

All of this information will determine the asymptotic size of $W_\nu(a)$, uniformly in $\nu$ and $a$. 

\subsection{Notation and hypotheses}\label{8not}
Let
\begin{equation*}
i\mathfrak{a}^*_{\rm sd}=\left\{\nu\in i\mathfrak{a}^*_+: \langle \nu,H_1\rangle=\langle \nu,H_2\rangle\right\}
\end{equation*}
be the center of the positive Weyl chamber $i\mathfrak{a}^*_+$, where $H_1={\rm 
diag}(1,-1,0)$, $H_2={\rm diag}(0,1,-1)$ are the standard co-roots. Unramified tempered 
principal series representations $\pi_\nu$ are self-dual precisely for $\nu\in 
i\mathfrak{a}^*_{\rm sd}$, whence the notation. Note that $\nu\in i\mathfrak{a}^*_{\rm 
sd}$ is the positive ray generated by $\nu_0=2\pi i(\varpi_1+\varpi_2)$, where $\varpi_i$ 
are the fundamental weights (the dual basis of the co-roots $H_i$). When studying the 
Lagrangian $\Lambda_{t\nu_0}^{\rm red}$ we can, without loss of generality, restrict to 
$t=1$; this follows from the scale invariance of the phase function in the $(\nu,a)$ 
parameters.

For notational simplicity, we shall work with Lie algebra structures rather than their duals. Thus instead of $i\mathfrak{a}^*$ we work with $\mathfrak{a}$, using the identification between the two given by the form $B(X,Y)={\rm tr}(XY)$. Thus, the matrix in $\mathfrak{a}$ corresponding to $\varpi_1+\varpi_2 \in \mathfrak{a}^*$ is $H:={\rm diag}(1,0,-1)$, and we shall work with $\Ad_K(H)$ rather than $\Ad_K^*(\varpi_1 + \varpi_2)$.

There is a Lagrangian mapping $\Lambda_H^{\rm red}\rightarrow A$ as described in \S\ref{sub:reduction}. Let
\[
\mathscr{F}(a)\subset\Lambda_H^{\rm red}
\]
denote the fiber over $a\in A$. Implementing~\eqref{trichotomy} we have 
$A=\mathsf{S}\sqcup\mathsf{L}\sqcup\mathsf{C}$, according to whether $\mathscr{F}(a)$ is 
empty, consists entirely of non-singular points, has at least one singular point, 
respectively.

\subsection{Statement of result}\label{sec:STATE}
We begin by defining certain subsets of $A$ which are represented graphically in Figure \ref{beautiful pic}. Let
\[
\mathsf{L}_1=\{27 y_1^4 y_2^4-18 y_1^2 y_2^2+4 y_2^2+4 y_1^2<1\}
\]
and
\[
\mathsf{L}_2=\{27 y_1^4 y_2^4-18 y_1^2 y_2^2+4 y_2^2+4 y_1^2>1\;\text{ and }\; y_1^2+y_2^2<1 \}. 
\]
Let $a_{\rm cusp}$ be the unique point in $A$ given by
\[
(y_1,y_2)=(1/\sqrt{3},1/\sqrt{3}).
\]
Finally put
\[
\mathsf{C}_1=\left\{y_1^2+y_2^2=1\right\},
\] 
and
\[
\mathsf{C}_2=\left\{(y_1,y_2)\neq (1/\sqrt{3},1/\sqrt{3}): 27 y_1^4 y_2^4-18 y_1^2 y_2^2+4 y_2^2+4 y_1^2=1\right\}.
\]

\begin{figure}
\ifpdf
\includegraphics[width=8cm]{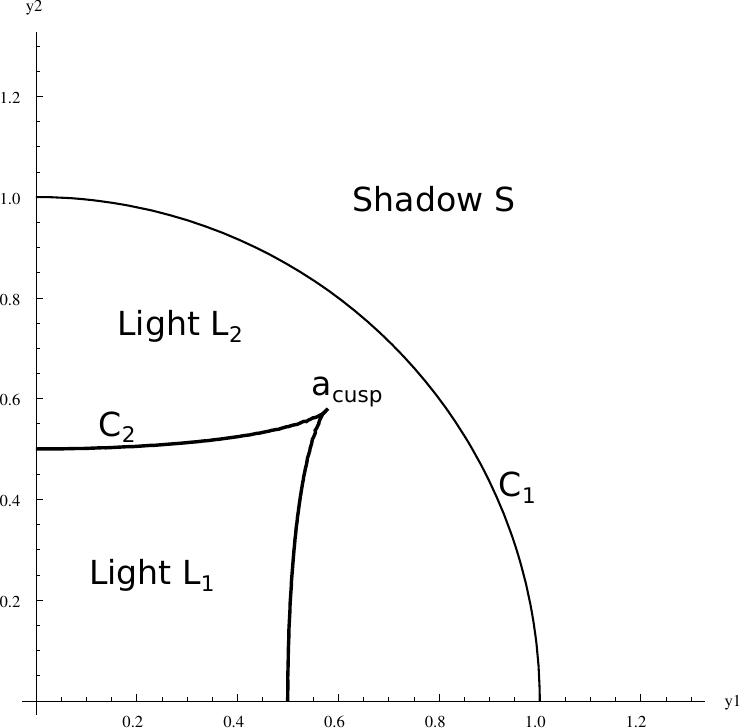}
\else
\fig{caustic-zones-shades.pdf}
\fi
\caption{Critical point configuration for self-dual spectral parameter.}\label{beautiful 
pic}
\label{fig:light-shadow}
\end{figure}

In this section we prove the following result.

\begin{thm}\label{prop:CRIT}
We have
\[
\mathsf{S}=\{y_1^2+y_2^2>1\},\qquad  \mathsf{L}=\mathsf{L}_1\sqcup\mathsf{L}_2,\qquad \mathsf{C}=\mathsf{C}_1\sqcup\{a_{\rm cusp}\}\sqcup\mathsf{C}_2.
\]
Moreover, we have the following critical point configurations:
\begin{enumerate}
\medskip
\item for all $a\in\mathsf{L}_1$ we have $|\mathscr{F}(a)|=6$;

\medskip
\item for all $a\in\mathsf{L}_2$ we have $|\mathscr{F}(a)|=2$;

\medskip
\item for all $a\in\mathsf{C}_1$ we have $|\mathscr{F}(a)|=1$, consisting of a point of fold type;
\medskip
\item for all $a\in\mathsf{C}_2$ we have $|\mathscr{F}(a)|=4$, two of which are non-degenerate, and two of which are degenerate of fold type;
\medskip
\item we have $|\mathscr{F}(a_{\rm cusp})|=2$, and the two points are of cuspidal type.
\end{enumerate}
\end{thm}
We note the above varieties and equations are invariant under involution $(y_1,y_2)\mapsto (y_2,y_1)$ which is the reflection across the diagonal. This is explained by the equivariant action of $\mathrm{Ad}_{\mathsf{w}}$ on $\Lambda_H^{\rm red} \to A$.

\subsection{Idea of proof}
We shall work with the traceless symmetric matrices $\mathfrak{p}$ and the $4$-dimensional subspace of tridiagonals $\mathcal{J}$, rather than $\mathfrak{p}^*$ and $\mathcal{J}^*$. We denote by $\mathcal{J}_+$ the open cone with positive entries on the first diagonal and we coordinatize $\mathcal{J}_+$ as
\begin{equation}\label{A(a)}
\mathcal{J}_+=\left\{\begin{pmatrix} \frac{1}{3}(2x_1+x_2) & y_1 & 0\\ y_1& \frac{1}{3}(x_2-x_1) & y_2\\ 0 & y_2 & -\frac{1}{3}(x_1+2x_2)\end{pmatrix}: x_1,x_2\in \R,\ y_1,y_2\in\R_{>0}\right\}.
\end{equation}
Define the natural map $\mathcal{J}_+\to A$ by $a=\Mdiag(y_1y_2,y_2,1)$.
We systematically use the coordinates on $A$ given by the positive simple roots $y_1=\alpha_1(a)$ and $y_2=\alpha_2(a)$, i.e. $a=\Mdiag(y_1y_2,y_2,1)$. 
In particular $e^{\langle \varpi_1, H(a) \rangle}=y_1^{\frac23}y_2^{\frac13}$ and 
$e^{\langle \varpi_2, H(a) \rangle}=y_1^{\frac13}y_2^{\frac23}$.
According to Theorem~\ref{t:critical} there is a canonical bijection between 
$\Ad_K(H)\cap\mathcal{J}_+$ and $\Lambda_H^{\rm red}$, and this bijection commutes with 
the two projection maps onto $A$.

Let $\mathscr{A}(a)$ denote the fiber over $a\in A$ under the above map $\mathcal{J}_+\rightarrow A$ (it is a $2$-dimensional affine space).
 Moreover  Theorem \ref{t:critical} provides an explicit description of $\mathscr{F}(a)$ inside $\mathscr{A}(a)$. Namely, if
\[
\chi_{\rm det}(a)=\{s\in\mathscr{A}(a): \det (s)=0\}\qquad\text{and}\qquad \chi_{\rm tr}(a)=\{s\in\mathscr{A}(a): {\rm Tr}(s^2)=2\},
\]
then
\begin{equation}\label{F(a)}
\mathscr{F}(a)=\chi_{\rm det}(a)\cap \chi_{\rm tr}(a).
\end{equation}
This is the starting point for studying $\mathscr{F}(a)$ and the partition $A=\mathsf{S}\sqcup\mathsf{L}\sqcup\mathsf{C}$.

From \eqref{A(a)} we get
\[
  \begin{aligned}
\chi_{\rm det}(a)&=\left\{9y_1^2(x_1+2x_2)-9y_2^2 (2x_1+ x_2)=2 (x_2^3- x_1^3)+3( x_1 x_2^2- x_1^2 x_2)\right\}\\
\chi_{\rm tr}(a)&=\left\{x_1^2+x_1x_2+x_2^2=3(1-y_1^2-y_2^2)\right\}.
\end{aligned}
\]
The idea of the proof of Theorem \ref{prop:CRIT} is to study the ``intersection configuration" of $\chi_{\rm det}(a)$ with $\chi_{\rm tr}(a)$ -- affine curves in $\mathscr{A}(a)$ of degree 3 and 2, respectively -- as $a$ varies throughout $A$. For $a$ near the origin they will intersect (transversally) in 6 points, and for $a$ large they will not intersect at all (for $\chi_{\rm tr}(a)$ will in fact be empty); these two extremal situations correspond to the regions $\mathsf{L}_1$ and $\mathsf{S}$. For intermediate ranges of $a$, transversal intersections will coalesce into points of tangency, before disappearing. This can happen in a few different ways, roughly corresponding to the ways in which a degree 6 polynomial can factorize over the reals. On the other hand, symmetry constraints will limit which combinations can arise. Once the intersection configuration has been mapped out, one can then read off the underlying singularity type by the numerical invariants recalled in \S\ref{num-inv}. 

\begin{figure}[h]
  \centering
\ifpdf
\begin{subfigure}[b]{0.25\textwidth}
  \centering
\includegraphics[width=\textwidth]{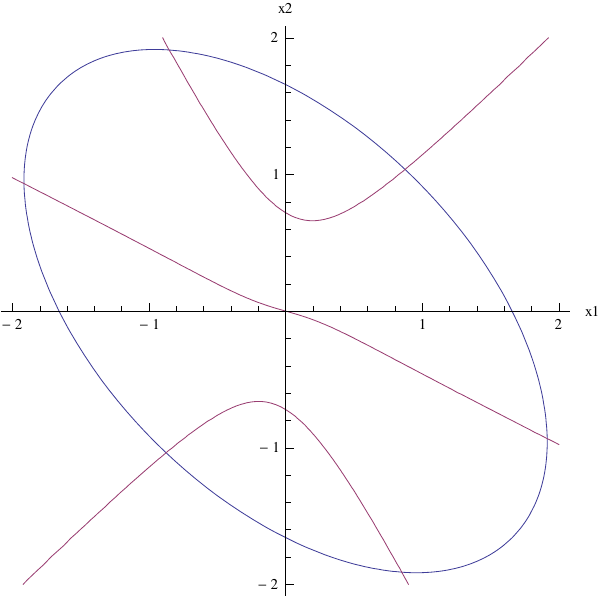}
\caption{Light $a\in \mathsf{L}_1$ \\  $y_1=.257$\ ; $y_2=.129$.}
\end{subfigure}
\qquad
\begin{subfigure}[b]{0.25\textwidth}
  \centering
\includegraphics[width=\textwidth]{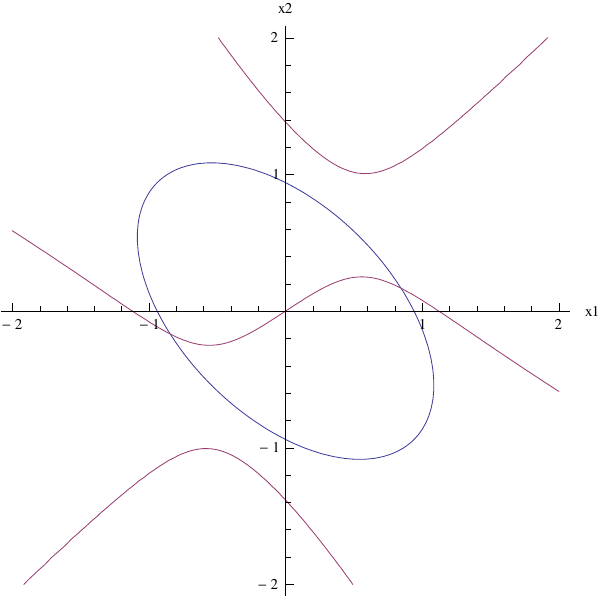}
\caption{Light $a\in \mathsf{L}_2$ \\ $y_1=.614$\ ; $y_2=.573$.}
\end{subfigure}
\qquad
\begin{subfigure}[b]{0.25\textwidth}
  \centering
\includegraphics[width=\textwidth]{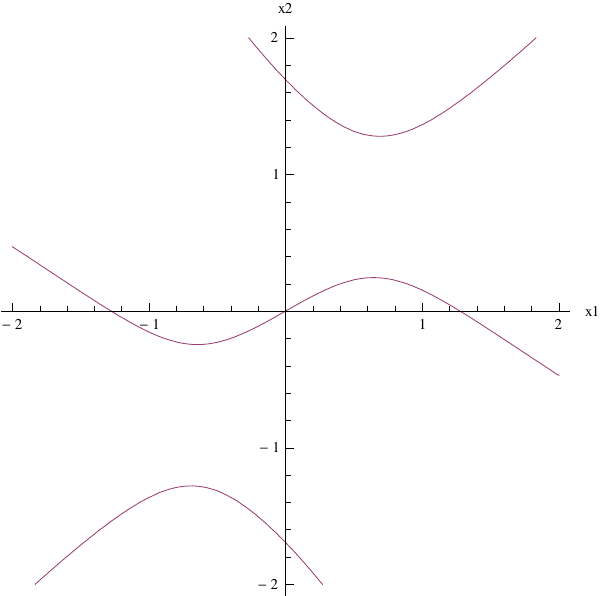}
\caption{Caustic $a\in \mathsf{C}_1$ \\ $y_1=.739$\ ; $y_2=.674$.}
\end{subfigure}

\begin{subfigure}[b]{0.25\textwidth}
  \centering
\includegraphics[width=\textwidth]{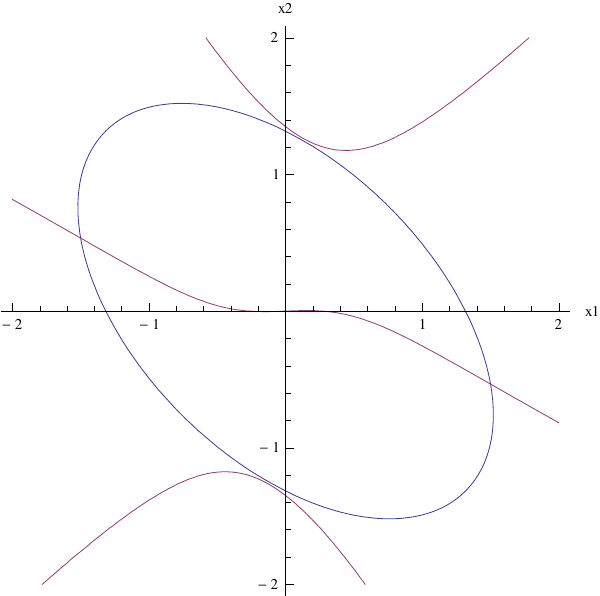}
\caption{Caustic $a\in \mathsf{C}_2$ \\ $y_1=.525$\ ; $y_2=.382$.}
\end{subfigure}
\qquad
\begin{subfigure}[b]{0.25\textwidth}
  \centering
\includegraphics[width=\textwidth]{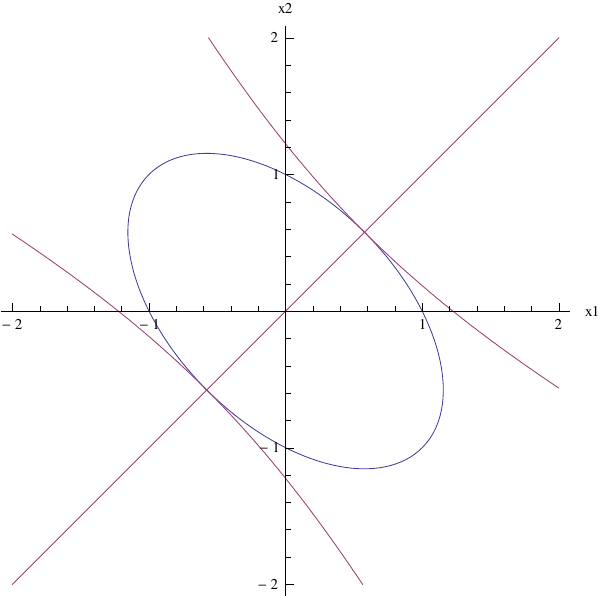}
\caption{Caustic $a= a_{\mathrm{cusp}}$ \\ $y_1=y_2=.57735$.}
\end{subfigure}
\else
\fig{intersections.}
\fi
\caption{The curves $\chi_{\rm det}(a)$ and $\chi_{\rm tr}(a)$ for different values of $a$}
\label{fig:intersections}
\end{figure}

In Figure \ref{fig:intersections}, we show five different intersection configurations corresponding to the five cases of Theorem \ref{prop:CRIT}. They can be mapped onto the corresponding strata of Figure \ref{beautiful pic}. Note that in the configuration (C) representing the outer caustic $\mathsf{C}_1$, the ellipse in (B) has collapsed to a single point; in the shadow zone $\mathsf{S}$ (not pictured) this point has disappeared. Compare Figure \ref{fig:intersections} to the classical bifurcation diagram of cuspidal singularities, as given, for example, in \cite[Figure 4]{CdV:singular-lagrangian}.

We illustrate the argument by carrying it out along the ray $y_1=y_2=y$ with $y>0$. In this case the equation for $\chi_{\rm det}(a)$ simplifies. Indeed the linear term $x_1-x_2$ factors, making $\chi_{\rm det}(a)$ the union of the line $x_1=x_2$ and the quadric hyperbola with equation
\[
9y^2 = 2x_1^2 + 5 x_1x_2 + 2x_2^2.
\]
The intersection with $\chi_{\rm tr}(a)$ can be easily computed and we find that the  different zones $a\in \mathsf{L}_1$, $a \in \mathsf{C}_2$, $a\in \mathsf{L}_2$, $a\in \mathsf{C}_1$, $a\in \mathsf{S}$ are given by the intervals
\[
0<y<\tfrac{1}{\sqrt{3}},\quad y=\tfrac{1}{\sqrt{3}},\quad \tfrac{1}{\sqrt{3}}<y<\tfrac{1}{\sqrt{2}},\quad y=\tfrac{1}{\sqrt{2}},\quad\text{and}\quad y>\tfrac{1}{\sqrt{2}},
\]
respectively, a result which agrees with Figure~\ref{beautiful pic} and Theorem~\ref{prop:CRIT}.

\subsection{The shadow zone}\label{sd-exact-locus}

In this section we establish the first statement in Theorem \ref{prop:CRIT} regarding shadow zone. We also establish the cardinality of the fibers in $\mathsf{C}_1$ and a lower bound in the fibers in the other regions.

\begin{prop}\label{sd-converse}
We have
\[
\mathsf{S}=\{y_1^2+y_2^2>1\}.
\]
Moreover, when $y_1^2+y_2^2=1$ there is one unique critical point, and if $y_1^2+y_2^2<1$ there are at least two distinct critical points.
\end{prop}

\begin{proof}
By \eqref{light-zone-ineq}, it suffices to prove that $\mathsf{S}\subseteq \{y_1^2+y_2^2>1\}$. According to \eqref{F(a)}, we must show that if $a\in A$ verifies $y^2_1+y^2_2\leq 1$ then $\chi_{\rm det}(a)\cap\chi_{\rm tr}(a)\neq\emptyset$. Any element in $\mathscr{A}(a)$ has norm-squared $2y_1^2+2y_2^2+d^2+e^2+f^2$, for some diagonal entries $d, e, f$. From the hypothesis $y_1^2+y_2^2\leq 1$ we deduce that $\chi_{\rm tr}(a)$ is not empty. It contains, say, the elements
\[
s_{\pm}=\pm {\rm diag}(\alpha,-\alpha,0)+\Upsilon (a), \qquad\text{where}\qquad \Upsilon (a)=\begin{pmatrix} 0 & y_1 & 0\\ y_1& 0 & y_2\\ 0 & y_2 & 0\end{pmatrix},
\]
for some $\alpha\geq 0$.

Note that $y_1^2+y_2^2=1$ if and only if $\chi_{\rm tr}(a)=\{\Upsilon(a)\}$. As $\Upsilon(a)$ has determinant $0$, we have $\Upsilon(a)\in\chi_{\rm det}(a)$ as required.

If $y_1^2+y_2^2<1$ then the points $s_{\pm}$ are distinct and $\det(s_{\pm})=\pm\alpha y_2^2$ are of opposite sign. Now $\chi_{\rm tr}(a)$, being an ellipse, is connected. By the intermediate value theorem, there is $s\in\chi_{\rm tr}(a)$ such that $\det(s)=0$. As the same is true of the antipode of $s$, there are at least two district points lying in $\chi_{\rm det}(a)\cap \chi_{\rm tr}(a)$.
\end{proof}

We illustrate the argument of the proof above with a graph
 in the $(x_1,x_2)$-plane of the two curves $\chi_{\rm det}(a)$ and $\chi_{\rm tr}(a)$. Let $s_i:= \alpha d_i+\Upsilon(a)$ where 
\begin{alignat*}{3}
  d_1:={\rm diag}(-1,1,0) \quad &  d_2:={\rm diag}(0,1,-1) \quad  & d_3:={\rm diag}(1,0,-1) \\
  d_4:={\rm diag}(1,-1,0) \quad &  d_5:={\rm diag}(0,-1,1) \quad & d_6:={\rm diag}(-1,0,1).
\end{alignat*}
Thus in particular $s_1=s_+$ and $s_4=s_-$. We have the property that $s_i\in\chi_{\rm tr}(a)$ for all $i=1,\ldots ,6$ and the points are cyclically ordered. Suppose by symmetry that $y_1>y_2$. Then it can be verified that $\det(s_i)>0$ if $i\in \{1,2,3\}$ while $\det(s_i)<0$ if $i\in \{4,5,6\}$. Thus there are at least two intersection points in $\chi_{\rm det}(a)\cap\chi_{\rm tr}(a)$ which was how we established Proposition~\ref{sd-converse}. We consider the following numerical values $y_1=.614$ and $y_2=.573$ in Figure~\ref{fig:s1-6}. 

\begin{figure}[h]
\ifpdf
\includegraphics[width=6cm]{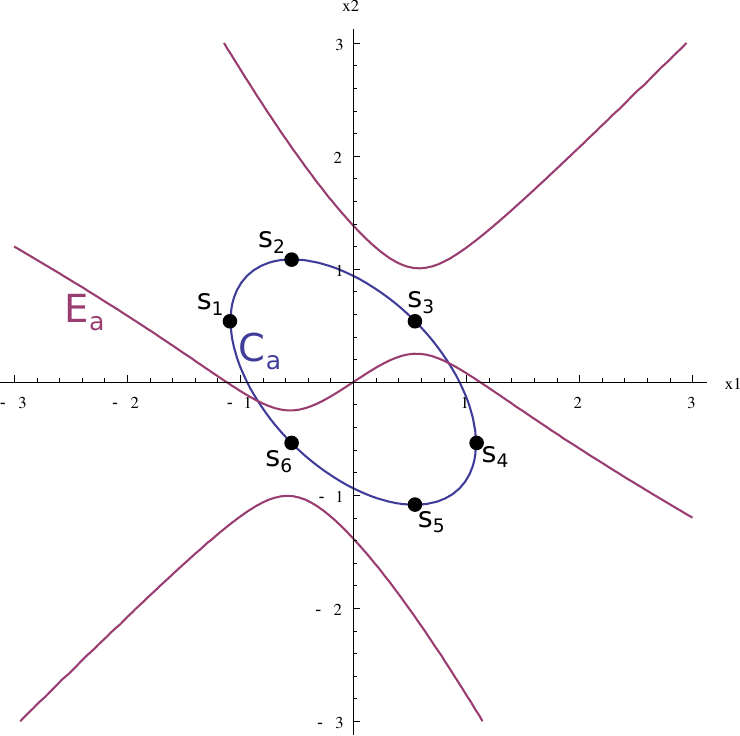}
\else
\fig{intersection-s1-6.pdf}
\fi
\caption{The curves $\chi_{\rm det}(a)$ and $\chi_{\rm tr}(a)$ for $a\in \mathsf{L}_2$.}\label{fig:s1-6}
\end{figure}

The case $y_1^2+y_2^2=1$ is even simpler since the ellipse collapses into a single point $\{\Upsilon(a)\}$ and the resulting configuration is shown in Figure~\ref{fig:intersections} on Page~\pageref{fig:intersections} with the numerical values $y_1=.739$ and $y_2=.673$. 

\subsection{Determination of the caustic locus}\label{s:n3:degenerate} 

For $a\in A$ consider the polynomials
\begin{align*}
C_a(x)&=x_1^2+x_1x_2+x_2^2-3 (1-y_1^2-y_2^2),\\
E_a(x)&=2x_1^3+3x_1^2x_2 +9x_1y_1^2+18x_2y_1^2- 2x_2^3-3x_1x_2^2-9x_2y_2^2-18 x_1y_2^2,\\
D_a(x)&=x_1y_1^2+x_2y_2^2 - x_1^2x_2-x_1x_2^2.
\end{align*}
The first two are the defining equations for the curves $\chi_{\rm tr}(a)$ and $\chi_{\rm det}(a)$, respectively. The role of the last one will be explained presently.

\begin{prop}\label{p:corank1}
We have $a\in\mathsf{C}$ if and only if there exists $x\in \R^2$ satisfying
\[
C_a(x)=E_a(x)=D_a(x)=0.
\]
Moreover, any singularity for $\Lambda_H^{\rm red}\rightarrow A$ (and hence any degenerate critical point of the Whittaker phase function $F_H$) is of corank $1$.\end{prop}

\begin{proof}
Using the coordinates \eqref{A(a)}, the existence of a solution to $C_a(x)=E_a(x)=0$ is equivalent to the fiber $\mathscr{F}(a)$ being non-empty. To characterize $a\in\mathsf{C}$ we must then, in view of \eqref{ker-Hess}, determine the $a$ for which there is $s\in\mathscr{F}(a)$ whose tangent space along $\Lambda_H^{\rm red}$ fails to be transverse to the vertical fiber $T_a^*A\simeq \mathfrak{a}$. If $\mathfrak{t}(s)$ denotes this tangent space, then this condition is equivalent to $\mathfrak{a}\cap\mathfrak{t}(s)\neq 0$. We shall show, again using the coordinates \eqref{A(a)}, that this is the same as the existence of a solution to $D_a(x)=0$.

Let $s\in \Ad_K(H)$ and write $T(s)$ for the tangent space of $s$ along the whole adjoint orbit $\Ad_K(H)$. Then we may identify $T(s)$ with $\{[k,s]: k\in\mathfrak{k}\}$. Now if $s\in \Ad_K(H)\cap \mathcal{J}$ then $\mathfrak{t}(s)$ may be identified with $T(s)\,\cap\,\mathcal{J}$. To compute this intersection explicitly we denote matrices in $\mathfrak{k}$ as
\begin{equation*}
k=\begin{pmatrix} 0 & b & c\\ -b& 0 & a\\ -c & -a & 0\end{pmatrix} \qquad (a,b,c\in\R).
\end{equation*}
Taking $s\in\Lambda_H^{\rm red}$, viewed as an element of ${\rm Ad}_K(H)\cap \mathcal{J}_+$ via the diagram \eqref{commutative} and with the coordinates of \eqref{A(a)}, and setting the upper right-hand entry of $[k,s]$ to zero, we find that $\mathfrak{t}(s)$ is the subspace of $T(s)$ cut out by the equation $-ay_1+by_2-c(x_1+x_2)=0$. Having computed $\mathfrak{t}(s)$, one then finds $\mathfrak{a}\cap\mathfrak{t}(s)$ by setting the off-diagonals of $\mathfrak{t}(s)$ to zero. This produces two extra linear constraints $cy_2-bx_1=0$, $ax_2+cy_1=0$. The determinant of this linear system is $D_a(x)$, which establishes the first claim.

To see the second claim, note that solutions to the above matrix equation precisely 
describe the kernel of the differential of the map $\Lambda_H^{\rm red}\rightarrow A$ at 
$s$ because it can identified as the intersection $\mathfrak{t}(s)\cap \mathfrak{a}$ of 
the tangent space and the vertical fiber. The corank $1$ property of singularities for 
this map is then evident since $y_1y_2\neq 0$. The link to the Whittaker phase function 
is made via \eqref{ker-Hess}.
\end{proof}

\begin{lem}\label{lem:C1} We have $\mathsf{C}_1\cup\{a_{\rm cusp}\}\subset\mathsf{C}$.
\end{lem}

\begin{proof} From Proposition \ref{sd-converse} it follows that every point of 
$\mathsf{C}_1$ is critical. To show that every $a\in \mathsf{C}_1$ is in fact degenerate, 
we note that the corresponding $s\in \mathscr{F}(a)$ has vanishing diagonal elements, so 
that equation $D_a(x)=0$ is trivially true.

Note that the symmetric matrices
\[
\splus= \frac{1}{\sqrt{3}} \begin{pmatrix} 1 & 1 & 0\\ 1 & 0 &1\\ 0 &1 & -1\end{pmatrix}\quad\text{and}\quad \sminus=\frac{1}{\sqrt{3}} \begin{pmatrix} -1 & 1 & 0\\ 1 & 0 &1\\ 0 &1 & 1\end{pmatrix}
\]
lie in $\Ad_K(H)$, for their characteristic polynomial $x-x^3$ is the same as that of $H$. Moreover, both $\splus$ and $\sminus$ lie in the affine subspace $\mathscr{A}(a_{\rm cusp})$. Thus both $\splus$ and $\sminus$ are critical points of $F_H$ over $a_{\rm cusp}$. Finally, $\splus$ and $\sminus$ verify the equation $D_a(x)=0$, which shows that they are degenerate. 
\end{proof}

Now observe that the equation $C_a(x)=0$ is that of a conic, which, if $y_1^2+y_2^2<1$, is not reduced to a point. We may therefore choose a birational map from $\mathbb{P}^1(\R)$ to its solution locus. We make the substitution
\begin{equation}\label{parametrization}
x_1=\frac{1-t^2}{1+t+t^2} \sqrt{3(1-y_1^2-y_2^2)},\qquad x_2=\frac{t(t+2)}{1+t+t^2} \sqrt{3(1-y_1^2-y_2^2)}.
\end{equation}
With this parametrization, the polynomials $E_a(x)$ and $D_a(x)$ become
\begin{align*}
E_a(t)&=y_1^2(t^2+4t+1)^3+y_2^2(t^2-2t-2)^3+2t^6+6t^5-15t^4-40t^3-15t^2+6t+2,\\
D_a(t)&=y_1^2(1-t^2)(t^2+4t+1)^2+y_2^2 (2t+t^2)(t^2-2t-2)^2+6t^5+15t^4-15t^2-6t,
\end{align*}
again under the hypothesis that $y_1^2+y_2^2<1$.

\begin{prop}\label{C-inside}
We have $\mathsf{C}\subset \mathsf{C}_1\sqcup \{a_{\rm cusp}\}\sqcup\mathsf{C}_2$.
\end{prop}

\begin{proof}
It suffices to show that if $a$ is in $\mathsf{C}$ but not in $\mathsf{C}_1$ then $a$ is in $\{a_{\rm cusp}\}\sqcup\mathsf{C}_2$.

We see that $a\in \mathsf{C}- \mathsf{C}_1$ satisfies $y_1^2+y_2^2<1$ and moreover there is $t\in \mathbb{P}^1(\R)$ such that $E_a(t)=D_a(t)=0$. This system has a {\it complex} solution $t\in\mathbb{P}^1(\C)$ if, and only if, the resultant $R(a)={\rm Res}(E_a(t),D_a(t))$ vanishes. One computes
\[
R(a)=(y_1^2+y_2^2-1)^4 (27 y_1^4 y_2^4-18 y_1^2 y_2^2+4 y_2^2+4 y_1^2-1)^2.
\]
The set of $a\in A$ such that $y_1^2+y_2^2<1$ and $R(a)=0$ is precisely $\{a_{\rm cusp}\}\sqcup \mathsf{C}_2$.
\end{proof}

Note that we have the relation
\[
D_a(t)= \left(\frac{-2t-1}{3}\right)E_a(t)+ \left(\frac{t^2+t+1}{9}\right) E_a'(t).
\]
From this it follows that
\begin{equation}\label{EE'}
\text{\it the real solutions of } E_a(t)=D_a(t)=0\; \text{\it are precisely those of } E_a(t)=E'_a(t)=0.
\end{equation}
This latter system is slightly more convenient, since it allows us to characterize degenerate critical points in terms of the multiplicities of roots of the polynomial $E_a(t)$. Note that the discriminant of $E_a(t)$ is proportional by an absolute constant to
\[
(y_1^2+y_2^2-1)^2 (27 y_1^4 y_2^4-18 y_1^2 y_2^2+4 y_2^2+4 y_1^2-1)^2,
\]
whose zero set agrees with the expression $R(a)$ above.

\subsection{Light configuration}\label{sub:light-config}
In this section we finish the proof of the light zone configuration in Theorem \ref{prop:CRIT}.

\begin{prop}\label{light-config}
We have the following critical point configurations:
\begin{enumerate}
\item\label{L1} for any $a\in\mathsf{L}_1$ one has $|\mathscr{F}(a)|=6$;
\item\label{L2} for any $a\in\mathsf{L}_2$ one has $|\mathscr{F}(a)|=2$.
\end{enumerate}
\end{prop}
\begin{proof}
Another way to state the proposition is that for $a\in\mathsf{L}_1$ (resp., $a\in\mathsf{L}_2$) there are $6$ (resp., $2$) distinct real solutions to $E_a(t)=0$.

Note that, for $i=1,2$, it is enough to show the stated root configuration for {\it some} value of $a\in\mathsf{L}_i$. Indeed, the root configuration cannot change within $\mathsf{L}_i$, since changing to any other root configuration would require hitting the caustic $\mathsf{C}$. By Proposition \ref{C-inside}, this is impossible.

For \eqref{L1} we can, for example, use $(y_1,y_2)=(1/2,1/2)$. In this case, equation $E_a(t)=0$ becomes $10 t^6+30 t^5-3 t^4-56 t^3-21 t^2+12 t+1=0$, which has 6 distinct real roots. For \eqref{L2} we can use the point $(y_1,y_2)=(\sqrt{3}/2\sqrt{2},\sqrt{3}/2\sqrt{2})$. In this case, we obtain $(2 t^2+2 t-1) (11 t^4+22 t^3+9 t^2-2 t+5)=0$, which has two distinct real roots.\end{proof}

As a corollary, we deduce the following result.

\begin{cor}\label{cor:real}
If $a\in\{a_{\rm cusp}\}\sqcup\mathsf{C}_2$, then any solution $t$ to $E_a(t)=0$ is real. In particular, we have $\mathsf{C}=\mathsf{C}_1\sqcup\{a_{\rm cusp}\}\sqcup \mathsf{C}_2$. 
\end{cor}
\begin{proof}
Suppose that for some $a\in\{a_{\rm cusp}\}\sqcup\mathsf{C}_2$ there is a pair of non-real, complex conjugate roots of $E_a(t)=0$. Then there exists a neighborhood $U$ of $a$ such that the same is true for every $a'\in U$. But this neighborhood necessarily intersects $\mathsf{L}_1$, where Proposition \ref{light-config} assures us that there are no complex roots. Contradiction.

The second statement follows from the proof of Proposition \ref{C-inside}.
\end{proof}

We note that $|\mathscr{F}(a)|$ is even for $a\in \mathsf{L}_1 \sqcup \mathsf{L}_2$. This is explained by the involution $(x_1,x_2)\to (-x_1,-x_2)$ which preserves $\chi_{\rm det}(a)$ and $\chi_{\rm tr}(a)$ above, and thus acts on the fibers $\mathscr{F}(a)$ for any $a\in A$. The only fixed points of the involution are $x_1=x_2=0$ which project to the caustic $\mathsf{C}_1$. In fact we shall see below that $|\mathscr{F}(a)|=1$ for every $a\in \mathsf{C}_1$ which is the only case where the multiplicity is odd.

\subsection{Degeneracy types} 
Having obtained the caustic locus in Corollary \ref{cor:real}, we now look at the fibers $\mathscr{F}(a)$ over caustic points. We first determine their multiplicities, which will be of great help in identifying their degeneracy type.

\begin{prop}\label{p:caustic-multiplicity}
We have the following critical point configurations:
\begin{enumerate}
\item\label{C1} for any $a\in\mathsf{C}_1$ one has $|\mathscr{F}(a)|=1$, of multiplicity 2;
\item\label{C2} one has $|\mathscr{F}(a_{\rm cusp})|=2$, each of multiplicity 3;
\item\label{C3} for any $a\in\mathsf{C}_2$ one has $|\mathscr{F}(a)|=4$, two of multiplicity 2, two of multiplicity 1.
\end{enumerate}
\end{prop}

\medskip

\noindent {\it Proof of \eqref{C1}}: We have already proved \eqref{C1} in Proposition \ref{sd-converse}.

\bigskip

\noindent {\it Proof of \eqref{C2}}: By \eqref{EE'} we must show that $E_{a_{\rm cusp}}(t)=0$ admits two distinct real roots, each of multiplicity 3. Inserting $(y_1,y_2)=(1/\sqrt{3},1/\sqrt{3})$ into the formula for $E_a(t)=0$ we obtain $(2t^2+2t-1)^3=0$.

\bigskip

\noindent {\it Proof of \eqref{C3}}: Let $a\in\mathsf{C}_2$. By \eqref{EE'} we must show that $E_a(t)=0$ admits four distinct real roots, of which two are double and two are simple. 

We will make use of the symmetry of the solution locus $C_a(x)=E_a(x)=D_a(x)=0$ given by $x\mapsto -x$. In the parametrization \eqref{parametrization}, this corresponds to $\sigma(t)= (t+2)/(-2t-1)$. We deduce that if $t\in\mathbb{P}^1$ satisfies $E_a(t)=D_a(t)=0$, then so does $\sigma (t)$. We deduce from \eqref{EE'} that the system $E_a(t)=E_a'(t)=0$ is also invariant under $\sigma$. In other words, $\sigma$ sends roots of $E_a(t)$ to roots of $E_a(t)$, and conserves their multiplicities.

By Corollary \ref{cor:real}, $E_a(t)$ admits 6 real roots, when counted with multiplicity. Since $\mathsf{C}_2\subset\mathsf{C}$, one of these roots must have multiplicity strictly greater than $1$. Since $a\notin\mathsf{C}_1$, any solution $x$ to $E_a(x)=0$ is non-zero, so that the map $x\mapsto -x$, and hence $\sigma$, has no fixed points. From these observations we deduce that either two roots are of multiplicity 2 and the others are non-degenerate (as is stated in the proposition) or that there are $2$ distinct real solutions, each with multiplicity $3$. We must show that for $a\in\mathsf{C}_2$ the latter cannot occur.

Recall from \cite{Coste} the notion of the principal subresultant coefficients ${\rm PSPC}_\ell (P,Q)$. These can be used to characterize the exact number of
roots a given polynomial has. For example, a degree $6$ polynomial $P$ has exactly $2$ distinct complex roots if, and only if, ${\rm PSPC}_4 (P,P')\neq 0$ and ${\rm
PSPC}_\ell (P,P')=0$ for $\ell=0,1,2,3$. If we show that the vanishing locus of ${\rm PSPC}_3(E_a,E_a')$ does not intersect $\mathsf{C}_2$, then this effectively
eliminates this root configuration. 

Recall that $a\in\mathsf{C}_2$ satisfies $F(x,y)=0$, where $F(x,y)=27 x^2y^2-18xy+4x+4y-1$ in the variables $x=y_1^2$, $y=y_2^2$. We furthermore compute ${\rm PSPC}_3(E_a,E_a')=G(x,y)$, where
\begin{align*}
G(x,y)=80(x+y)&-50(x^2+y^2)+7(x^3+y^3)-51(x^2y+xy^2)\\
&+57(x^3y+xy^3)+249x^2y^2-166xy-25.
\end{align*}
The resultant with respect to $y$ of $F$ and $G$ is
\[
27(y+1)^4 (3y-1)^2 (3y^2-3y+1) (1323y^4 - 1809y^3 + 1602y^2 - 765y + 133).
\]
The only real solution in the positive reals is $y=1/3$. We conclude by recalling that $a\in\mathsf{C}_2$, by definition, is distinct from $(1/\sqrt{3},1/\sqrt{3})$.
\qed

\bigskip

We may now determine the degeneracy type of each of the degenerate singularities lying over a caustic point.  

\begin{cor}\label{p:degeneracy} 
We have the following description of the degeneracy types in the critical locus:
  \begin{enumerate} 
	\item\label{D1} For any $a\in \mathsf{C_1}$, the unique critical point of $\mathscr{F}(a)$ is degenerate of type $A_2$.

	\item\label{D2} The two distinct critical points of $\mathscr{F}(a_{\rm cusp})$ are degenerate of type $A_3$.
  
	\item\label{D3} For any $a\in \mathsf{C}_2$, the two degenerate critical points of $\mathscr{F}(a)$ are of type $A_2$.
   \end{enumerate}
	\end{cor}

	\begin{proof}[Proof of \eqref{D1}]
In Proposition~\ref{p:caustic-multiplicity} it was shown that the multiplicity is $2$. This is enough to pinpoint $A_2$ as the degeneracy type, since a singularity of type $A_k$ has multiplicity $k$. \end{proof}

\begin{proof}[Proof of \eqref{D2}]
In Lemma~\ref{lem:C1}, we found the two critical points $u_{\rm cusp}^\pm$ and showed in Proposition~\ref{p:corank1} that the corresponding Hessians $\nabla^2 F_H(u_{\rm cusp}^\pm, a_{\rm cusp})$ are both of corank $1$. It follows that $u_{\rm cusp}^\pm$ are of degeneracy type $A$. By Proposition~\ref{p:caustic-multiplicity} the multiplicity of both $u_{\rm cusp}^\pm$ is $3$. Hence the critical points are of type $A_3$.
\end{proof}

\begin{proof}[Proof of \eqref{D3}] If $a\in \mathsf{C_2}$, then according to Proposition~\ref{p:caustic-multiplicity} among the four distinct critical points of $\mathscr{F}(a)$ two are non-degenerate and two are degenerate of multiplicity $2$. Since $A_2$ is the unique singularity class with multiplicity $2$, we deduce that the two degenerate critical points in $\mathscr{F}(a)$ are fold singularities.
\end{proof}

\begin{remark}
In the proof of~\eqref{D2} above we could bypass the use of Proposition~\ref{p:corank1} and only use the fact that the multiplicity of the singularity is $3$. Indeed this implies that the singularity is simple~\cite[Lemma~4.2]{A2}, and the classification theorem of Arnol'd then shows that it is of type $A_3$. 
\end{remark}

\subsection{Proof of Theorem ~\ref{cor:Pearcey}, conditional on Hypothesis~\ref{hypothesis}}

We derive from the critical point 
configuration described in Theorem \ref{prop:CRIT} a 
lower bound on the $\PGL_3(\R)$ Jacquet-Whittaker functions $W_\nu(a)$ in the vicinity of 
the cuspidal point $a_{\rm cusp}=\Mdiag(\frac13,\frac{1}{\sqrt{3}},1)$. Namely, we shall show the existence of a constant $C>0$ and a neighborhood $V$ of $a_{\rm cusp}$, such that the 
following property holds: for all $t>1$ there is $a\in V$ such that $|W_{t\nu_0}(ta)|\geq 
Ct^{3/4}$. The approach described here depends on the verification of Hypothesis~\ref{hypothesis}; we give an unconditional proof in the next section. 

Denote by $u_{\rm cusp}^\pm$ the two singular points of $u\mapsto \varphi(u,a_{\rm 
cusp})$ given in part \eqref{D2} of Corollary \ref{p:degeneracy}. We analyse the oscillatory integrals
\begin{equation*}
I_\pm(t,a):=(2\pi)^{-\frac32}\int_U  \delta (\mathsf{w}u)^{1/2} 
\alpha_\pm(u)e^{it\varphi(u,a)}du,
\end{equation*} 
where $\varphi(u,a)=B(H,H({\sf w}u))-\ell_1(aua^{-1})$, and $\alpha_+(u)$ and 
$\alpha_-(u)$ are smooth 
functions of disjoint compact support which are identically 1 in a neighborhood of 
$u_{\rm cusp}^+$ and $u_{\rm cusp}^-$, respectively.
Recall from 
\S\ref{8not} that 
$H={\rm diag}(1,0,-1)$. 

Let $Q^\pm$ denote the Hessian of $\varphi$ at $(u_{\rm cusp}^\pm,a_{\rm cusp})$; if $(p^
\pm,q^\pm)$ is its signature, write $\sigma^\pm=p^\pm-q^\pm$.

Shrinking the support of $\alpha_\pm$ if necessary, Lemma \ref{degen F2} below provides 
for non-zero functions $r_1^\pm,r_2^\pm,\alpha^\pm_j,s_\pm\in C^\infty_c(A)$, with 
$r_i^\pm(a_{\rm 
cusp})=0$ and $s_\pm(a_{\rm cusp})=\varphi(u_{\rm cusp}^\pm,a_{\rm cusp})$, such that
\[
I_\pm (t,a)=e^{i\pi\sigma^\pm/4} e^{its_\pm(a)}t^{-\frac{5}{4}}\sum_{j=0,1,2}
\alpha_j^\pm(a)\mathsf{Pe}_j\left(t^{3/4}r_1^\pm(a),t^{1/2}r_2^\pm(a)\right)
t^{-\frac{j}{4}}+O(t^{-9/4}),
\]
uniformly for all $a$ in a sufficiently small compact $V$ about $a_{\rm cusp}$ and all 
$t>1$. 
If $a$ is sufficiently close to $a_{\rm cusp}$, then we have
\[
I_\pm (t,a)=
e^{i\pi\sigma^\pm/4} 
e^{its_\pm(a)}
\alpha_0^\pm(a)\mathsf{Pe}_0\left(t^{3/4}r_1^\pm(a),t^{1/2}r_2^\pm(a)\right)
t^{-5/4}+O(t^{-3/2}).
\]
We note for later use that, using Lemma \ref{degen F2}, and the definition of 
$\alpha_\pm$, we have
\begin{equation}\label{alpha0}
\alpha_0^\pm(a_{\rm cusp})>0.
\end{equation} 
From \eqref{def-Pearcey} and \eqref{alpha0} it follows that 
$\mathsf{Pe}_0\left(t^{3/4}r_1^\pm(a),t^{1/2}r_2^\pm(a)\right)$ and $\alpha_0^\pm (a)$ 
are non-vanishing for $a$ sufficiently close to $a_{\rm cusp}$.

We shall show that there is $C>0$ such that for every $t>1$ there is $a\in V$ 
satisfying
\begin{equation}\label{eq:cusp-lower-bd}
|I_+(t,a)+I_-(t,a)|\geq C t^{-5/4}.
\end{equation}
Let $P(a,t):=e^{i\pi(\sigma^+ -\sigma^-)} \frac{\alpha_0^+(a)}{\alpha_0^-(a)} 
\frac{\mathsf{Pe}_0\left(t^{3/4}r_1^+(a),t^{1/2}r_2^+(a)\right)}
{\mathsf{Pe}_0\left(t^{3/4}r_1^-(a),t^{1/2}r_2^-(a)\right)}$. 
For $a$ in a ball of radius $O(t^{-\frac34})$ about $a_{\rm cusp}$,
\[
|I_+(a,t)+I_-(a,t)|\asymp |1+e^{it S(a)}P(a,t)|t^{-5/4}+O(t^{-3/2}),
\]
where $S(a):=s_+(a)-s_-(a)$. In order to prove \eqref{eq:cusp-lower-bd} it suffices to show that for every $t>1$, 
 there exists $a$ close to $a_{\rm cusp}$ such that 
\begin{equation}\label{eq:lower-phase-osc}
|1+e^{it S(a)}P(a,t)|\geq 1.
\end{equation} 

Write $\theta(a,t)$ for the complex argument of $P(a,t)$, and let $\theta:=\theta(a_{\rm 
cusp},t)$ which is independent of $t$. 
It suffices to show that for every $t\gg 1$, there exists $a$ close to $a_{\rm cusp}$ 
such that
\[
\theta(a,t)+tS(a)\in [-\pi/2,\pi/2]\;\text{ mod } 2\pi.
\]
Indeed this implies that $|1+e^{i(\theta(a,t) + t S(a))} \rho|\geq 1$, for any $\rho\ge 
0$, as desired. 
If $a$ is sufficiently close to $a_{\rm cusp}$, then $|\theta(a,t)-\theta|\le 
\frac{\pi}{4}$. Thus to establish~ \eqref{eq:lower-phase-osc} it suffices to show 
that for every $t\gg 1$, there exists $a$ close to $a_{\rm cusp}$ such that
\begin{equation}\label{tSa}
tS(a)\in [-\theta-\pi/4,-\theta+\pi/4]\;\text{ mod } 2\pi.
\end{equation}

We use the fact that the differential of $S$ is non-zero:
\[
\nabla S(a_{\rm cusp})=\nabla s_+(a_{\rm cusp})-\nabla s_-(a_{\rm cusp})=\nabla 
\varphi(u_{\rm cusp}^+,a_{\rm cusp})-\nabla \varphi(u_{\rm cusp}^-,a_{\rm cusp})\neq 0,
\]
which is a consequence of Theorem~\ref{prop:CRIT} which says that $|\mathscr{F}(a_{\rm 
cusp})|=2$. 
Thus there exists some $X\in\mathfrak{a}$ for which the directional derivative $\nabla_X 
S(a_{\rm cusp})$ is non-zero. The image of 
the ball of radius $O(t^{-\frac{3}{4}})$ about 
$a_{\rm cusp}$ under 
 $a\mapsto S(a)$ contains an interval of length $\gg t^{-\frac34}$ around $S(a_{\rm 
 cusp})$. Thus 
the image of $tS(a)$ contains an interval of length $\gg t^{\frac14}$. This establishes 
\eqref{tSa}.

Recalling \eqref{eq:cusp-lower-bd} we obtain $\left|I(t,a)\right|\gg t^{-5/4}$
for some $a\in V$. Since $a=\Mdiag(y_1y_2,y_2,1)$, and 
$ta=\Mdiag(t^2y_1y_2,ty_2,1)$
 thus
 $\delta(a)^{1/2}=t^2y_1y_2$, the formula \eqref{W-form2} reads
\[
|W_{t\nu_0}(ta)|=t^2y_1y_2\left|I(t,a)\right|.
\]
Since we can choose $V$ such that for all $a\in V$ we have $|y_1y_2 - \frac{1}{3}|$ 
arbitrarily small, for example bounded by $\frac16$, the claimed lower bound follows.

\begin{remark}\label{rem-S-cusp}
We indicate how to find the difference of critical values $S(a_{\rm cusp})$, which is 
useful for obtaining uniform asymptotics 
in the transition region, as in Theorem \ref{thm:large-values-n=3} below.
If $u=\left(\begin{smallmatrix}1 & u_1 & u_3\\ 0 & 1 & u_2\\ 0 & 0 & 
1\end{smallmatrix}\right)$, $u_4:=u_1u_2-u_3$, and $a={\rm diag}(y_1y_2,y_2,1)$ then
\[
(2\pi)^{-1}\varphi(u,a)=-y_1u_1-y_2u_2-\frac{1}{2}\log (1+u_1^2+u_4^2)-\frac{1}{2}\log 
(1+u_2^2+u_3^2).
\]
Then the critical set $\Sigma$ is defined by the system
\[
y_1= -\frac{u_1+u_2u_4}{1+u_1^2+u_4^2},\quad y_2=-\frac{u_2+u_1u_3}{1+u_2^2+u_3^2},\quad 
\frac{u_3}{1+u_2^2+u_3^2}= \frac{u_4}{1+u_1^2+u_4^2}.
\]
Recall that $a_{\rm cusp}$ corresponds to $y_1=y_2=1/\sqrt{3}$. One then verifies that  
\[
u_{\rm cusp}^\pm=\left(\begin{smallmatrix} 1 & \mp 1-\sqrt{3} & 2\pm\sqrt{3}\\ 0 & 1 & 
\mp 1-\sqrt{3}\\ 0 & 0 & 1\end{smallmatrix}\right),
\]
are the two critical points, i.e. 
$(u_{\rm cusp}^\pm,a_{\rm cusp})$ lie in $\Sigma$. A direct computation then shows that
\[
(2\pi )^{-1}S(a_{\rm cusp})=(2\pi)^{-1}s^+(a_{\rm cusp})
-
(2\pi)^{-1}s^-(a_{\rm cusp})
=
\frac{4}{\sqrt{3}} - 2 \log(2+\sqrt{3}).
\]
\end{remark}

\section{Proof of Theorem~\ref{cor:Pearcey}}\label{LOWER-GL3}
We continue to assume that $\nu$ is self-dual and retain the notation 
from \S\ref{8not}. Recall that regular self-dual spectral parameters for $\GL_3$ are 
given by $t\nu_0$ for 
$t>0$ and $\nu_0=2\pi i(\varpi_1+\varpi_2)$, and
the cuspidal point $a_{\rm cusp}=\Mdiag(\frac13,\frac{1}{\sqrt{3}},1)$.
The following theorem provides a lower
bound for $W_{t\nu_0}(ta)$ for $a$ close to $a_{\rm cusp}$. Since $\|W_\nu\|_2=1$, this 
will complete the proof of Theorem~\ref{cor:Pearcey}. 

\begin{thm}\label{n=3 uniform statement} 
There exist a constant $C>0$ and a neighborhood $V$ of $a_{\rm cusp}$, such that the 
following property holds: for all $t>1$ there is $a\in V$ such that $|W_{t\nu_0}(ta)|\geq 
Ct^{3/4}$.
\end{thm}

\subsection{Asymptotics associated to cuspidal singularities}\label{sec:asymp-cusp}
We give here a more precise description of $A_k$ singularities than that given in 
\S\ref{num-inv}. 
\begin{defn}[\cite{AV}, Part  II, \S 11.1]\label{A2 def} Let $x_c$ be 
a critical point of a function $\varphi\in C^\infty(\R^m,\R)$ and $Q$ be the 
Hessian quadratic form. We say that $x_c$ is an $A_k$ 
singularity, $k\ge 2$, if
\begin{enumerate}
\item\label{A1} $Q$ has corank $1$, and
\item\label{A2} in a neighborhood of $x_c$, the function $x\mapsto \varphi(x)$ is 
equivalent (via a germ of diffeomorphism of $\C^m$) to
\[
x_1^2+x_2^2+\cdots+x_{m-1}^2+x_m^{k+1}.
\]
\end{enumerate}
\end{defn}

The Airy function is the first in a series of special functions associated to 
singularities of type $A_k$. The {\it generalized Airy function of order $k$} 
is defined as
\[
Ai_k(y_1,\ldots ,y_{k-1})=\int_\R \exp\left(i\left(y_{k-1}x+\cdots 
+y_1\frac{x^{k-1}}{k-1}+\frac{x^{k+1}}{k+1}\right)\right)dx,
\]
the integral is improper (converges in the limit but not absolutely). For $k=2$ we 
recover the Airy function: $Ai_2(y)=(2\pi)\mathsf{Ai}(y)$. In general, the Airy function 
of order $k$ governs the asymptotic behavior of families of oscillatory integrals whose 
phase 
functions have $A_k$ type singularities. For more information on generalized Airy 
functions we refer the reader  to \cite[\S 7.9]{G-S}.

In this section we shall be interested in cusp singularities. The Airy function of order 
$3$ bears a special name: one calls
\begin{equation*}
\mathsf{Pe}(y_1,y_2)=Ai_{3}(y_1,y_2)=\int_\R 
\exp\left(i\left(y_2x+y_1\frac{x^2}{2}+\frac{x^4}{4}\right)\right)dx
\end{equation*}
the {\it Pearcey function}. It was first introduced (and numerically computed) in 
\cite{Pe}. Below we shall need the fact that
\begin{equation}\label{def-Pearcey}
\mathsf{Pe}(0,0)=\Gamma(\tfrac14) e^{i\pi/8}/\sqrt{2}\neq 0.
\end{equation}
Note that unlike the Airy function $\mathsf{Ai}(y)$, the Pearcey function (and indeed all 
higher order Airy functions) is a complex valued function. 

The phase function
\[
\varphi(x,y)=y_2x+y_1\frac{x^2}{2}+\frac{x^4}{4},\qquad (x\in\R,\; y=(y_1,y_2)\in\R^2),
\]
has critical set $\{(x,y)\in\R\times\R^2: y_2+xy_1+x^3=0\}$. The horizontal projection is 
singular at $y=0$. The image of the critical set under the map (analogous to that in 
\eqref{immersion}) which sends $(x,y)\in\R\times\R^2$ to $(y,\partial_y \varphi(x,y))\in 
T^*\R^2$ is the Lagrangian surface $\{(y_1,y_2; x^2/2,x)\in T^*\R^2: y_2+xy_1+x^3=0\}$.

The following lemma allows one to reduce the asymptotic behavior of an oscillatory 
integral whose phase function has a cusp singularity to the asymptotic behavior of the 
Pearcey function. For notational convenience, we set $\mathsf{Pe}_0:=\mathsf{Pe}$, 
$\mathsf{Pe}_1:=\partial_{y_1}\mathsf{Pe}$, and 
$\mathsf{Pe}_2:=\partial_{y_2}\mathsf{Pe}$.

\begin{lem}\label{degen F2} Let $\varphi\in C^\infty(\R^m\times \R^n,\R)$, and $y_c\in 
\R^n$ be 
such that 
$x\mapsto \varphi(x,y_c)$ admits a critical point $x_c\in \R^m$, which is of singularity 
type $A_3$ according to Definition \ref{A2 def}. Let $Q$ denote the Hessian quadratic 
form, and write $\sigma=p-q$, where $(p,q)$ is the signature of $Q$.
There exist
\begin{enumerate}
\item a compact neighborhood $K'\times K\subset \R^m\times \R^n$ of $(x_c,y_c)$;
\item real valued functions $r_1,r_2,s\in C^\infty_c(\R^n)$, satisfying 
$r_1(y_c)=r_2(y_c)=0$, $s(y_c)=\varphi (x_c,y_c)$, and $\nabla s(y_c) = \nabla_y 
\varphi(x_c,y_c)$;
\end{enumerate}
such that for all $\alpha\in C^\infty_c(\R^m)$ with support in $K'$ and all $y\in K$ and 
$t \geq 1$ 
the integral
\begin{equation}\label{appendix int}
\left(\frac{t}{2\pi}\right)^{\frac{m}{2}}\int_{\R^m} \alpha(x)e^{it \varphi(x,y)}dx
\end{equation}
is equal to
\[
e^{i\pi\sigma/4}e^{its(y)}t^{\frac{1}{4}}
\sum_{j=0,1,2}t^{-\frac{j}{4}}\alpha_j(y)
\mathsf{Pe}_j\left(t^{3/4}r_1(y),t^{1/2}r_2(y)\right)
+O\left(t^{-\frac34}\right),
\]
for some functions $\alpha_j\in C^\infty_c(\R^n)$. Moreover, we have
\[
\alpha_0(y_c)=(2\pi)^{-1/2} 
(\partial_w 
W(w_c,y_c))^{-1}
 |\det\nolimits' Q|^{-1/2} \alpha(x_c),
\]
where $\det\nolimits' Q$ is the product of all non-zero eigenvalues of $Q$, and the 
function $W$ is specified in the proof of the lemma.
\end{lem}

\begin{proof}
The proof proceeds in the same way as the analogous result \cite[Theorem 7.7.19]{Ho} for 
fold singularities.

Let $V_0$ denote the orthogonal complement inside $\R^m$ of $\ker Q$. For any $(x,y)\in 
\R^m\times \R^n$ let $Q_0(x,y)$ denote the matrix of second order partial derivatives of 
$\varphi$ relative to a basis of $V_0$. One writes the integration domain as 
$\R^m=\R^{m-1}\times \ker Q$, where $\R^{m-1}=V_0$. We write the generic 
element of $\R^{m-1}$ as $v$ and the generic element of $ \ker Q$ as $w$. The 
coordinates of the critical point $x_c\in \R^m$ are denoted $(v_c,w_c)\in \R^{m-1}\times 
\ker Q$.

Keeping $w\in \ker Q$ fixed, the leading term asymptotic of the integral
\[
\left(\frac{t}{2\pi}\right)^{\frac{m-1}{2}}\int_{\R^{m-1}} 
\alpha(v,w)e^{it\varphi(v,w,y)}dv
\]
is given by
\[
e^{i\pi\sigma/4}e^{it\varphi (\overline{v}(w,y),w,y)}|\det 
Q_0(\overline{v}(w,y),w,y)|^{-1/2}\alpha (\overline{v}(w,y),w),
\]
where $\overline{v}(w,y)$ is the unique critical point of $v\mapsto \varphi(v,w,y)$.
(For this, see \cite[Theorem 7.7.6]{Ho} or \cite[Theorem 
2.9]{Var}). One then applies the $A_3$ stationary phase lemma to the remaining 
one-dimensional integral 
\[
\left(\frac{t}{2\pi}\right)^{\frac12}\int_{\ker Q}\beta(w,y)e^{it \phi(w,y)} dw.
\]
Here we have written $\beta(w,y)=|\det Q_0 (\overline{v}(w,y),w,y)|^{-1/2}\alpha 
(\overline{v}(w,y),w)$ and $\phi(w,y)=\varphi (\overline{v}(w,y),w,y)$. For the $A_3$ 
asymptotic, see \cite[\S 7, Theorem 9.1]{G-S} or \cite[(3.12)]{KKM}; moreover, one can 
 adapt \cite[Theorem 7.7.18]{Ho} to the case of cusp singularities.

The result is a leading term asymptotic of the form specified in Lemma \ref{degen F2}, 
{\it but without the explicit expression for $\alpha_0(y_c)$}. Indeed, in these 
references 
no formula for $\alpha_0(y_c)$ is given. One can, however, extract this value from the 
proof of \cite[Theorem 7.7.18]{Ho}. We indicate how to do so now.

The Malgrange preparation theorem~\cite[Theorem 7.5.13]{Ho}, when applied to our phase 
function $\phi$, shows the existence of real valued functions
\begin{enumerate}
\item $W\in C^\infty(\ker Q\times \R^n)$ satisfying $W(w_c,y_c)=0$, $\partial_w 
W(w_c,y_c)>0$, 
\item $r_1,r_2,s\in C_c^\infty(\R^n)$ satisfying $r_1(y_c)=r_2(y_c)=0$ and 
$s(y_c)=\phi(w_c,y_c)$, 
\end{enumerate}
such that
\[
\phi(w,y)=\frac{W^4}{4}+r_2(y)\frac{W^2}{2}+r_1(y)W+s(y)
\]
in a neighborhood of $(w_c,y_c)$. 
We see that $\nabla s(y_c) = \nabla_y \phi(w_c,y_c)=\nabla_y \varphi(v_c,w_c,y_c)$ 
because 
$\nabla_v \varphi(v_c,w_c,y_c)=0$ since $v_c$ is a critical point.

Concerning our amplitude function $\beta$, a slightly 
different 
version of the Malgrange preparation theorem~\cite[Theorem 7.5.6]{Ho} shows the existence 
of functions $q\in C^\infty(\ker Q\times \R^n)$ and $A_0,A_1,A_2\in C^\infty(\R^n)$, 
verifying
\[
(\partial_w W(w,y))^{-1}
\beta(w,y)=(W^3+r_2(y)W+r_1(y))q(w,y)+A_2(y)W^2+A_1(y)W+A_0(y)
\]
in a neighborhood of $(w_c,y_c)$. Following the argument of H\"ormander in \cite[Theorem 
7.7.18]{Ho}, one sees that the leading term asymptotics for the $\R^m$-integral are given 
by
\begin{align*}
\sum_{j=0,1,2}&(2\pi)^{-1/2}e^{i\pi\sigma/4}e^{its(y)}A_j(y)\int_{\ker Q} 
W^je^{it\left(\frac{W^4}{4}+r_2(y)\frac{W^2}{2}+r_1(y)W\right)} dW\\
&=\sum_{j=0,1,2}
(2\pi)^{-1/2}e^{i\pi\sigma/4}e^{its(y)}
A_j(y)t^{\frac{1-j}{4}}\mathsf{Pe}_j(t^{3/4}r_1(y),t^{1/2}r_2(y)).
\end{align*}
The functions $(2\pi)^{-1/2} A_j$ are the $\alpha_j$ appearing in the statement of Lemma 
\ref{degen F2}.
One computes the value of each $\alpha_j(y_c)$ by evaluating 
$\partial^j_y\beta(w_c,y_c)$. For example,
\begin{align*}
(2\pi)^{\frac12} \alpha_0(y_c)=\frac{ \beta(w_c,y_c)}{\partial_w 
W(w_c,y_c)}
&=(\partial_w 
W(w_c,y_c))^{-1}
 |\det Q_0 (\overline{v}(w_c,y_c),w_c,y_c)|^{-1/2}\alpha 
(\overline{v}(w_c,y_c),w_c)\\
&=
(\partial_w 
W(w_c,y_c))^{-1}
 |\det Q_0|^{-1/2}\alpha (x_c).
\end{align*}
This proves the lemma.
\end{proof}

\subsection{Takhtadjan-Vinogradov formula}
A formula of Takhtadjan--Vinogradov yields an integral representation of $W_\nu$ 
involving the product of two $\GL(2)$-Bessel functions. We will use below this integral 
representation to establish Theorem \ref{n=3 uniform statement}.
As a first step towards this, we now briefly explain how one can use the Takhtadjan--Vinogradov formula in the self-dual case to recover
\begin{enumerate}
\item the description of the shadow zone $\mathsf{S}$ as described in Proposition \ref{sd-exact-locus};
\item the caustic line $\mathsf{C}_1$, as described in Theorem \ref{prop:CRIT},
\end{enumerate}
both of which were obtained using the $\GL(3)$ Jacquet integral.

\begin{prop}[Takhtadjan--Vinogradov] For every $a=\operatorname{diag}(y_1y_2,y_2,1)\in 
A$, and $\nu=i(t_1 \varpi_1 + t_2 
\varpi_2)\in i\mathfrak{a}^*$, the spectrally normalized Whittaker function 
\[
\Gamma_\R(1+it_1)\Gamma_\R(1+it_2)\Gamma_\R(1+it_1+it_2) W_\nu(a)
\]
is equal to
\begin{equation*}
y_1^{1+i\frac{t_1-t_2}{6}}y_2^{1+i\frac{t_2-t_1}{6}}
\int_0^\infty
K_{i\frac{t_1+t_2}{2}}(2\pi y_1 \sqrt{1+u}) K_{i\frac{t_1+t_2}{2}}(2\pi y_2 
\sqrt{1+u^{-1}}) u^{\frac{i(t_1-t_2)}{4}} \frac{du}{u}.
\end{equation*}
\end{prop}

If $t_1=t_2$, then the term $u^{\frac{i(t_1-t_2)}{4}}$ is identically 1. 
Up to a bounded multiplicative constant, we have for every $t\in \R_{>0}$ and $a\in A$,
\begin{equation}\label{TV-formula}
W_{t\nu_0}(a) \asymp
y_1y_2 e^{2\pi^2 t} \int_0^\infty
K_{2\pi it}\left(2\pi y_1 \sqrt{1+u} \right)  K_{2\pi it}\left(2\pi y_2 \sqrt{1+u^{-1}}
\right) \frac{du}{u}.
\end{equation}
The integrand is negligible if one of the $K$-Bessel functions is in the region of 
uniform rapid decay. If this happens for all $u\in \mathbb{R}_{>0}$ then $W_\nu(a)$ is 
also of rapid decay. This motivates the following

\begin{lem}\label{l:min-max}
	For every $y_1,y_2\in \R_{>0}$, we have
\begin{equation*}\label{min-max}
\min\limits_{u\in \R_{>0}} 
\max\left(y_1 \sqrt{1+u},y_2 \sqrt{1+u^{-1}}\right)
	=
	\sqrt{y_1^2 + y_2^2}.
\end{equation*}
\end{lem}

\begin{proof} The value on the right-hand side is achieved for $u:=y_2^2/y_1^2$.
Conversely, for any $u,y_1,y_2\in \R_{>0}$, we apply the inequality $\max(a,b)\ge 
\sqrt{\frac{a^2+ub^2}{1+u}}$ for $a=y_1\sqrt{1+u}$ and $b=y_2\sqrt{1+u^{-1}}$, to find
$\max(a,b) \ge \sqrt{y_1^2 + y_2^2}.$
\end{proof}

We deduce from Lemma~\ref{l:min-max} that 
\[
\bigcap_{u\in \R_{>0}}
\left\{2\pi \abs{t} \leq 2\pi y_1\sqrt{1+u}\right\}  
\cup
\left\{2\pi \abs{t} \leq 2\pi y_2\sqrt{1+u^{-1}}\right\}
=
\left\{
t^2 \leq y_1^2+y_2^2
\right\}.
\]
This is the ``essential support" for the integral~\eqref{TV-formula}. We thus recover the shadow zone $\mathsf{S}$ for $W_{t\nu_0}(a)$, 
given by $a\notin \operatorname{Im}(\Lambda_{t\nu_0} \to A)$, and
previously found in Proposition~\ref{sd-converse}.

Next, let us choose $y_1,y_2$ such that $y_1^2+y_2^2=t^2$.
It is well-known that the transition range for the Bessel function $K_{it}(y)$ is
$y=t+O(t^{\frac13})$, and that $K_{it}(t) \asymp t^{-\frac13}e^{-\pi t/2}$.
 By the same reasoning as above, the integrand of~\eqref{TV-formula} is negligible unless 
 $u=y_2^2/y_1^2 + O(t^{-\frac{2}{3}})$.
In this range of $u$ the integrand is mildly oscillating and the expected size 
of~$|W_{t\nu_0}(a)|$ is 
\[
y_1y_2 \cdot e^{2\pi^2 t} \cdot t^{-\frac23} e^{-2\pi^2 t} \cdot t^{-\frac23} 
\asymp
t^{\frac23}.
\]
This is consistent with the exponent that arises in Theorem~\ref{n=3 tau}, because $a\in 
\operatorname{Im}(\Lambda_{t\nu_0}^{(1)} - \Lambda^{(0)}_{t\nu_0})$ belongs to the 
component $\mathsf{C}_1$ of the caustic set, and is a type $A_2$ singularity. 

\subsection{Cuspidal points}
We now undertake a more detailed analysis of the Takhtadjan--Vinogradov formula (again, 
in the self-dual case), using it to provide an independent proof of the existence of a 
type-$A_3$ singularity, as in Theorem \ref{prop:CRIT}.

From now on, we specialize to $y_1=y_2=y$ in the interval $(\tfrac12,2)$. Hence
$a=\operatorname{diag}(y^2,y,1)$, and recall $\nu_0=2\pi i(\varpi_1+\varpi_2)$.
\begin{lem}\label{lem:oscillatory-Whittaker}
There are smooth compactly supported
functions $\alpha_1,\alpha_2$, resp. $\beta$, which are constant and non-vanishing in 
$[\tfrac18,2]$, resp.
in $[\tfrac13,3]$, such that 
\[
|W_{t\nu_0}(ta)|
\asymp
t^2
\bigg|\int^\infty_{0}\int^\infty_{0}\int^\infty_{0}
e^{2\pi it\phi_a(x_1,x_2,u)}
\alpha_1(x_1) \alpha_2(x_2)\beta(u)
\frac{dx_1}{x_1}\frac{dx_2}{x_2}\frac{du}{u}\bigg|
 + O_N(t^{-N}),
\]
where the phase 
function is
\[
\phi_a(x_1,x_2,u)=
\log x_1 + \log x_2 + \frac12\log u - \frac{1}{2}\bigg((1+u)x_1-\frac{y^2}{x_1}\bigg)
- \frac{1}{2}
\bigg(x_2-\frac{(1+u) y^2}{u x_2}\bigg).
\]
\end{lem}

\begin{proof}
In view of the exponential decay of 
Bessel functions of large argument, we may truncate the $u$-integral in the 
Takhtajan-Vinogradov formula to a compact subset 
of $(0,\infty)$. We do the truncation smoothly with a function $\beta$ supported inside 
$[\tfrac13,3]$:
\[
|W_{t\nu_0}(ta)| \asymp
(ty)^2 e^{2\pi^2 t} \bigg|\int_0^\infty
K_{2\pi it}\left(2\pi ty \sqrt{1+u} \right)  K_{2\pi it}\left(2\pi ty \sqrt{1+u^{-1}}
\right) \beta(u) \frac{du}{u}\bigg|
+ O(t^{-N}).
\]
Indeed, since $y\in (\tfrac12,2)$, the inequality $y \sqrt{1+u}\le 1$ implies $u< 3$, and 
the inequality $y   
\sqrt{1+u^{-1}}\le 1$ implies $\tfrac13 < u$.

We write the two Bessel functions as follows:
\[
K_{2\pi it}\left(2\pi ty \sqrt{1+u} \right) =
\left(
\frac{\sqrt{1+u}}{y}
\right)^{2\pi i t}
\int_0^\infty e^{2\pi t (i\log x_1 - \frac12 ((1+u)x_1 + \frac{y^2}{x_1}))}
\frac{dx_1}{x_1},
\]

\[
K_{2\pi it}\left(2\pi ty \sqrt{1+u^{-1}}\right) =
\left(
\frac{\sqrt{u}}{y\sqrt{1+u}}
\right)^{2\pi i t}
\int_0^\infty e^{2\pi t (i\log x_2 - \frac12 (x_2 + \frac{(1+u)y^2}{ux_2}))}
\frac{dx_2}{x_2}. 
\]
We move the contour of integration for both $x_1$ and $x_2$ from the interval 
$(0,\infty)$ to the interval $(0,i\infty)$. The critical points of the respective 
integrals are located inside the interval $[\tfrac18,2]$, as can be seen 
from elementary manipulations, see also Lemma~\ref{quadratic-x1c} below.
\end{proof}

We now analyse the critical points of the phase function $\phi_a$.

\begin{lem}\label{quadratic-x1c}
The phase function $\phi_a(x_1,x_2,u)$ has two degenerate critical points of type $A_3$ when $a=a_{\rm cusp}$.
\end{lem}
\begin{proof}
The vanishing of the $x_1$-derivative and the $x_2$-derivative yield the quadratic 
equations
\[
(1+u)x_{1c}^2+y^2=2x_{1c},\quad
ux_{2c}^2 +(1+u)y^2 = 2x_{2c}.
\]
If $u$ is small enough (resp. large enough), then there are two distinct positive 
solutions to the 
first equation (resp. second 
equation). We find
\[
\frac{1}{2}\bigg(x^\pm_{1c}(1+u)-\frac{y^2}{x^\pm_{1c}}\bigg)
=
\mp \sqrt{1 - y^2(1+u)},\quad
\frac{1}{2}\bigg(x^\pm_{2c}-\frac{(1+u)y^2}{ux^\pm_{2c}}\bigg)
=
\mp \sqrt{1 - y^2(1+1/u)}.
\]
Specializing to $u=1$, we find
\[
x^\pm_{1c} = \frac{1 \mp \sqrt{1-2y^2}}{2},\quad
x^{\pm}_{2c}=  1 \mp \sqrt{1-2y^2}.
\]
The $u$-derivative vanishes for $(x^+_{1c},x^+_{2c},1)$, and for 
$(x^-_{1c},x^-_{2c},1)$, 
regardless of the value of 
$y$.

Specializing to $y=1/\sqrt{3}$, we obtain two isolated critical points
\[
\left(\tfrac12(1 - \frac{1}{\sqrt{3}}),1 - \frac{1}{\sqrt{3}},1
\right),\quad
\left(\tfrac12(1 + \frac{1}{\sqrt{3}}),1 + \frac{1}{\sqrt{3}},1
\right).
\]
We compute the Hessian matrices
\[
Q^+= 
\Mtroistrois{6+4 \sqrt{3}}{0}{-1/2}
{0}{\frac32+\sqrt{3}}{\frac12+ \frac14\sqrt{3}}
{-1/2}{\frac12+\frac14\sqrt{3}}{\frac{1}{2\sqrt{3}}},
\quad
Q^-= 
\Mtroistrois{6-4 \sqrt{3}}{0}{-1/2}
{0}{\frac32 -\sqrt{3}}{\frac12 - \frac14\sqrt{3}}
{-1/2}{\frac12 - \frac14\sqrt{3}}{\frac{-1}{2\sqrt{3}}}
\]
We have $\det(Q^+)=\det(Q^-)=0$, and
\[
\operatorname{ker}(Q^+)=
\left(1,
-2\sqrt{3}-4,
8\sqrt{3}+12
\right)^T
,\quad
\operatorname{ker}(Q^-)=
\left(1,2\sqrt{3}-4,12-8\sqrt{3} \right)^T.
\]
Hence $(x^\pm_{1c},x^\pm_{2c},1)$ are two singularities of type $A$.
To complete the proof, we establish that their Milnor numbers are equal to $3$ as follows.
A Gr\"obner basis calculation shows that these are all the critical points of 
$\phi_{a_{\rm cusp}}$.

Let $\mathfrak{m}_\pm$ be the maximal ideal of $\C[x_1,x_2,u]$ generated by 
$(x_1-x^{\pm}_{1c},x_2-x^{\pm}_{2c},u-1)$. 
A computer calculation of the Taylor expansion of $\phi_{a_{\rm cusp}}$ to order three 
shows that:
\begin{equation}\label{determinacy}
\mathfrak{m}^5_\pm \subset \mathfrak{m}^2_\pm (\nabla \phi_{a_{\rm cusp}}) + 
\mathfrak{m}^6_\pm.
\end{equation}
Thus Nakayama's lemma implies that it is sufficient to consider the 
truncated Jacobian rings by the ideal $\mathfrak{m}_\pm^5$.
A computer calculation of the Taylor expansion to order four then gives the dimension of 
the Jacobian ring:
\[
\dim_\C \C[x_1,x_2,u]/ (\mathfrak{m}^5_\pm,\nabla \phi_{a_{\rm cusp}}) =3.\qedhere
\]
\end{proof}
\begin{remark}
The inclusion~\eqref{determinacy} is the hypothesis of the determinacy 
theorem~\cite[Thm.~2.23]{Greuel-Lossen}, which implies the stronger statement that the 
the singularity is $4$-determined, in the sense that its germ is determined up to 
diffeomorphism by the derivatives of order less than four.
\end{remark}

\subsection{Transition region of the Whittaker function}
Having laid the groundwork, we now come to the proof of Theorem~\ref{n=3 uniform 
statement}. We wish to analyse the oscillatory integral
(compared to Lemma~\ref{lem:oscillatory-Whittaker} we have removed a factor of 
$t^{\frac12}$, which will be added back in~\eqref{cusp-waves})
\[
\left(
\frac{t}{2\pi}
\right)^{\frac32}
\int^\infty_{0}\int^\infty_{0}\int^\infty_{0}
e^{2\pi it\phi_a(x_1,x_2,u)}
\alpha_1(x_1)\alpha_2(x_2)\beta(u)
\frac{dx_1}{x_1}\frac{dx_2}{x_2}\frac{du}{u}
\]
of Lemma \ref{lem:oscillatory-Whittaker}. In view of Lemma \ref{quadratic-x1c}, we may apply Lemma \ref{degen F2} to find 
non-zero functions $\mathsf{r}_1^\pm,\mathsf{r}_2^\pm,\beta_j^\pm,\mathsf{s}_\pm\in 
C^\infty_c(A)$, 
with 
$\mathsf{r}_i^\pm(a_{\rm 
cusp})=0$ and $\mathsf{s}_\pm(a_{\rm cusp})=\phi_{a_{\rm 
cusp}}(x_{1c}^\pm,x^\pm_{2c},1)$, 
such that
the above integral is asymptotic to
\[
t^{\frac14} 
\sum_\pm
e^{2\pi it\mathsf{s}_\pm(a)}\sum_{j=0,1,2} 
p_j^\pm(a,t)t^{-\frac{j}{4}}+O(t^{-\frac34}),
\]
where we have put
\[
p_j^\pm(a,t)=e^{i\pi\sigma^\pm/4}\beta_j^\pm(a)
\mathsf{Pe}_j\left(t^{3/4}\mathsf{r}_1^\pm(a),t^{1/2}\mathsf{r}_2^\pm(a)\right),
\]
uniformly for all $a$ in a sufficiently small compact $V$ about $a_{\rm cusp}$, and all 
$t>1$. 

Next, pick small enough constants $c_1,c_2>0$ to be chosen later, and restrict to the 
compact 
subsets
\[
V_\pm(t):=\{a\in V:\; t^{3/4}|\mathsf{r}^\pm_1(a)|\leq c_1\;\text{ and }\; 
t^{1/2}|\mathsf{r}^\pm_2 
(a)|\leq c_2\}.
\]
Write $V(t):=V_+(t) \cap V_-(t)$. On $V(t)$, the first term in the asymptotic dominates, 
so the integral is asymptotic to
\[
t^{\frac14}\sum_\pm e^{2\pi it\mathsf{s}_\pm(a)}p_0^\pm(a,t)+O(1).
\]
Thus, on the 
whole neighborhood $a\in V$ we have that $W_{t\nu_0}(ta)$ is a superposition of 
Pearcey functions, whereas
 in the shrinking neighborhood $a\in V(t)$, the asymptotic of the Whittaker function 
simplifies as a superposition of 
plane waves; indeed we have shown 
\begin{equation}\label{cusp-waves}
|W_{t\nu_0}(ta)| \asymp t^{\frac34} 
\left|
p_0^+(a,t) e^{2\pi i t\mathsf{s}_+(a)}
+ p_0^-(a,t) e^{2\pi i t\mathsf{s}_-(a)}
\right|
+O(t^{\frac12}).
\end{equation}
This is parallel to \S\ref{s:superposition-waves} which discusses the light zone, except 
there we don't need to shrink the neighborhood $V$. 
Note that $ta$ for $a\in V(t)$ contains a ball of width $O(t^{\frac14})$ around the 
critical  point $ta_{\rm cusp}$. By comparison, the width of the transition region for the classical Bessel function is well-known to be $O(t^{\frac13})$.

The argument below will ensure that the two plane waves in~\eqref{cusp-waves} have 
distinct amplitudes, from which we shall deduce the following strengthening of 
Theorem~\ref{n=3 uniform statement}.

\begin{thm}\label{thm:large-values-n=3} There is an absolute constant $C>0$, such that 
for every $t\gg 1$ and $a\in V(t)$
\[
    |W_{t\nu_0}(ta)| \ge C t^{3/4}.
\]
\end{thm}

\begin{proof}
We shall show that for every $t>1$, and 
every $a\in V(t)$,
\begin{equation}\label{end-of-the-proof}
|p_0^+(a,t) e^{2\pi its_+(a)}
+ p_0^-(a,t) e^{2\pi its_-(a)}|\gg 1.
\end{equation}

Indeed, we have for every $t$,
\[
p_\pm(a_{\rm cusp},t)=p_\pm(a_{\rm cusp},0) = e^{i\pi\sigma^\pm/4} 
\beta_0^\pm(a_{\rm cusp})
\mathsf{Pe}(0,0).
\]
Now $\mathsf{Pe}_0\left(t^{3/4}r_1^\pm(a),t^{1/2}r_2^\pm(a)\right)$ and $\beta_0^\pm 
(a)$ 
are non-vanishing at $a_{\rm cusp}$. Thus $p_0^\pm(a,t)\gg 1$ for every $t>1$ and $a\in 
V(t)$, if $c_1,c_2$ are chosen small enough. 

The characteristic polynomial of $Q^\pm$ is $X(X^2 + (- 
\frac{15}{2} \mp \frac{31}{2\sqrt{3}} )X  
+ \frac{365}{16} \pm 13\sqrt{3})$.
 As the 
quadratic factor admits two negative (resp., positive) roots, according to the sign of 
$Q^\pm$, the signature of $Q^+$ (resp., $Q^-$) is $(2,0)$ (resp., $(0,2)$).
Hence $\sigma^\pm=\pm 2$.
From the expression for the 
characteristic polynomial of $Q^\pm$ above, it follows that $|\det\nolimits' 
Q^\pm|= \frac{365}{16}  \pm 13\sqrt{3}$, and a numerical calculation yields

\[
\frac{\beta_0^+(a_{\rm cusp})}{\beta^-_0(a_{\rm cusp})}
=
\frac{\partial_w W(x^-_{1c},x^-_{2c},1)|\det\nolimits' 
Q^-|^{\frac12}x^-_{1c} x^-_{2c}}
{\partial_w W(x^+_{1c},x^+_{2c},1) |\det\nolimits' Q^+|^{\frac12} x^+_{1c} 
x^+_{2c}}
\neq 1.
\]
Thus the amplitudes in~\eqref{cusp-waves} are distinct at $a_{\rm cusp}$, which concludes 
the proof.
\end{proof}

\begin{remark}

The theorem holds in particular for $a=a_{\rm cusp}$. In fact we can also write the 
following asymptotic as $t\to \infty$,
\[
W_{t\nu_0}(ta_{\rm cusp}) \asymp t^{\frac34}
\left(
\beta_0^+(a_{\rm cusp})
e^{2\pi it\mathsf{s}_+(a_{\rm cusp})}
-
\beta_0^-(a_{\rm cusp})
e^{2\pi it\mathsf{s}_-(a_{\rm cusp})}
\right)
+O(t^{\frac12}),
\]
where the amplitudes $\beta^\pm_0(a_{\rm cusp})$ are non-zero by~\eqref{def-Pearcey}, and the 
critical values of the phase 
function are
\[
\mathsf{s}^\pm(a_{\rm cusp})=\phi_{a_{\rm 
cusp}}(x^{\pm}_{1c},x^{\pm}_{2c},1)=\frac{\pm 
2}{\sqrt{3}}+\log(\frac{2}{3} \mp \frac{1}{\sqrt{3}}).
\]
In particular, we evaluate $\mathsf{s}^+(a_{\rm cusp})-\mathsf{s}^-(a_{\rm cusp})$ to be $\frac{4}{\sqrt{3}}-2\log(2+\sqrt{3})$, which is in agreement with Remark \ref{rem-S-cusp}.
\end{remark}

\nocite{Var}

\bibliographystyle{abbrv}

\bibliography{BT,Ball}

\end{document}